%% file: main.tex
\documentclass[12pt, reqno]{amsart}

\usepackage{graphicx,graphics}
\usepackage{amsmath, amssymb, amscd, amsthm, comment, amsxtra}
\usepackage{amsfonts}
\usepackage{latexsym}
\usepackage[abs]{overpic}

\usepackage[pdftex, linktocpage=true]{hyperref}
\newtheorem{thm}{Theorem}[section]
\newtheorem{cor}[thm]{Corollary}

\newtheorem{lem}[thm]{Lemma}

\newtheorem{prop}[thm]{Proposition}

\newenvironment{pf*}[1]{\proof[#1]}{\endproof}
\usepackage{euscript}

\usepackage[OT2,OT1]{fontenc}

\newcommand{\ignore}[1]{{}}

\theoremstyle{definition}
\newtheorem{defn}{Definition}[section]

\theoremstyle{remark}

\newcommand{\tl}{\tilde}

\newcommand{\eps}{\epsilon}

\renewcommand{\Re}{\operatorname{Re}}
\renewcommand{\Im}{\operatorname{Im}}

\numberwithin{equation}{section}
\newcommand{\thmref}[1]{Theorem~\ref{#1}}
\newcommand{\propref}[1]{Proposition~\ref{#1}}
\newcommand{\secref}[1]{\S\ref{#1}}
\newcommand{\lemref}[1]{Lemma~\ref{#1}}

\newcommand{\C}[1]{{\Bbb C_{#1}}}
\newcommand{\I}{P}

\newcommand{\cA}{{\cal A}}
\newcommand{\cU}{{\cal U}}
\newcommand{\cW}{{\cal W}}
\newcommand{\cM}{{\cal M}}
\newcommand{\cV}{{\cal V}}
\newcommand{\cF}{{\cal F}}

\newcommand{\cB}{{\aaa B}}
\newcommand{\cT}{{\cal T}}
\newcommand{\cI}{{\cal I}}
\newcommand{\cN}{{\cal N}}
\newcommand{\cP}{{\cal P}}
\newcommand{\cC}{{\cal C}}
\newcommand{\cH}{{\cal H}}
\newcommand{\cR}{{\cal R}}
\newcommand{\cL}{{\cal L}}
\newcommand{\cD}{{\cal D}}
\newcommand{\cE}{{\cal E}}
\newcommand{\cS}{{\cal S}}

\newcommand{\CC}{{\Bbb C}}
\newcommand{\RR}{{\Bbb R}}

\newcommand{\NN}{{\Bbb N}}
\newcommand{\DD}{{\Bbb D}}

\input{macros.tex}

\begin{document}
\addtolength{\evensidemargin}{-0.7in}
\addtolength{\oddsidemargin}{-0.7in}






\setcounter{tocdepth}{1}

\title[Siegel disk universality and renormalization]{Golden mean Siegel disk universality and renormalization}

\author{Denis Gaidashev}
\address{Denis Gaidashev, Uppsala University, Uppsala, Sweden}
\email{gaidash@math.uu.se}

\author{Michael Yampolsky}
\address{Michael Yampolsky, University of Toronto, Toronto, Canada}
\email{yampol@math.utoronto.edu}

\date{\today}

\maketitle

\begin{abstract}
We provide a computer-assisted proof of one of  the central open questions in one-dimensional renormalization theory -- universality of the golden-mean Siegel disks. We further show that for every function in the stable manifold of the golden-mean renormalization fixed point the boundary of the Siegel disk is a quasicircle which coincides with the closure of the critical orbit, and that the dynamics on the boundary of the Siegel disk is rigid. 

Furthermore, we extend the renormalization from one-dimensional analytic maps with a golden-mean Siegel disk 
to two-dimensional  dissipative H\'enon-like maps and show that the renormalization hyperbolicity result still 
holds in this setting. 
\end{abstract}

\tableofcontents

\input{intro}
\input{symmetric}
\input{nonsymmetric}
\input{circle}

\input{holomorphic}

\input{dissipative}

\input{comp_proof}

\input{references}
\end{document}

%% file: macros.tex
\theoremstyle{plain}

\newtheorem*{cmp}{Contraction Mapping Principle}
\newtheorem{definition}{Definition}

\newtheorem{lemma}{Lemma}

\newcommand{\field}[1]{\mathbb{#1}} 
\newcommand{\diam}[1]{{\rm diam} { \ #1 }}

\newcommand{\dist}[1]{{\rm dist} { #1 }}

\def\id{{\rm id}}

\def\I{\mathbb{I}}
\def\C{\mathbb{C}}
\def\D{\mathbb{D}}
\def\R{\mathbb{R}}

\def\eps{\epsilon}

\def\la{\lambda}
\def\las{\lambda_*}

\def\span{{\rm span}}

\def\arg{\mathrm{arg}}

\def\fC{\field{C}}
\def\fD{\field{D}}

\def\fR{\field{R}}

\def\bF{\mathbf{F}}

\def\cA{\mathcal{A}}
\def\cB{\mathcal{B}}
\def\cC{\mathcal{C}}
\def\cD{\mathcal{D}}
\def\cE{\mathcal{E}}
\def\cF{\mathcal{F}}

\def\cH{\mathcal{H}}

\def\cI{\mathcal{I}}

\def\cL{\mathcal{L}}
\def\cN{\mathcal{N}}
\def\cM{\mathcal{M}}
\def\cT{\mathcal{T}}
\def\cR{\mathcal{R}}
\def\cRG{\mathcal{RG}}
\def\cP{\mathcal{P}}
\def\cS{\mathcal{S}}

\def\cU{\mathcal{U}}
\def\cW{\mathcal{W}}
\def\cV{\mathcal{V}}

\def\cZ{\mathcal{Z}}
\def\cO{\mathcal{O}}

\def\tV{\tilde{V}}
\def\tU{\tilde{U}}
\def\tr{\tilde{r}}
\def\tc{\tilde{c}}

\newcommand{\vareps}{\varepsilon}

%% file: intro.tex
\section{Introduction}

One of the central examples of universality on one-dimensional dynamics is
provided by Siegel disks of quadratic polynomials. Let us consider, for instance,
the mapping 
$$P_\theta(z)=z^2+e^{2\pi i\theta}z,\text{ where }\theta=(\sqrt{5}-1)/2$$
is the inverse golden mean. By a classical result of Siegel \cite{Sieg}, the dynamics of $P_\theta$
is linearizable near the origin. The Siegel disk of $P_\theta$, which we will further
denote $\Delta_\theta$ is the maximal neighborhood of zero in which a conformal
change of coordinates reduces $P_\theta$ to the form $w\mapsto e^{2\pi i\theta}w$.
By the results of Douady, Ghys, Herman, and Shishikura \cite{Do}, the topological disk $\Delta_\theta$
is bounded by  a Jordan curve which contains the only critical point of $P_\theta$ and thus coincides with the
postcritical set of $P_\theta$.

It has  been observed numerically (cf. the work of Manton and Nauenberg \cite{MN}), 
that the boundary of $\Delta_\theta$ is asymptotically
self-similar near the critical point. Moreover, the scaling factor is universal
in a large class of analytic mappings with a golden-mean Siegel disk.
In 1983 Widom \cite{Wi} defined a renormalization procedure for $P_\theta$ which
``blows up'' a part of the invariant curve $\partial \Delta_\theta$ near the critical
point, and conjectured that the renormalizations of $P_\theta$ converge to a fixed point.
In addition, he conjectured that in a suitable functional space this fixed point
is hyperbolic with one-dimensional unstable direction.

In 1986
MacKay and Persival \cite{MP} extended the conjecture to other rotation numbers,
postulating the 
existence of a hyperbolic renormalization horseshoe corresponding to Siegel disks of 
analytic maps,
analogous to Lanford's horseshoe for critical circle maps \cite{La1,La2}. 

In 1994 Stirnemann \cite{Stir} gave a computer-assisted proof of the existence of a
renormalization fixed point with a golden-mean Siegel disk, see also \cite{Bur}.
In 1998, McMullen \cite{McM} proved the asymptotic self-similarity of  golden-mean
Siegel disks in the quadratic family. He constructed a version of renormalization 
based on holomorphic commuting pairs of de~Faria \cite{dF1,dF2}, and
showed that the renormalizations of a quadratic polynomial with a golden Siegel disk
 near the critical point converge
to a fixed point geometrically fast. 
More generally, he constructed a renormalization horseshoe for bounded
type rotation numbers, and used renormalization to show that the Hausdorff dimension of
the corresponding quadratic Julia sets is strictly less than two.

A major break-through in the renormalization theory for Siegel disks is due to the concept of {\it cylinder renormalization} which was first introduced by the second author in the context of critical circle maps \cite{Ya1}. For rotation numbers of {\it high type} (that is, for those, whose continued fraction expansion has a large lower bound on the coefficients) cylinder renormalization can also be described as ``almost parabolic renormalization'', based on the nonlinear perturbed Fatou coordinate change. In a seminal paper with far-reaching consequence \cite{IS} H.~Inou and M.~Shishikura proved a version of {\it a priori} bounds for Siegel renormalization of high type, and proved a renormalization convergence result. In 
\cite{Ya3} the second author used the results of \cite{IS} to construct a cylinder renormalization horseshoe for rotation numbers of high type. However, these results do not cover the golden-mean case, which is, in a sense, the extreme opposite of the almost parabolic case.

The issue of hyperbolicity of cylinder renormalization for golden-mean Siegel disks  has been studied in our earlier paper \cite{GaYa}. In that work we provided a numerical evidence for the existence of a hyperbolic fixed point of cylinder renormalization with a single unstable direction, and computed the bounds on its spectrum. The principal difficulty in this numerical study was the computation of the nonlinear change of coordinates in the definition of cylinder renormalization. It was implemented using the constructive Beltrami equation solver developed by the first author in \cite{Ga}. This difficulty has also precluded us from turning the numerical study of \cite{GaYa} into a rigorous computer-assisted proof.

Having thus attracted much attention, the renormalization hyperbolicity part of the conjecture of Widom for golden-mean
Siegel disks has been  open until now. In this paper we finally settle it. To do this, we return to the setting of {\it almost 
commuting holomorphic pairs}, used by Stirnemann \cite{Stir} in his proof of the existence of a golden-mean fixed point. 
Surprisingly, this approach, which was essentially abandoned in the field since the late 1990's has finally bore fruit.

\begin{figure}
\centering
\vspace{2mm}       
\begin{tabular}{c c} 
\includegraphics[height=5cm]{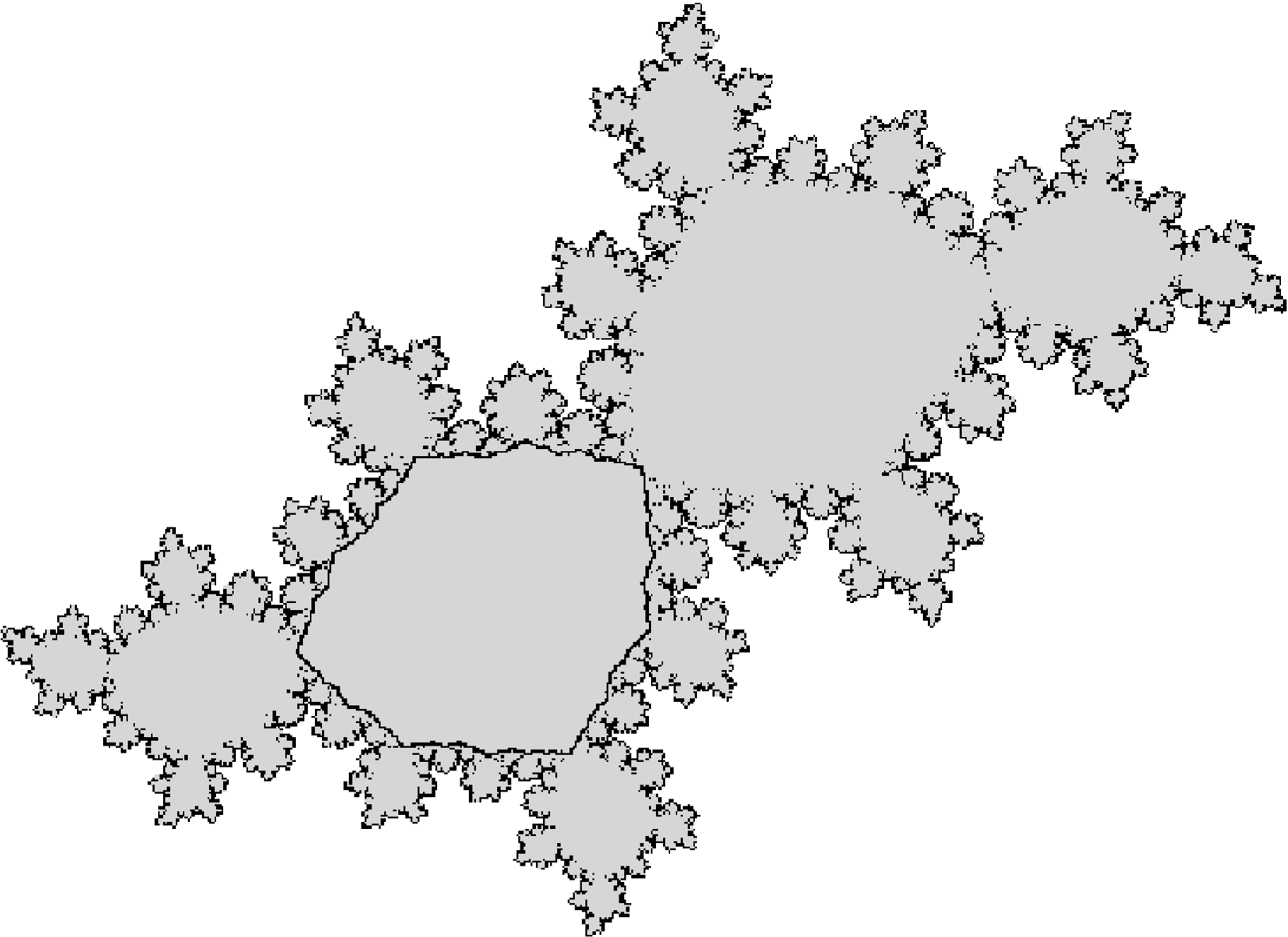} \quad & \quad \includegraphics[width=5cm,height=5cm]{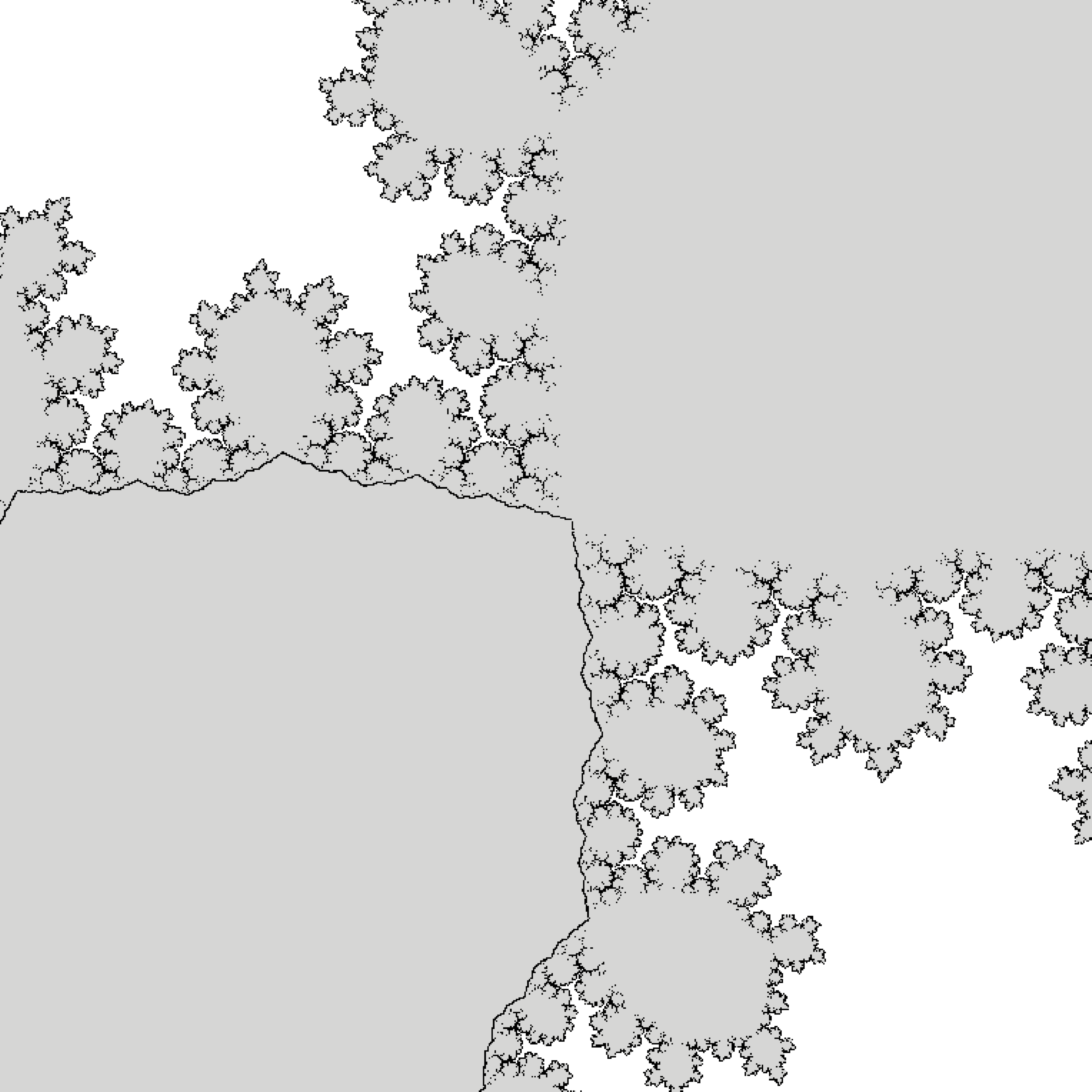}
\end{tabular}
\vspace{4mm}       
\caption{The filled Julia set of the golden mean quadratic polynomial $e^{2 i \pi \theta} z+ z^2$ (grey), together with a blow up around the critical point. The Siegel disk is bounded by the black curve.}
\label{fig:1}       
\end{figure}

Having shown the renormalization hyperbolicity, we use an approach of \cite{Ya3} to prove that for every map in the stable manifold of the golden-mean renormalization fixed point the Siegel disk is a quasidisk, whose boundary coincides with the closure of the critical orbit. Combining this with results of McMullen's  \cite{McM}, we show that the natural quasiconformal conjugacy between any such map and $P_\theta$ (defined in a neighborhood of the Siegel disk) is, in fact, $C^{1+\alpha}$ conformal on the boundary of the Siegel disk. This implies, in particular, that $P_\theta$ lies in the stable manifold of the fixed point, and thus our renormalization fixed point coincides with the one constructed in \cite{McM}.

Our proof of hyperbolicity of renormalization has an important bonus. We are able to extend the definition of Siegel renormalization to 
dissipative H\'enon-like maps, similarly to the way it was done for Feigenbaum renormalization by A.~de~Carvalho, M.~Lyubich, and M.~Martens in \cite{CLM}. 
In this larger function space we show that the golden-mean fixed point
remains hyperbolic with a single unstable direction (in fact, no new non-trivial eigendirections are added to its spectrum by passing to a larger space). In a forthcoming joint paper with R.~Radu \cite{GaRYa} we use this result to settle an important open problem in two-dimensional dynamics -- we show that the Siegel disk of a highly dissipative H\'enon map with a golden-mean semi-Siegel fixed point is bounded by a topological circle, dynamics on which is quasisymmetrically conjugate to a rigid rotation.


\subsection*{Structure of the paper}
The following brief guide to the layout of the paper will be useful to the reader. The next section (\secref{sec:symmetric}) contains the principal definitions of the function spaces of almost commuting pairs and the renormalization transformation acting upon them. In \secref{sec:ACSM} we specialize to considering the subspace of symmetric almost commuting pairs. The somewhat restrictive quadratic symmetry condition is, however, broad enough to include renormalizations of Siegel quadratic polynomials. Further in this section we formulate \thmref{mainthm1} on hyperbolicity of renormalization of symmetric almost commuting pairs, which is our first principal result. The proof of \thmref{mainthm1} occupies a substantial portion of the paper and includes rigorous computer-assisted estimates; we postpone it until  \secref{sec:pfmain1}. 

In \secref{sec:ACM} we generalize the hyperbolicity result to pairs without the symmetry condition. A key ingredient here is a non-linearity norm which is discussed in the opening pages of the section. The hyperbolicity of renormalization result in the space of general almost commuting pairs is proved (in two different Banach metrics) in \thmref{mainthm2} and \thmref{mainthm3}. 

The subject of \secref{sec:quasiarc} is \thmref{mainthm4}, which establishes that the postcritical set of an almost commuting pair in the stable manifold of the renormalization fixed point $\zeta_*$ is an invariant quasi-arc. This result, together with McMullen's theory of holomorphic pairs is used in \secref{sec:holomorphic} to demonstrate that the golden-mean Siegel quadratic polynomial lies in $W^s(\zeta_*)$. 

In \secref{sec:RenACM} we extend our results to dissipative maps in two variables, and prove our final renormalization hyperbolicity result for these maps (\thmref{mainthm6}).

Finally, as already mentioned above, \secref{sec:pfmain1} contains the computer-assisted proof of \thmref{mainthm1}.

%% file: symmetric.tex
\section{Almost commuting pairs and definition of renormalization}
\label{sec:symmetric}

Renormalization of maps with Siegel disks is commonly defined in the language of commuting pairs of analytic maps (cf. \cite{McM,Ya3}). There is, however, no natural structure of a Banach manifold on pairs of commuting holomorphic maps. This has led the second author to introduce the notion of {\it cylinder renormalization operator} $\cR_{\text{cyl}}$ for Siegel maps in \cite{Ya3}, and state renormalization hyperbolicity results for $\cR_{\text{cyl}}$. In this paper we use a different approach, going back to the work of Stirnemann \cite{Stir}, and consider {\it almost commuting pairs}. 

For a topological disk $Z\ni 0$ denote $\cH(Z)$ the Banach space of holomorphic functions $f$ in $Z$ with the uniform norm:
\begin{equation}
\label{eq:unorm1}\| f\|=\sup_{z \in Z} |f|,
\end{equation}
and set $\cH(Z,W)\equiv \cH(Z)\times\cH(W)$.
We will typically use the notation $(\eta,\xi)$ for an element of $\cH(Z,W)$.

We let $\cC(Z,W)$ denote the Banach submanifold of $\cH(Z,W)$ given by the linear conditions 
$$\eta'(0)=\xi'(0)=0.$$
We say that a pair $(\eta,\xi)\in\cC(Z,W)$ is  {\it almost commuting to order $k\geq 0$} if the following holds:  
\begin{equation}\label{eq:accond}
(\eta \circ \xi)^{(n)}(0)=(\xi \circ \eta)^{(n)}(0)), \ 0\leq n\leq k;\; \eta''(0)>0;\;\xi''(0)>0,\;\text{ and }\xi(0)=1.
\end{equation}
The following is immediate:
\begin{prop}
Let $\zeta=(\eta,\xi)\in\cC(Z,W)$ be a pair such that $\xi(0)=1$. Then $\zeta$ is almost commuting to order $k$ if and only if the commutator 
\begin{equation}\label{eq:commutator}
[\eta,\xi]\equiv \eta\circ \xi(z)-\xi\circ\eta(z)=o(z^k).
\end{equation}
\end{prop}

In the case $k=2$, we will simply call the pair {\it almost commuting (or a.c.)}; almost commuting pairs will be the principal subject of this paper.
We denote $\cB(Z,W)$ the subset of $\cC(Z,W)$ consisting of a.c. pairs.  

Let $c(z)$ denote the operation of complex conjugation: $c(z)=\bar{z}$. 
Let $(\phi,\psi)\in\cE(U,V)$ and, as usual, set $\eta=\phi\circ q_2$, $\xi=\psi\circ q_2$.

\begin{definition}\label{def:renormalizability} 
We will say that the  pair $(\eta,\xi)\in\cH(Z,W)$ is {\it renormalizable}, if 
\begin{eqnarray}
\label{eq:inclusion1} \lambda(c(W)) &\Subset& Z, \\
\label{eq:inclusion2} \lambda(c(Z)) &\Subset& W, \\
\label{eq:inclusion3} \xi(\lambda(c(Z))) &\Subset& Z,
\end{eqnarray}
where $\lambda(z)=\xi(0) \cdot z$. The renormalization of a pair $(\eta,\xi)$ will be defined as 
\begin{equation}\label{eq:renpairs}
\cR (\eta,\xi)=\left(c \circ \lambda^{-1} \circ \eta \circ \xi \circ \lambda \circ c  ,  c \circ \lambda^{-1} \circ \eta \circ \lambda \circ c \right).
\end{equation}
\end{definition}

We note:
\begin{prop}
\label{prop-m-invt}
For every $k\geq 2$, 
renormalization preserves the set of pairs which almost commute to order $k$.
\end{prop}
\begin{proof}
\ignore{
 Indeed,  using the relation $\partial_z c(\eta(\bar{z}))=c(\partial_{\bar{z}}\eta(\bar{z}))=c(\eta'(\bar{z}))$, the set of conditions $(\ref{eq:accond})$ for the renormalized pair $(\tilde{\eta},\tilde{\xi})$ becomes
\begin{eqnarray}
\nonumber (\tilde{\eta} \circ \tilde{\xi})(0)&-&(\tilde{\xi} \circ \tilde{\eta})(0)=\\
\nonumber &=& \left(c \circ \lambda^{-1} \circ \eta \circ \xi  \circ \eta \circ \lambda \circ c\right)(0)-\left(c \circ \lambda^{-1} \circ \eta  \circ \eta \circ \xi \circ \lambda \circ c \right)(0)\\
\nonumber &=& \left(c \circ \lambda^{-1} \circ \eta \circ \xi  \circ \eta \right)(0)-\left(c \circ \lambda^{-1} \circ \eta  \circ \eta \circ \xi \right)(0)\\
\label{eq:rcommute0} & =& 0 \quad \underline{{\rm by} \ \eta(\xi(0))=\xi(\eta(0))}, \\
\nonumber  (\tilde{\eta} \circ \tilde{\xi})'(0)&-&(\tilde{\xi} \circ \tilde{\eta})'(0)=\\
\nonumber &=& \left(c \circ \lambda^{-1} \circ \eta \circ \xi \circ \eta \circ  \lambda \circ c \right)'(0)- \left(c \circ \lambda^{-1} \circ \eta \circ \eta \circ \xi \circ \lambda \circ c \right)'(0)\\
 \nonumber &=& c \left( \eta' \circ \xi(\eta(0)) \cdot (\xi \circ \eta)'(0) \right) - c \left( \eta' \circ \eta(\xi(0)) \cdot (\eta \circ \xi)'(0) \right) \\
 \label{eq:rcommute1}  &=&0  \quad \underline{{\rm by} \ \eta(\xi(0))=\xi(\eta(0)) \ {\rm and} \ (\eta \circ \xi)'(0)=(\xi \circ \eta)'(0)},\\
\nonumber  (\tilde{\eta} \circ \tilde{\xi})''(0)&-&(\tilde{\xi} \circ \tilde{\eta})''(0)=\\
\nonumber &=& \left(c \circ \lambda^{-1} \circ \eta \circ \xi \circ \eta \circ  \lambda \circ c \right)''(0)- \left(c \circ \lambda^{-1} \circ \eta \circ \eta \circ \xi \circ \lambda \circ c \right)''(0)\\
 \nonumber &=& c \left( \eta' \circ \xi(\eta(0)) \cdot (\xi \circ \eta)''(0) \right) - c \left( \eta' \circ \eta(\xi(0)) \cdot (\eta \circ \xi)''(0) \right) \\
 \nonumber  &\phantom{=}& \phantom{c \left( \eta' \circ \xi(\eta(0)) \cdot (\xi \circ \eta)''(0) \right)} + c \left( \eta'' \circ \xi(\eta(0)) \cdot ((\xi \circ \eta)'(0))^2 \right) \\
\nonumber  &\phantom{=}& \phantom{c \left( \eta' \circ \xi(\eta(0)) \cdot (\xi \circ \eta)''(0) \right)} - c \left( \eta'' \circ \eta(\xi(0)) \cdot ((\eta \circ \xi)'(0))^2 \right) \\
\label{eq:rcommute2}  &=&0  \quad \underline{{\rm by} \ (\eta \circ \xi)^{(n)}(0)=(\xi \circ \eta)^{(n)}(0), \ n=0,1,2.}
\end{eqnarray}
}
Suppose  (\ref{eq:commutator}) holds for $\zeta=(\eta,\xi)$. Consider the commutator of $\eta\circ \xi$ and $\eta$:
$$\eta\circ\xi\circ\eta(z)-\eta^2(z)\circ\xi(z)=\eta(\xi\circ\eta(z))-\eta(\eta\circ\xi(z)).$$
Denoting $x=\xi\circ\eta(z)$ and $y=\eta\circ\xi(z),$ we have
$$\eta(x)-\eta(y)=\eta'(x)(x-y)+o(x-y).$$
Thus, 
$$\eta\circ\xi\circ\eta(z)-\eta^2(z)\circ\xi(z)=o(z^k),$$
and hence $\cR\zeta$ is almost commuting to order $k$.
\end{proof}

\section{Renormalization of almost commuting symmetric pairs}\label{sec:ACSM}
In this section we work with the same definitions as Stirnemann~\cite{Stir}, within a class of almost commuting pairs $\cM$ with a quadratic symmetry condition. The class $\cM$ is too narrow for generalizations, and we will expand it in the following section -- however, the results we prove for $\cM$ will form a basis of  further discussion.


Let $U$ and $V$ be two simply-connected domains containing the origin. Set $$q_2(z)=z^2$$ and let 
$$Z\equiv q^{-1}_2(U),\;W\equiv q^{-1}_2(V).$$
We denote $\cE(U,V)$ the subset of $\cH(Z,W)$ consisting of all pairs which factor as 
\begin{equation}\label{eq:factors}
\eta=\phi\circ q_2,\;\xi=\psi\circ q_2,
\end{equation}
where $\phi$ and $\psi$ are holomorphic functions on $U$ and $V$ respectively. We will identify $\cE(U,V)$ with the Banach space $\cH(Z,W)$, and will use the notation $(\eta,\xi)$ and $(\phi,\psi)$ for a pair in $\cE(U,V)$interchangeably, whenever it does not cause any ambiguity.  We will refer to $(\phi,\psi)$ as {\it factors} of $(\eta,\xi)$.

We say that $(\eta,\xi)\in\cE(U,V)$ is an {\it almost commuting symmetric (a.c.s.) pair} if $(\eta,\xi)\in\cB(Z,W)$.
In this case, we will call the maps $\phi$ and $\psi$ {\it a.c.s. factors}. 
We will denote $\cM(U,V)\subset \cE(U,V)$ the set of a.c.s. factors (again equipped with the uniform norm coming from $\cH(U,V)$).

We note:

\begin{prop}
\label{prop-submfld1} Let $\cW\subset \cE(U,V)$  be an open set in which $\psi'(0)\neq 0$ (for instance, small open neighborhoods
 of a pair of factors in which $\psi$ is locally univalent at the origin).
Then the  space  $\cM(U,V)\cap \cW$ is a Banach submanifold of $\cE(U,V)$.
\end{prop}
\begin{proof}

We would like to write out the commutation conditions $(\ref{eq:accond})$ at $0$ up to order $2$. First, we remark that 
\begin{eqnarray}
\label{eq:1der} \eta'(0)&=&\phi'(0) q_2'(0)=0,  \hspace{43mm}  \xi'(0)=\psi'(0)  q_2'(0)=0,\\
\label{eq:2der} \eta''(0)&=&\phi''(0) (q'_2(0))^2 +\phi'(0) q_2''(0)=2 \phi'(0), \quad  \xi''(0)=2 \psi'(0),
\end{eqnarray}
and that 
\begin{eqnarray}
\label{eq:2etaxider} (\eta \circ \xi)''(0)&=&\eta' \circ \xi(0) \cdot \xi''(0)=4 \phi' \circ q_2 \circ \psi(0) \cdot \psi (0) \cdot \psi'(0), \\
\label{eq:2xietader} (\xi \circ \eta)''(0)&=&4 \psi' \circ q_2 \circ \phi(0) \cdot \phi (0) \cdot \phi'(0),
\end{eqnarray}

{ The first of the equations $(\ref{eq:accond})$, the equation  for $n=0$}, is simply
\begin{equation}\label{eq:accond0}
\eta \circ \xi(0) -\xi \circ \eta(0)=0 \Rightarrow \phi(\psi(0)^2)-\psi(\phi(0)^2)=0.
\end{equation}

{ The second of the equations $(\ref{eq:accond})$, the equation  for $n=1$}, becomes
\begin{equation}\label{eq:accond1}
\eta' \circ \xi(0) \cdot \xi'(0) -\xi' \circ \eta(0) \cdot \eta'(0)=0,
\end{equation}
which is automatically satisfied  because of the condition $(\ref{eq:1der})$.

{ The third of the equations $(\ref{eq:accond})$, the equation  for $n=2$}, becomes
\begin{equation}\label{eq:accond2}
 (\eta \circ \xi)''(0) -(\xi \circ \eta)''(0)=0 \Rightarrow \phi'(\psi(0)^2)  \psi (0)  \psi'(0)-\psi'(\phi(0)^2)  \phi (0)  \phi'(0)=0.
\end{equation}

Let us write $\phi$ and $\psi$ as power series centered at $0$:
$$\phi(z)=\sum_{i=0}^\infty \hat\phi_i z^i\text{, and } \psi(z)=\sum_{i=0}^\infty \hat\psi_i z^i.$$ 

Set
$q(z)=\phi(z)-\hat \phi_2 z^2 -\hat \phi_3 z^3$,    $s(z)=\psi(z)-\hat\psi_0$. We obviously have a parametrization of
$\cE(U,V)$ by the analytic coordinates $$(\hat \phi_2,\hat \phi_3, \hat\psi_0,q(z),s(z)).$$ 
We will show that the set of a.c.s. pairs is parametrized by the coordinates
other than $\hat\phi_2$, $\hat\phi_3$, $\hat\psi_0$. 
To that end consider the set of conditions $(\ref{eq:accond0})$, $(\ref{eq:accond2})$, together with $\psi(0)=1$:,  which define
$\cM(U,V)$ as a subset of $\cH(U,V)$ and write them in coordinates:
\begin{equation} \label{classF}
F(\hat\phi_2,\hat\phi_3,\hat\psi_0; q,s)=0,
\end{equation}
where $F: \cH(U,V) \to  \fR^3$ is given by 
\begin{eqnarray}
\nonumber F_1(\hat\phi_2,\hat\phi_3, \hat\psi_0; q, s)&=&\phi(\psi(0)^2) - \psi(\phi(0)^2) \\
\nonumber     &=&q((s(0)+\hat\psi_0)^2) +\hat\phi_2 (s(0)+\hat\psi_0)^{4}+ \hat\phi_3 (s(0)+\hat\psi_0)^{6}- s(q(0)^2)-\hat\psi_0 \\
\nonumber   F_2(\hat\phi_2,\hat\phi_3, \hat\psi_0; q,s)&=&\phi'(\psi(0)^2) \psi(0) \psi'(0) - \psi'(\phi(0)^2) \phi(0) \phi'(0) \\
\nonumber     &=&\left(q'(\psi(0)^2)+2 \hat\phi_2 (s(0)+\hat\psi_0)^{2}+3 \hat\phi_3 (s(0)+\hat\psi_0)^{4} \right) \psi(0) \psi'(0) - \\
\nonumber  &\phantom{=}&  - s'(q(0)^2) q(0) q'(0) \\
\nonumber      F_3(\hat\phi_2,\hat\phi_3, \hat\psi_0; q, s)&=&\psi(0)-1=s(0)+\hat\psi_0-1.
\end{eqnarray}
Consider the differential of the above equations with respect to the first three coordinates:
\begin{equation}\label{eq:determinant}
\left| \! \!
\begin{array}{c c c}
\psi(0)^{4}   &  \psi(0)^{6} &  2 \phi'(\psi(0)^2) \psi(0)-1   \\
 2 \psi(0)^{3}  \psi'(0)  &   3 \psi(0)^{5} \psi'(0)  &   2 \phi''(\psi(0)^2) \psi(0)^{2} \psi'(0)+ \phi'(\psi(0)^2) \psi'(0) \\
 0  &   0  &  1
\end{array}
\! \! \right|=\psi(0)^{9} \psi'(0).
\end{equation}
Since $\psi(0)\approx 1$ on an open neighborhood of $\cM(U,V)$, the determinant is close to $\psi'(0)$, and hence is non-zero.
By the Regular Value Theorem (see e.g. \cite{ArBo}), the set $F^{-1}(0)$ is an analytic submanifold, parametrized by $q$ and $s$. 

\end{proof}

\subsection{Renormalization for a.c.s. pairs}\label{sec:RenACFM}

If a pair $(\eta,\xi)\in\cE(U,V)$ is renormalizable, we can write the renormalization in terms of its factors as follows:
\begin{equation}
\label{eq:renfactors}\cR(\phi,\psi)=\left(c \circ \lambda^{-1} \circ \phi \circ q_2 \circ \psi \circ \lambda^2 \circ c  ,  c \circ \lambda^{-1} \circ \phi \circ \lambda^2 \circ c \right).
\end{equation}
We will use one and the same symbol $\cR$ for both the renormalization operator $(\ref{eq:renpairs})$  acting on pairs $(\eta,\xi)$ and the operator $(\ref{eq:renfactors})$ acting on the factors; the domain of the operator $\cR$ will be always clear from the context.

According to Proposition \ref{prop-m-invt}, the Banach  submanifold $\cM(U,V)$ is invariant under the renormalization operator. Our hyperbolicity result
holds for the restriction of $\cR$ to this submanifold: otherwise, the codimension of the strong stable manifold of the operator $\cR$ will be greater than one (and hence, the universality statement will not hold). However, for the purposes of a computer-assisted proof, the restriction
$\cR|_{\cM(U,V)}$ is not a good object to work with, since we do not have an explicit analytic parametrization of $\cM(U,V)$.
Therefore, we will borrow a page from the work of Stirnemann~\cite{Stir}, and define an extension of $\cR$ from $\cM(U,V)$ which maps
$$\cE(U,V)\cap \cU\to\cM(U,V), \text{ where }\cU \text{ is an open neighborhood of }\cM(U,V).$$
Namely, we prove the following:
\begin{prop}
\label{prop-proj}
There exists an open neighborhood $\cU$ of $\cM(U,V)$ in $\cE(U,V)$ such that the following holds.
Suppose that a pair of maps $(\eta,\xi) \in \cU$. Then there exists a unique triple of complex numbers $a$, $b$, $c$ such that 
\begin{eqnarray}
\label{Pproj}
\cP (\eta,\xi)(z)\equiv\left( \eta(z)+a z^4+b z^6, \xi(z)+c \right)\in\cM(U,V).
\end{eqnarray}
Moreover, the dependence of the numbers $a$, $b$, and $c$ on $(\eta,\xi)\in\cU$ is analytic.
\end{prop}
\begin{proof}
Set 
 $\hat{\phi}(z)=\phi(z)+a z^2+b z^3$ and $\hat{\psi}(z)=\psi(z)+c$.
We would like $(\hat{\phi},\hat{\psi})$ to satisfy the set of conditions  $(\ref{eq:accond0})$, $(\ref{eq:accond2})$ together with $\hat{\psi}(0)=1$. We immediately have
\begin{equation}\label{eq:newFs1}
c=1-\psi(0).
\end{equation}
Furthermore, the linearization of equations  $(\ref{eq:accond0})$ and $(\ref{eq:accond2})$ for $a$ and $b$ results in
\begin{equation} \label{eq:newFs2}
\left[\!
\begin{array}{c c}
\hat{\psi}(0)^{4} \  & \ \hat{\psi}(0)^{6}   \\
 2 \hat{\psi}(0)^{3}  \hat{\psi}'(0)   &   3 \hat{\psi}(0)^{5} \hat{\psi}'(0)
\end{array}
\!\right]\! \cdot \!
\left[\!
\begin{array}{c}
a \\
b
\end{array}
\!\right]\!=\!
\left[\!
\begin{array}{c}
\hat{\psi}(\phi(0)^2)-\phi(\hat{\psi}(0)^2) \\
\hat{\psi}'(\phi(0)^2) \phi(0) \phi'(0)-\phi'(\hat{\psi}(0)^2) \hat{\psi}(0) \hat{\psi}'(0)
\end{array}
\!\right]
\end{equation}
The determinant of this linear system  is equal to
\begin{equation}\label{eq:newJac}
(\psi(0)+c)^{9} \psi'(0)=(1+c)^{9} \psi'(0).
\end{equation}
Therefore, if $(\phi,\psi)$ is small perturbation of a.c.s. factors $(\phi,\psi)$ of an a.c.s. pair $(\eta,\xi)$, then the determinant $(\ref{eq:newJac})$ is nonzero, and the set of equations $(\ref{eq:accond0})$, $(\ref{eq:accond2})$ and $(\ref{eq:newFs1})$ for the pair $\cP (\eta,\xi)$ has a unique solution $(a,b,c)$ which depends on $(\eta,\xi)$ analytically.

\end{proof}

\begin{defn}
We set $$\cRG\equiv \cP \circ \cR.$$
\end{defn}
By Propositions \ref{prop-m-invt} and \ref{prop-proj},
\begin{cor}
The operator $\cRG:\cU\to\cM(U,V)$ is analytic. Moreover, it coincides with $\cR$ on the submanifold $\cM(U,V)$.
\end{cor}

\subsection{Statement of results}\label{sec:resultsACSP}

\begin{figure}
\centering
\vspace{2mm}       
\begin{tabular}{c c} 
\includegraphics[height=7cm,angle=-90]{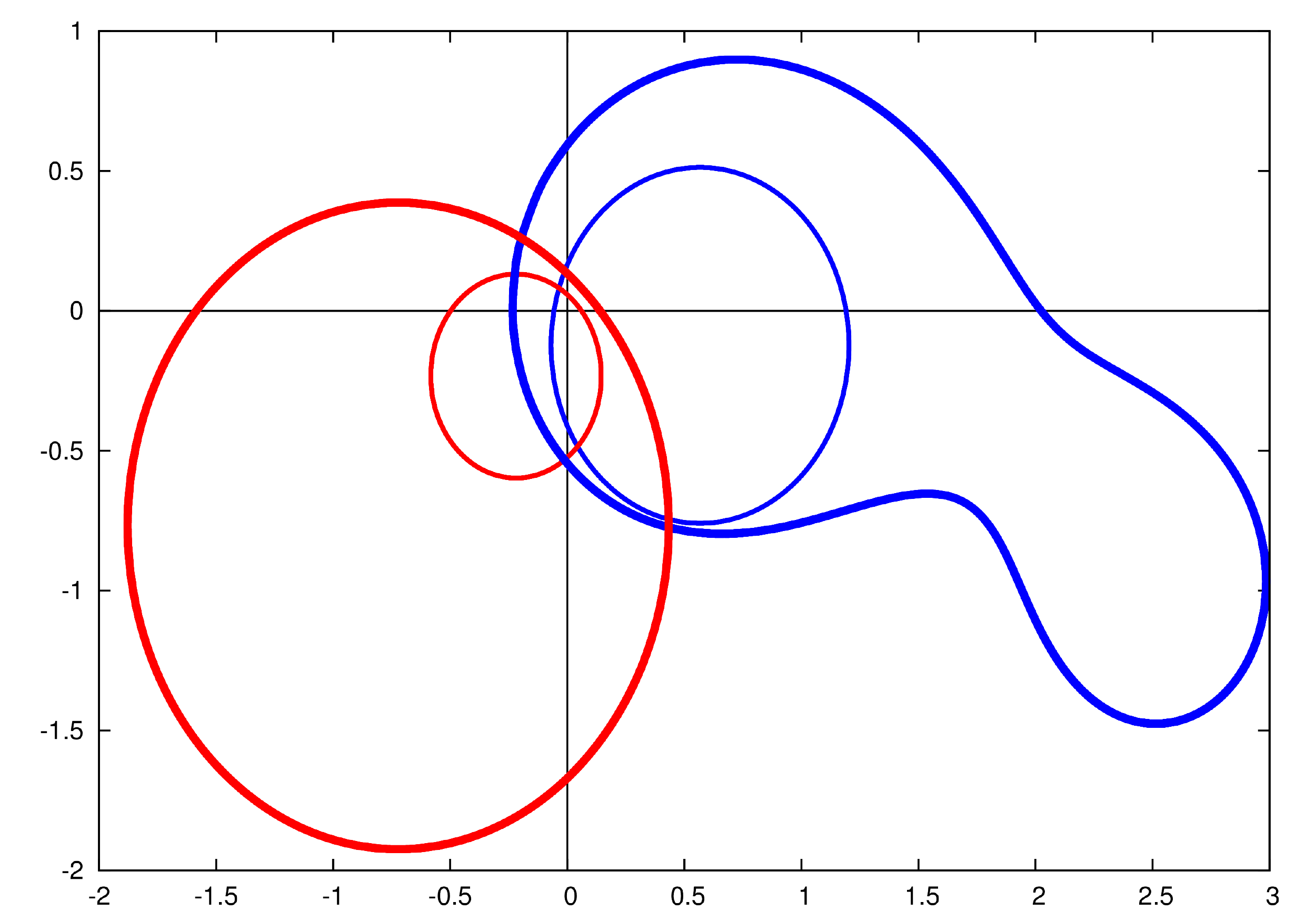} \quad & \quad \includegraphics[height=7cm,angle=-90]{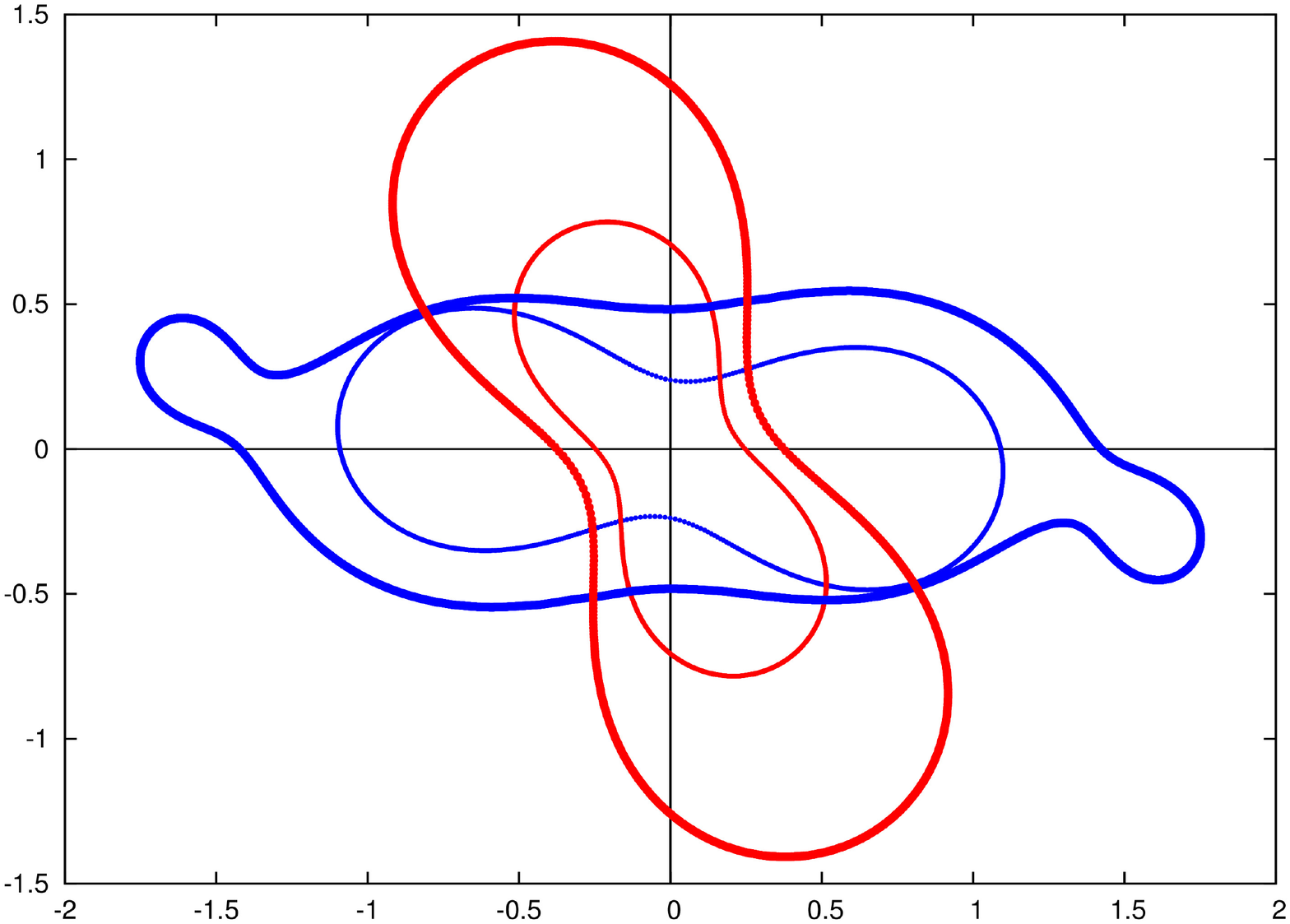}
\end{tabular}
\caption{a) Domains $U$ and $V$ (light blue and red) together with the domains of definition of the renormalization fixed point $(\phi^*,\psi^*)$ (dark blue and red); b) Domains $Z$ and $W$ (light blue and red) together with the domains for the fixed point $\zeta^*$ (dark blue and red)}
\label{fig:domains}       
\end{figure}

Let us specialize to the case when $U$ and $V$ are two disks. Hereafter an
open disk of radius $r$ centered at $z$ in $\fC$ will be denoted by $\fD_r(z)$, and the unit disk will be denoted by $\fD$.
We let
$$U=\D_{r_\phi}(c_\phi)\ni 0\text{ and  }V=\D_{r_\psi}(c_\psi)\ni 0.$$
It will be useful for us to introduce a weighted-$l_1$ Banach norm on a.c.s factors defined in $U$ and $V$ (cf. \cite{GaYa}). Namely, 
we let
 $$\phi_j\equiv\frac{\phi^{(j)}(c_\phi)r_\phi^j}{j!}\text{ and }\psi_j\equiv\frac{\psi^{(j)}(c_\psi)r_\psi^j}{j!},$$
and let 
$\cL_1$ denote the Banach space of pairs $(\phi,\psi)$ 
with the norm
$$||(\phi,\psi)||_1\equiv \sum_{i=1}^\infty \left(  | \Re \phi_i|+| \Im \phi_i|   +  |\Re   \psi_i|+|  \Im \psi_i|    \right).$$
We further  let $\cA_1(U,V)$ be the set of all a.c.s. pairs $(\phi,\psi)$ in $\cL_1$.

\begin{prop} 
\label{norm-equiv}
We note that: 
\begin{enumerate}
\item Let $(\phi,\psi)$ in $\cL_1$. Then $(\phi,\psi)$ in $\cE(U,V)$ and $$||(\phi,\psi)||\leq ||(\phi,\psi)||_1.$$

\item Let $M>1$ and let $\phi$, $\psi$ be two bounded analytic functions in $\DD_{Mr_\phi}(c_\phi)$, $\DD_{Mr_\psi}(c_\psi)$ respectively. Then
$$||(\eta,\xi)||_1<\frac{2M}{M-1}(\sup_{\DD_{Mr_\phi}(c_\phi)}|\phi|+\sup_{\DD_{Mr_\psi}(c_\psi)}|\psi|).$$
\end{enumerate}
\end{prop}

In the statement below we make the following specific choice of centers and radii of the domains $U$, $V$:
\begin{eqnarray}
\nonumber c_\phi&=& \phantom{+} 0.5672961438978619-0.1229664702397770 \cdot i , \quad r_\phi=0.636, \\
\nonumber  c_\psi &=&-0.2188497414079558 -0.2328147240271490 \cdot i, \quad r_\psi= 0.3640985354093064.   
\end{eqnarray}

\begin{thm} $\phantom{aaa}$\\
\label{mainthm1}
Let $U$ and $V$ be as above. Then there exists a pair of polynomials $(\phi_0,\psi_0) \in \cA_1(U,V)$
such that the following holds:
\begin{itemize}
\item[(i)] the operator $\cRG$ is anti-analytic in $B_r((\phi_0,\psi_0))$ in $\cA_1(U,V)$, $r=6.6610992 \cdot 10^{-11}$;
\item[(ii)] there exists a pair $(\phi_*,\psi_*)$ in  $B_r((\phi_0,\psi_0))$ which is fixed by $\cRG$;
\item[(iii)] the functions $\phi_*$ and $\psi_*$ extend analytically and univalently to some domains  $\tU \supset \fD_{\tr_\phi} (\tc_\phi) \Supset U$ and $\tV \supset \fD_{\tr_\psi}(\tc_\psi) \Supset V$, respectively, with $\tr_\phi=0.937$, $\tr_\psi=0.874$,   $\tc_\phi=0.6+i 0.09$ and $\tc_\psi=c_\psi$;
\item[(iv)] the differential $D\cRG|_{(\phi_*,\psi_*)}$ is a compact anti-linear operator;
\item[(v)] the compact linear operator $K\equiv  D \cRG|_{(\phi_*,\psi_*)} \circ c$ (where $c$ is complex conjugation) has a single simple 
eigenvalue outside the closed unit disk, and the rest of the spectrum lies inside the open unit disk.
\item[(vi)] the scaling factor $\lambda_*=\eta_*(0)\approx 0.220265-i0.708481.$
\end{itemize}
\end{thm}


\noindent
The following was proved in \cite{Bur}:
\begin{prop}
\label{fixed-pt-commute}
The renormalization fixed point $\zeta_*=(\eta_*,\xi_*)$ is commuting:
$$\eta_*\circ\xi_*=\xi_*\circ\eta_*,$$
where defined.
\end{prop}

\begin{proof}

Let us assume the contrary, so that
$$[\eta_*,\xi_*](z) = az^k+o(z^k),\text{ where }k\geq 3\text{ and }a\neq 0.$$
Let us first calculate the constant 
$$C=\eta_*'(\eta_*\circ \xi_*(0)).$$
To that end, we start with one of the two fixed point equations for $\zeta_*$ in the form
$$\eta_*\circ\psi_*(z)=c\circ \lambda_*^{-1}\circ\eta_*\circ\xi_* \circ\phi_*(\lambda_*^2\bar{z}).$$
Differentiating the above equation, and using the equality of the expressions (\ref{eq:2etaxider}) and (\ref{eq:2xietader}),
we obtain $$C=\lambda_*^{-1}.$$
Now, write
\begin{eqnarray}
\nonumber [c\circ \! \lambda_*^{-1}\eta_*\circ\xi_* \circ \lambda_* \circ c,c \circ \! \lambda_*^{-1}\eta_* \circ \lambda_* \circ c](z)\!\!&\!\!=\!\!&\!\!c\circ \! \lambda_*^{-1}\eta_*(\xi_*\circ\eta_*(\lambda_* \bar z))-c\circ \! \lambda_*^{-1}\eta_*(\eta_*\circ\xi_*(\lambda_* \bar z))\\
\nonumber \!\!& \!\! = \!\!& \!\!c(\lambda_*^{-1}C a (\lambda_*\bar z)^k)+o((\lambda_*\bar z)^k) \\
\nonumber \!\!& \!\! = \!\!& \!\!\bar{a} \bar{\lambda}_*^{k-2} z^k+o((\lambda_*\bar z)^k).
\end{eqnarray}
Since $$|\lambda_*^{k-2}|\leq|\lambda_*| < 1,$$
we have arrived at a contradiction.
\end{proof}

%% file: nonsymmetric.tex
\section{Renormalization for almost commuting holomorphic pairs}\label{sec:ACM}
We are now going to use Main Theorem 1 to extend our hyperbolicity result to a.c. pairs in $\cB(Z,W)$.

Let us state several useful lemmas first.
\begin{lemma}
\label{lem-cpt1}
Let $Z'\Supset Z$, $W'\Supset W$, $a\in Z$, $b\in W$, and $C>0$. Let $S$ denote the subset of $\cH(Z',W')$ consisting of the pairs $(\eta,\xi)$
with $|\eta'(a)|<C$ and $|\xi'(b)|<C$. Then $S$ is pre-compact in $\cH(Z,W)$.
\end{lemma}
\begin{proof}
This is an immediate corollary of Koebe Distortion and Arzel{\`a}-Ascoli Theorems.
\end{proof} 

The next statement is a direct analogue of Proposition \ref{prop-proj}:
\begin{prop}
\label{prop-proj2}
There exists an open neighborhood $\cU$ of $\cB(Z,W)$ in $\cC(Z,W)$ such that the following holds.
Suppose that a pair of maps $(\eta,\xi) \in \cU$. Then there exists a unique triple of complex numbers $a$, $b$, $c$ such that 
\begin{eqnarray}
\cP (\eta,\xi)(z)\equiv\left( \eta(z)+a z^4+b z^6, \xi(z)+c \right)\in\cB(Z,W).
\end{eqnarray}
Moreover, the dependence of the numbers $a$, $b$, and $c$ on $(\eta,\xi)\in\cU$ is analytic.
\end{prop}
\begin{proof}
The computation is even simpler than before. Again, 
\begin{equation}\label{eq:newFs1a}
c=1-\xi(0).
\end{equation}
Furthermore, the commutation conditions reduce to:
$$\left\{ \begin{array}{c}a+b=\xi(\eta(0)) +c -\eta(1),\\
4a+6b=\xi'(\eta(0))\eta''(0)/\xi''(0)-\eta'(1)\end{array}\right.$$
Therefore, if $(\phi,\psi)$ is a small perturbation of  an a.c. pair $(\eta,\xi)$, then this set of equations has a unique solution $(a,b,c)$ which depends on $(\eta,\xi)$ analytically.
\end{proof}

To define renormalization of pairs $(\eta,\xi)\in\cC(Z,W)$ without a condition of quadratic symmetry, we note that every 
such pair can be decomposed as 
\begin{equation}\label{def:acm}
\zeta=(\eta,\xi)=(\phi \circ q_2 \circ \alpha,  \psi \circ q_2 \circ \beta),
\end{equation}
where $\alpha$ and $\beta$ are holomorphic on $Z$ and $W$ respectively, and satisfy 
\begin{equation}\label{def:tangenttox}
\alpha(x)=x+O(x^2), \quad \beta(x)=x+O(x^2).
\end{equation}
Let us denote $\cW(Z)$ the space of holomoprhic maps in $Z\ni 0$ with the property $\alpha(x)=x+O(x^2)$ -- this is clearly an (affine) Banach submanifold of the space $\cH(Z)$.
Let $\cU(Z)\subset \cW(Z)$ consist of maps with $|\alpha'(z)|>0$ for $z\in Z$, and let $\cU(Z,W)\equiv \cU(Z)\times\cU(W)$.
Following Martens \cite{Mar}, we introduce:

\begin{definition}
The {\it nonlinearity operator} $N: \cU(Z) \mapsto \cH(Z)$
is defined as
$$N[\alpha]=D \log{  D \alpha }={\alpha'' \over \alpha'},$$
so that $N[\alpha](x)$ is the nonlinearity of $\alpha$ at the point $x$.
\end{definition}

\noindent
The following Lemma results from a straightforward computation.

\begin{lemma}
The nonlinearity operator $N:\cU(Z)\to N(\cU(Z))$ is a bijection onto its image, the inverse being defined as 
\begin{equation}\label{inv_non_op}
N^{-1}[\alpha](x)= \int_0^x \exp\left\{\int_0^z\alpha(w) d w  \right\} d z,
\end{equation}
where the integrals are taken along any Jordan arc contained in $Z$ with the endpoints $0$ and $x$.
\end{lemma}

\noindent
Let us denote $$\cS(Z)\equiv N^{-1}(\cH(Z))\subset \cU(Z)\text{, and }\cS(Z,W)\equiv \cS(Z)\times\cS(W).$$
The identification of $\cS(Z)$ with its image under $N$ induces a Banach space structure with the 
{\it nonlinearity norm}:
\begin{equation}
\label{n_lnorm}\|\alpha\|_N=\sup_{x \in Z} \left|N[\alpha](x) \right|\text{ for }(\alpha,\beta)\in \cS(Z),
\end{equation}
and a linear structure given by 
\begin{equation} 
\label{nl_lin_struct} a \pmb{\cdot} \alpha \pmb{+} b \pmb{\cdot} \beta=N^{-1}[a  N[\alpha]+b  N[\beta]],\text{ for }a,b\in\CC.
\end{equation}

For $\eps>0$ let us denote $\cS(Z,W,\eps)$ the $\eps$-ball
$$\cS(Z,W,\eps)=\{(\alpha,\beta)\in\cS(Z,W)\;|\;\max(||\alpha||_N,||\beta||_N)<\eps\}.$$
Let $U$, $V$ be two subdomains of $\CC$ and let $$Z\Supset q_2^{-1}(U),\;W\Supset q_2^{-1}(V).$$
Let $\eps>0$ be sufficiently small, so that for every $(\alpha,\beta)\in \cS(Z,W,\eps)$ we have 
\begin{equation}\label{eq:domains}
\alpha(Z)\Supset q_2^{-1}(U),\;\beta(W)\Supset q_2^{-1}(V).
\end{equation} 
We denote $$\cE(U,V,Z,W,\eps)\equiv \cE(U,V)\times \cS(Z,W,\eps).$$
Every pair $(\phi,\psi)\in\cE(U,V)$, $(\alpha,\beta)\in \cS(Z,W,\eps)$  naturally corresponds to  
\begin{equation}
\label{comp-pair}
\zeta=(\eta,\xi)\equiv (\phi\circ q_2\circ \alpha,\psi\circ q_2\circ \beta),
\end{equation}
and we will use the notations interchangeably.
We let $\cM(U,V,Z,W,\eps)\subset\cE(U,V,Z,W,\eps)$ be the subset consisting of almost commuting pairs (\ref{comp-pair}).

Similarly, we let
$$\cL_1(U,V,Z,W,\eps)\equiv \cL_1(U,V)\times \cS(Z,W,\eps),$$
and again identify a point in this Banach manifold with a pair $(\eta,\xi)$ via (\ref{comp-pair}). 
Consider a natural metric on this product  given by 
\begin{equation}
\label{eq:unorm}d_{1,N}(\zeta_1,\zeta_2)=\|(\phi_1,\psi_1)-(\phi_2,\psi_2)\|_1+\|(\alpha_1,\beta_1) \pmb{-} (\alpha_2,\beta_2)\|_N.
\end{equation}
We let $\cA(U,V,Z,W,\eps)\subset\cL_1(U,V,Z,W,\eps)$ be the subset consisting of almost commuting pairs.

Recall, that for any two $\alpha_i$, $i=1,2$, in  $S(Z)$,
\begin{equation}\label{n3a}
e^{-\| \alpha_1  \pmb{-} \alpha_2 \|_N } \le \left| { \alpha_1'(x) \over  \alpha_2'(x)} \right|  \le e^{\| \alpha_1 \pmb{-} \alpha_2 \|_N } \quad  {\rm and} \quad  {1 \over C} e^{-\| \alpha_i \|_N } \le \left| \alpha_i'(x) \right| \le C e^{\| \alpha_i \|_N },
\end{equation}
where $C$ depends only on the domain $Z$ and its image $\alpha_i(Z)$ (notice, that the difference  $\alpha_1  \pmb{-} \alpha_2$ above is defined as $(\ref{nl_lin_struct})$). Therefore,
\begin{equation}\label{n3}
|\alpha_1(z)-\alpha_2(z)|=\left| \int_0^z \alpha_1'(z) \left(1-{\alpha_2'(z) \over \alpha_1'(z) }  \right)  d z  \right| \le  C  e^{\| \alpha_1 \|_N } \left(e^{\| \alpha_1 \pmb{-}\alpha_2 \|_N }  -1 \right),   
\end{equation}
where, the constant depends on the domains $Z$ and its image $\alpha_i(Z)$.

Notice, that if $+$ and $\| \cdot \|$ are the usual operations of addition and the norm in $\cH(Z,W)$, then, whenever $q_2(\alpha_i(Z)) \Subset U$  and $0 \in \alpha_i(Z)$,  and  $\alpha \in \cS(Z)$ with $||\alpha||_N<\eps$,  according to $(\ref{n3})$,
\begin{eqnarray}
\nonumber \|\eta_1 -\eta_2 \| & = & \| \phi_1 \circ q_2 \circ \alpha_1 -  \phi_2 \circ q_2 \circ \alpha_2   \| \\
\nonumber &=&   \| (\phi_1-\phi_2) \circ q_2 \circ \alpha_1 +(\phi_2 \circ q_2 \circ \alpha_1 - \phi_2 \circ q_2 \circ \alpha_2)   \|  \\
\nonumber &\le&  \| \phi_1-\phi_2 \|+  \sup_{z \in \alpha_1(Z) \cup \alpha_2(Z)}\{\phi_2'(z^2) 2 z    \}   \|(\alpha_1 -  \alpha_2)   \| \\
\nonumber &\le&   \| \phi_1-\phi_2 \|_1+  C \| \alpha_1 \pmb{-} \alpha_2\|_N,
\end{eqnarray}
where $C$ depends on the domains, their images and $\epsilon$. A similar conclusion holds for  $\|\eta_1 -\eta_2 \|$. Therefore, if two pairs $\zeta_1$, $\zeta_2$ of the form (\ref{comp-pair}) belong to $\cA(U,V,Z,W,\eps)$, then
\begin{equation} \label{eq:normsbounds}
\|\zeta_1-\zeta_2 \| \le C \ d_{1,N}(\zeta_1,\zeta_2).
\end{equation}

We remark, that completely analogously to Proposition \ref{prop-submfld1}, we have the following:
\begin{prop}
\label{prop-submfld2} Let $\cW\subset \cE(U,V)$  be an open set in which $\psi'(0)\neq 0$ (for instance, small open neighborhoods
 of a pair of factors in which $\psi$ is locally univalent at the origin).
Then, for $\eps>0$ small enough, the  space  $(\cM(U,V)\cap \cW)\times\cS(Z,W,\eps)$ is a Banach submanifold of $\cE(U,V,Z,W,\eps)$.
\end{prop}
\begin{proof}
Indeed,
the commutation conditions (\ref{eq:accond0}) and (\ref{eq:accond1}) for the a.c.h. pairs do not change, while the condition (\ref{eq:accond2}) becomes:
$$\phi'\left(\left(\alpha(\psi(0))\right)^2 \right)  \alpha(\psi (0))  \alpha'(\psi(0)) \psi'(0)-\psi'(\beta(\phi(0))^2)  \beta(\phi (0)) \beta'(\phi(0))  \phi'(0)=0.$$
This condition is a perturbation of the condition (\ref{eq:accond2}). Therefore, if $\eps$ is sufficiently small, then the determinant $(\ref{eq:determinant})$ is non-zero.
\end{proof}

Consider a renormalizable a.c. pair (\ref{comp-pair}) (in the sense of the Definition \ref{def:renormalizability}).

Its renormalization is given by 
\begin{eqnarray}
\nonumber \cR (\eta,\xi)&=&\left(c \circ \lambda^{-1} \circ \eta \circ \xi \circ \lambda \circ c  ,  c \circ \lambda^{-1} \circ \eta \circ \lambda \circ c \right)\\
\nonumber &=& \left(c \circ \lambda^{-1} \circ \phi \circ q_2 \circ \alpha \circ \psi \circ q_2  \circ \beta \circ \lambda \circ c  ,  c \circ \lambda^{-1} \circ \phi \circ q_2  \circ \alpha \circ \lambda \circ c \right) \\
\nonumber &=& \left(\left\{c  \hspace{-0.4mm} \circ \hspace{-0.4mm} \lambda^{-1} \! \! \circ \hspace{-0.4mm} \phi \hspace{-0.4mm} \circ \hspace{-0.4mm} q_2 \hspace{-0.4mm} \circ \hspace{-0.4mm} \alpha \hspace{-0.4mm} \circ \hspace{-0.4mm} \psi \hspace{-0.4mm} \circ \hspace{-0.4mm} \lambda^2 \hspace{-0.4mm} \circ \hspace{-0.4mm} c \hspace{-0.4mm} \right\} \! \circ \hspace{-0.4mm} q_2 \hspace{-0.4mm} \circ  \! \left\{ \hspace{-0.4mm} c \hspace{-0.4mm} \circ \hspace{-0.4mm} \lambda^{-1} \! \! \circ \hspace{-0.4mm} \beta \hspace{-0.4mm} \circ \hspace{-0.4mm} \lambda \hspace{-0.4mm} \circ \hspace{-0.4mm} c \hspace{-0.4mm} \right\}\right. , \\
\nonumber &\phantom{=}& \left. \ \left\{ \hspace{-0.4mm} c \hspace{-0.4mm} \circ \hspace{-0.4mm} \lambda^{-1} \! \! \circ \hspace{-0.4mm} \phi \hspace{-0.4mm}\circ \hspace{-0.4mm} \lambda^2 \hspace{-0.4mm} \circ \hspace{-0.4mm} c \hspace{-0.4mm} \right\} \! \circ  \hspace{-0.4mm} q_2 \hspace{-0.4mm} \circ \! \left\{ \hspace{-0.4mm} c \hspace{-0.4mm} \circ  \hspace{-0.4mm}  \lambda^{-1}\! \!  \circ \hspace{-0.4mm} \alpha \hspace{-0.4mm} \circ \hspace{-0.4mm} \lambda \hspace{-0.4mm} \circ \hspace{-0.4mm} c \hspace{-0.1mm} \right\}\right) \\
\label{eq:tildezeta} &=& \left( \tilde{\phi} \circ q_2 \circ \tilde{\alpha}, \tilde{\psi} \circ q_2 \circ \tilde{\beta} \right).
\end{eqnarray}
Note that $\tilde{\alpha}$ and $\tilde{\beta}$ are again tangent to the identity at the origin.

We further set
\begin{equation}
\cRG(\eta,\xi)=P \circ \cR(\eta,\xi),
\end{equation}
where $P$ is the projection on the almost commuting subspace (Proposition \ref{prop-proj2}).

\noindent
Using the same notation as for symmetric almost commuting pairs should not lead to any confusion since, by Proposition~\ref{prop-proj2} the two definitions of $\cRG$ coincide for a symmetric renormalizable pair.

We have 
$$N[\tilde{\alpha}]=\overline{ \lambda N[\beta] \circ \lambda \circ c}, \quad N[\tilde{\beta}]=\overline{ \lambda N[\alpha] \circ \lambda \circ c}.$$
Therefore, the inclusions $(\ref{eq:inclusion1})$-$(\ref{eq:inclusion3})$ imply that
\begin{eqnarray}
\label{2nd-coord-shrink}
\|(\tilde{\alpha},\tilde{\beta} )\|_N \le |\lambda| \|(\alpha,\beta) \|_N. 
\end{eqnarray}

Therefore, if  $\zeta_1=(\phi \circ q_2, \psi \circ q_2)$ and $\zeta_2=(\phi \circ q_2 \circ \alpha, \psi \circ q_2 \circ \beta)$ are renormalizable  pairs, then
\begin{eqnarray}
\nonumber d_{1,N}(\cR \zeta_1,\cR \zeta_2)\!\! & \!\! = \!\! & \!\! \|(c \circ \lambda^{-1} \circ \phi \circ q_2 \circ \psi \circ \lambda^2 \circ c   -  c \circ \lambda^{-1} \circ \phi \circ q_2 \circ \alpha \circ \psi \circ \lambda^2 \circ c  ,0)\|_1+\\ 
\nonumber &\phantom{=}&\phantom{\!\! \|(c \circ \lambda^{-1} \circ \phi \circ q_2 \circ \psi \circ \lambda^2 \circ c}+\|(\tilde{\alpha},\tilde{\beta})\|_N \\
 \label{eqnren1} & <& C \| \alpha \|_N +|\lambda| \|(\alpha,\beta) \|_N,
\end{eqnarray}
where $C$ depends on the domains and the function $\phi$, and
\begin{eqnarray}\label{eq:lambdaNnorm1}
\nonumber d_{1,N}((\tilde{\phi} \! \circ \! q_2,\tilde{\psi} \! \circ \! q_2) ,\cR \zeta_1 )\!\!&\!\!=\!\!&\!\! \|(c \! \circ \! \lambda^{-1} \!  \circ \!  \phi \!  \circ  \! q_2 \! \circ \! \psi \! \circ \!   \lambda^2 \!  \circ \! c   \! - \! c \!  \circ \!  \lambda^{-1} \!  \circ \!  \phi \!  \circ \!  q_2 \! \circ \!  \alpha \!   \circ \!  \psi \!  \circ \!  \lambda^2 \!  \circ  \!  c,0)\|_1  \\
\nonumber &<& C \| \alpha \|_N.
\end{eqnarray}

As a consequence, let $B_r((\phi_0,\psi_0))$ be as in \thmref{mainthm1}. Then there exist $\eps>0$ and $\gamma\in(0,1)$ such that  for all $\zeta$ with $||(\alpha,\beta)||_{N}<\eps$ and $(\phi,\psi) \in B_r((\phi_0,\psi_0))$ we have:
\begin{equation}\label{eq:lambdaNnorm2}
d_{1,N}((\tilde{\phi} \circ q_2,\tilde{\psi} \circ q_2),\cR \zeta) \le \gamma \|(\alpha,\beta)\|_N.
\end{equation}


\subsection{Hyperbolicity of renormalization for non-symmetric pairs}
Let $\upsilon$ be the forgetful map from $\cA(U,V,Z,W)$ to $\cA_1(U,V)$ sending
$$(\phi,\psi,\alpha,\beta)\mapsto (\phi,\psi).$$
Denote $\iota$ its one-sided inverse: the natural inclusion of $\cA_1(U,V)$ into $\cA(U,V,Z,W)$ given by
$$(\phi,\psi)\mapsto(\phi,\psi,\text{Id},\text{Id}).$$
Let $U$, $V$, $\tl U$, $\tl V$  be as in \thmref{mainthm1}, and let $(\phi_*,\psi_*)$ be the hyperbolic fixed point of $\cRG$ in
$\cA_1(U,V)$ constructed in this theorem. 
Clearly, 
$$\zeta_*\equiv (\phi_*,\psi_*,\text{Id},\text{Id})$$
is a fixed point of $\cRG$ in 
$$\cA\equiv \cA(U,V,Z,W).$$
Let  $L\equiv  D \cRG|_{(\phi_*,\psi_*)} \circ c$, where $c$ is the complex conjugation, as before.

We prove:
\begin{thm}\label{mainthm2} 
There exist topological disks $\tl Z \Supset Z$ and $\tl W \Supset W$,  neighborhood $\cW$ of $\zeta_*$ in  $\cA(U,V,Z,W)$,  such that the following holds.

\begin{itemize}
\item[$(i)$] The transformation $\cRG$ is an anti-analytic operator from $\cW$ to $\cA(\tl U,\tl V,\tl Z, \tl W)$.
\item[$(ii)$] The differential $D \cRG|_{\zeta_*}$ is a compact anti-linear operator in $T_{\zeta_*}\cA(U,V,Z,W)$. The operator $L=D\cRG|_{\zeta_*} \circ c$ has a single eigenvalue outside the closed unit disk, which coincides with the unstable eigenvalue of the operator $K$, constructed in \thmref{mainthm1}. The rest of the spectrum lies incide the open unit disk.
\end{itemize}
\end{thm}

\begin{proof} $\phantom{a}$ \\
\medskip
\noindent $(i)$ We recall that by $(\ref{eq:domains})$
$$\alpha(Z)\Supset q_2^{-1}(U),\;\beta(W)\Supset q_2^{-1}(V).$$
 By  part  $(iii)$   of  \thmref{mainthm1}, there exists  $\cW \subset \cA(U,V,Z,W)$ such that  $\cRG \zeta$  for every $\zeta \in \cW$  extends analytically to some $\tl Z  \cup \tl W$, such that
$$\alpha(\tl Z)\Supset q_2^{-1}(\tl U),\;\beta(\tl W) \Supset q_2^{-1}(\tl V).$$

\noindent $(ii)$  Compactness of $D\cRG_{\zeta_*}$ follows from part $(i)$ and Lemma~\ref{lem-cpt1}. By the Spectral Theory of compact operators, every element of the spectrum of the operator $L$, except the origin,  is an isolated eigenvalue. Let us first note that if $\mu\in\text{sp}(L)$ and $|\mu|\geq 1$ then every corresponding eigenvector has a form:
$$((\bar u,\bar v), (0, 0)).$$
In other words, all eigenvectors which are not stable must lie in the tangent space to the natural inclusion 
$\iota(\cA_1(U,V))$. 
Indeed, since the projection $P$  analytically depends on the pair (Proposition~\ref{prop-proj2}), and coincides with the identity on almost commuting pairs, if $\cW$ is sufficiently small, then there exists $\lambda_* <\tl \lambda <1$ such that  the inequality (\ref{2nd-coord-shrink}) holds for $(\hat{\alpha},\hat{\beta})$ with $\tl \lambda$, where 
$$\cRG \zeta=( \hat{\phi} \circ q_2 \circ \hat{\alpha}, \hat{\psi} \circ q_2 \circ \hat{\beta}).$$ 
Hence the part of the spectrum of $L$ with non-zero $(\alpha,\beta)$-coordinates must lie in the disk $\{|z|<\lambda\}$.

Finally, the spectrum of the operator $L \arrowvert_{T_{\zeta_*}\iota(A_1(U,V))}$, coincides with that of $K$ from \thmref{mainthm1}, and the rest of the claim follows.

\end{proof}

Recall, that $\cC(Z,W)$ denotes the Banach submanifold of $\cH(Z,W)$ given by the linear conditions 
$$\eta'(0)=\xi'(0)=0,$$ and that $\cB(Z,W)$ is the submanifold of $\cC(Z,W)$ consisting of almost commuting pairs. We reformulate Theorem\ref{mainthm2}  using the uniform norms as follows:

\begin{thm}\label{mainthm3} Let $Z$, $W$, $\tl Z$ and $\tl W$ be as in Theorem~\ref{mainthm2}. There exist a pair of topological disks $Z_1$ and $W_1$,
\begin{equation} \label{eq:ZWs}
\tl Z\Supset Z_1 \Supset Z, \quad \tl W \Supset W_1 \Supset W,
\end{equation}
such that
\begin{itemize}
\item[$(i)$]   The operator $\cRG$ is an anti-analytic operator from an open neighborhood of $\zeta_*$ in $\cC(Z_1,W_1)$ to $\cC(\tl Z, \tl W)$. 
\item[$(ii)$]  The differential $D \cRG \arrowvert_{\zeta_*}$ is a compact anti-linear operator. Denote $M\equiv D {\cRG} \arrowvert_{\zeta_*}\circ c$. Then $M$ has a single simple eigenvalue outside of the closed unit disk, and the rest of the spectrum of $M$ lies inside the open unit disk.
\end{itemize}
\end{thm}
\begin{proof} $\phantom{oo}$\\
\noindent $(i)$  By part $(1)$ of  Proposition $\ref{norm-equiv}$ the map $\cRG(\zeta) \in  \cC(\tl Z,\tl W)$ whenever $\zeta \in \cW \subset \cA(U,V,Z,W)$.  Choose any topological disks $Z_1$ and $W_1$ satisfying $(\ref{eq:ZWs})$, and such that, additionally, $q_2(Z_1) \Subset U$, $q_2(W_1) \Subset V$.

Let $\zeta \in C(Z_1,W_1)$ and $h \in C(Z_1,W_1)$, then by part $(2)$ of  Proposition $\ref{norm-equiv}$, $\zeta+t h$ is in $\cA(U,V,Z,W)$, and by part $(i)$ of  Theorem~\ref{mainthm2},
$\cRG(\zeta) + t D \cRG \arrowvert_\zeta h \in \cA(\tl U, \tl V, \tl Z,\tl W)$, and 
$$ d_{1,N} \left( \cRG(\zeta+t h), \cRG(\zeta) + t D \cRG \arrowvert_\zeta h \right)=\eps(\zeta,h,t)=o(t).$$

By $(\ref{eq:normsbounds})$, 
$$  \| \cRG(\zeta+t h) -\left(\cRG(\zeta) + t D \cRG \arrowvert_\zeta h \right) \| \le C \ d_{1,N} \left( \cRG(\zeta+t h), \cRG(\zeta) + t D \cRG \arrowvert_\zeta h \right)=o(t),$$
where $\| \cdot \|$ is the norm in $C(\tl Z, \tl W)$ and   $d_{1,N}$ is the distance in $\cA(\tl U, \tl V,\tl Z, \tl W)$, and we see that $\cRG$ is analytic from $C(Z_1,W_1)$ to $C(\tl Z, \tl W)$.

\medskip

\noindent $(ii)$ Compactness of $D_{\zeta_*}\cRG$ follows by  Lemma~\ref{lem-cpt1}. By the Spectral Theory of compact operators, every element of the spectrum of the operator $M$, except the origin,  is an isolated eigenvalue. Let $\bar v\in\cC(Z_1,W_1)$ be an eigenvector of $M$ such that the corresponding eigenvalue $\lambda$ is non-zero. By \thmref{mainthm1}, the vector $\bar v$ has an analytic continuation to $(\tl Z,\tl W)$. Let us represent $\bar v$ as a decomposition
$$\bar v=(\phi_1\circ q\circ\alpha_1,\psi_1\circ q\circ \beta_1),$$
 where the conformal maps $\alpha_1$ and $\beta_1$ are defined in $\tl Z$, $\tl W$, and therefore have finite nonlinearity norms in $Z_1$, $W_1$. Since $M\bar v=\lambda\bar v$, the same considerations as in the 
proof of Theorem\ref{mainthm2} imply that 
$$\alpha_1=\beta_1=\text{Id}.$$
\end{proof}

%% file: circle.tex
\section{Invariant quasi-arc}
\label{sec:quasiarc}

\begin{thm}
\label{mainthm4}
Consider a pair $\zeta\in W^s(\zeta_*) \subset C(Z_1,W_1)$. Assume further that $\zeta=(\eta,\xi)$ is a commuting, rather than approximately commuting, pair, that is
$$\eta\circ\xi(z)=\xi\circ\eta(z)$$
in a neighborhood of $0$ for the analytic continuations of the two maps. This means that successive pre-renormalizations of $\zeta$
are compositions of iterates of $\eta$ and $\xi$.
Then there exists a $\zeta$-invariant Jordan arc $\gamma\ni 0$ such that
the restriction of $\zeta$ to $\gamma$ is quasi-symmetrically conjugate to the pair $H$ by a conjugacy that maps $0$ to $0$.
\end{thm}
The proof follows the work of Stirnemann \cite{Stir}. We start by introducing a convenient notation of multi-indices below.

\subsection{Dynamical partition and multi-indices}
Consider the space $\cI$ of multi-indices $\bar s=(a_1,b_1,a_2,b_2,\ldots,a_n,b_n)$ where $a_j\in \NN$ for $2\leq n$, $a_1\in\NN\cup\{0\}$,
$b_j\in\NN$ for $1\leq j\leq n-1$, and $b_n\in\NN\cup\{0\}$. 
We introduce a partial ordering on multi-indices:
$\bar s\succ \bar t$ if $\bar s=(a_1,b_1,a_2,b_2,\ldots,a_n,b_n)$, $\bar t=(a_1,b_1,\ldots,a_k,b_k,c,d)$, where $k<n$ and 
either $c< a_{k+1}$ and $ d=0$ or $c=a_{k+1}$ and $d< b_{k+1}$. For such a pair, we also define
$$\bar q\equiv \bar s\ominus \bar t:$$
\begin{itemize}
\item in the case when $d=0$, $\bar q=(a_{k+1}-c,b_{k+1},\ldots, a_n,b_n)$;
\item in the other case, $\bar q=(0,b_{k+1}-d,a_{k+1},b_{k+2},\ldots, a_n,b_n).$
\end{itemize}

For a pair of maps $\zeta=(\eta,\xi)$ and $\bar s $ as above we will denote 
$$\zeta^{\bar s}\equiv\xi^{b_n}\circ\eta^{a_n}\circ\cdots\circ\xi^{b_2}\circ\eta^{a_2}\circ\xi^{b_1}\circ \eta^{a_1}.$$
Similarly, 
$$\zeta^{-\bar s}\equiv (\zeta^{\bar s})^{-1}=(\eta^{a_1})^{-1}\circ(\xi^{b_1})^{-1}\circ\cdots\circ(\eta^{a_n})^{-1}\circ (\xi^{b_n})^{-1}.$$

Consider a pair $\zeta\in W^s(\zeta_*)$. 

Consider the $n$-th pre-renormalization of $\zeta$:
$$p\cR^n\zeta=\zeta_{n}=(\eta_{n}|_{Z_n},\xi_{n}|_{W_n}).$$
We define  $\bar s_n,\bar t_n\in\cI$ to be such that
$$\eta_{n}=\zeta_{n}^{\bar s_n},\text{ and }\xi_{n}=\zeta_{n}^{\bar t_n}.$$

Let $T:\RR\to\RR$ be the translation $x\mapsto x+\theta_*$, with $\theta_*=(\sqrt{5}-1)/2$. Define
$$f(x)=T^2(x)-1\text{ and }g(x)=T(x)-1,$$
and set  $$I=[g(0),0],\; J=[0,f(0)],\text{ and }H=(f|_I,g|_J).$$ Define $H_n=(f_n,g_n)=(H^{\bar s_n},H^{\bar t_n})$, and set 
$$I_n=[0,g_n(0)],\;J_n=[0,f_n(0)].$$

Now consider the collection of intervals
$$\cP_n\equiv \{H^{\bar w}(I_n)\text{ for all }\bar w\prec \bar s_n\text{ and }H^{\bar w}(J_n)\text{ for all }\bar w\prec \bar t_n\}.$$
It is easy to see that:
\begin{itemize}
\item[(a)] $\underset{T\in\cP_n}\cup T=I\cup J$;
\item[(b)] for any two distinct elements $T_1$ and $T_2$ of $\cP_n$, the interiors of $T_1$ and $T_2$ are disjoint.
\end{itemize}
In view of the above, we call $\cP_n$ the $n$-th dynamical partition of the segment $I\cup J$.


Consider the sequence of domains $$\cV_n \equiv \{\zeta^{\bar w}(Z_n)\text{ for all }\bar w\prec \bar s_n\text{ and }\zeta^{\bar w}(W_n)\text{ for all }\bar w\prec \bar t_n\}.$$

By analogy with the above definition (and somewhat abusing the notation) we call $\cV_n$ the $n$-th dynamical partition of the pair $\zeta$.

\begin{figure}
\centering
\vspace{2mm}       
\begin{tabular}{c c c} 
\includegraphics[height=50mm,width=45mm, angle=-90]{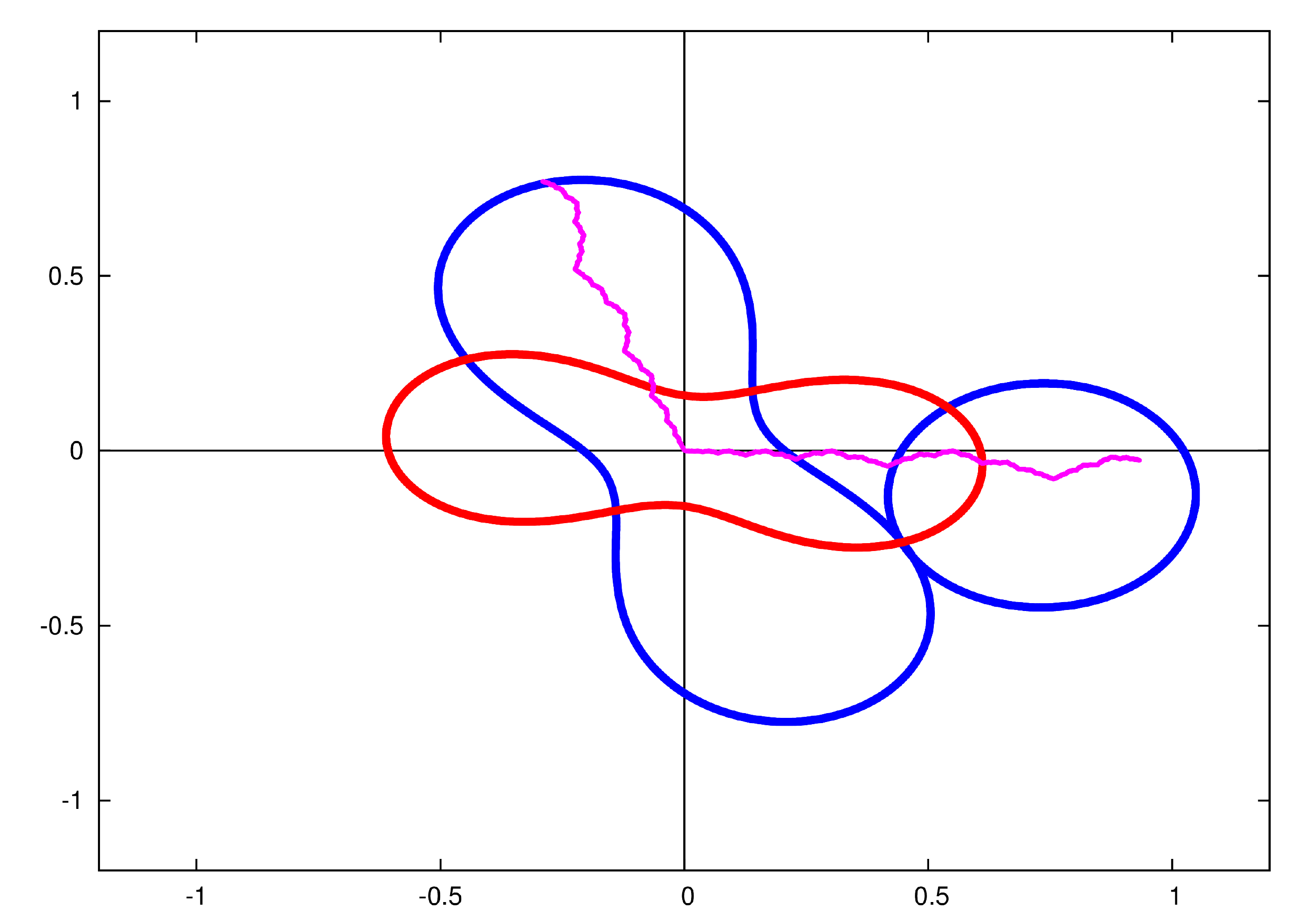} & \includegraphics[height=50mm,width=45mm, angle=-90]{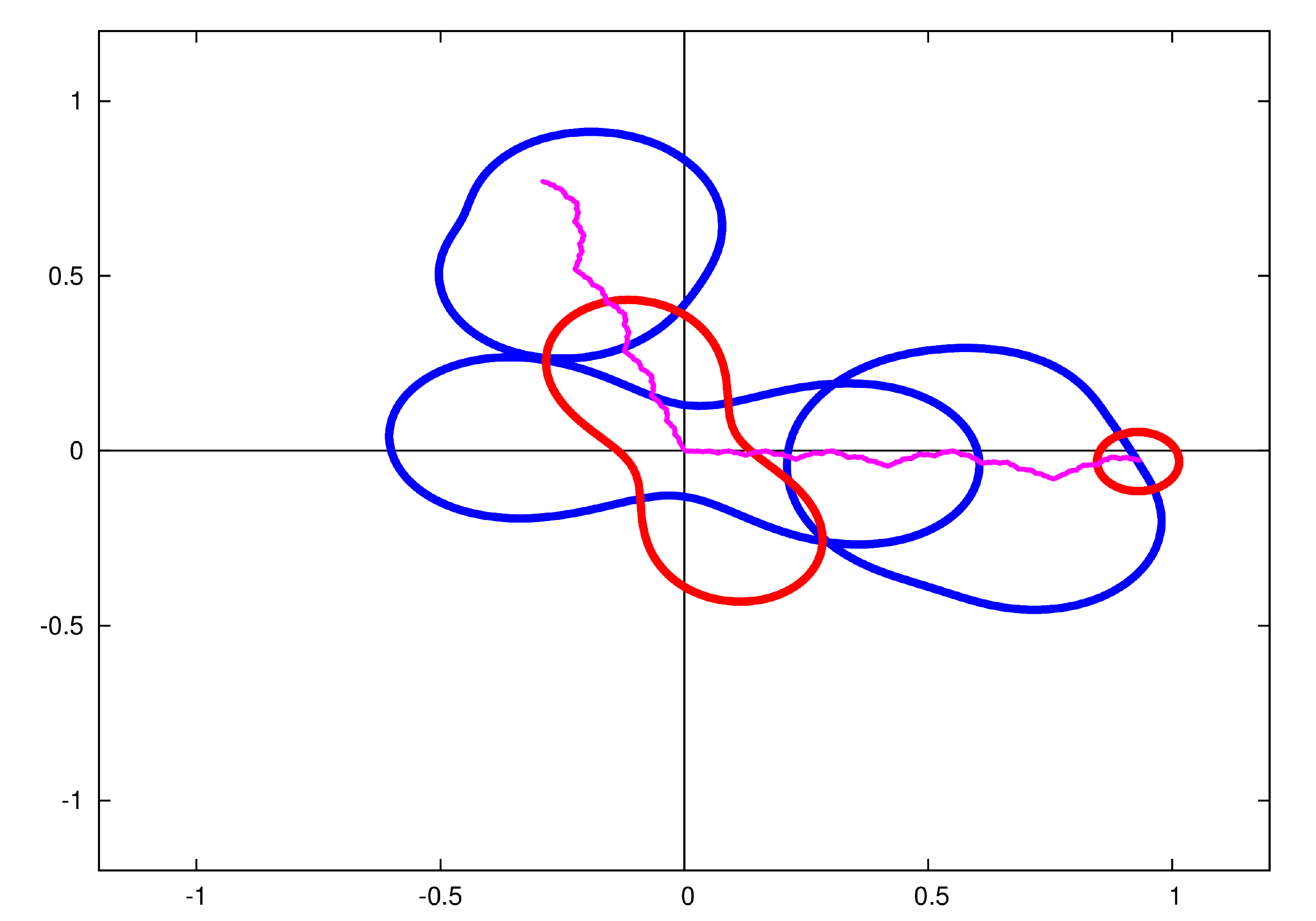}  &  \includegraphics[height=50mm,width=45mm,angle=-90]{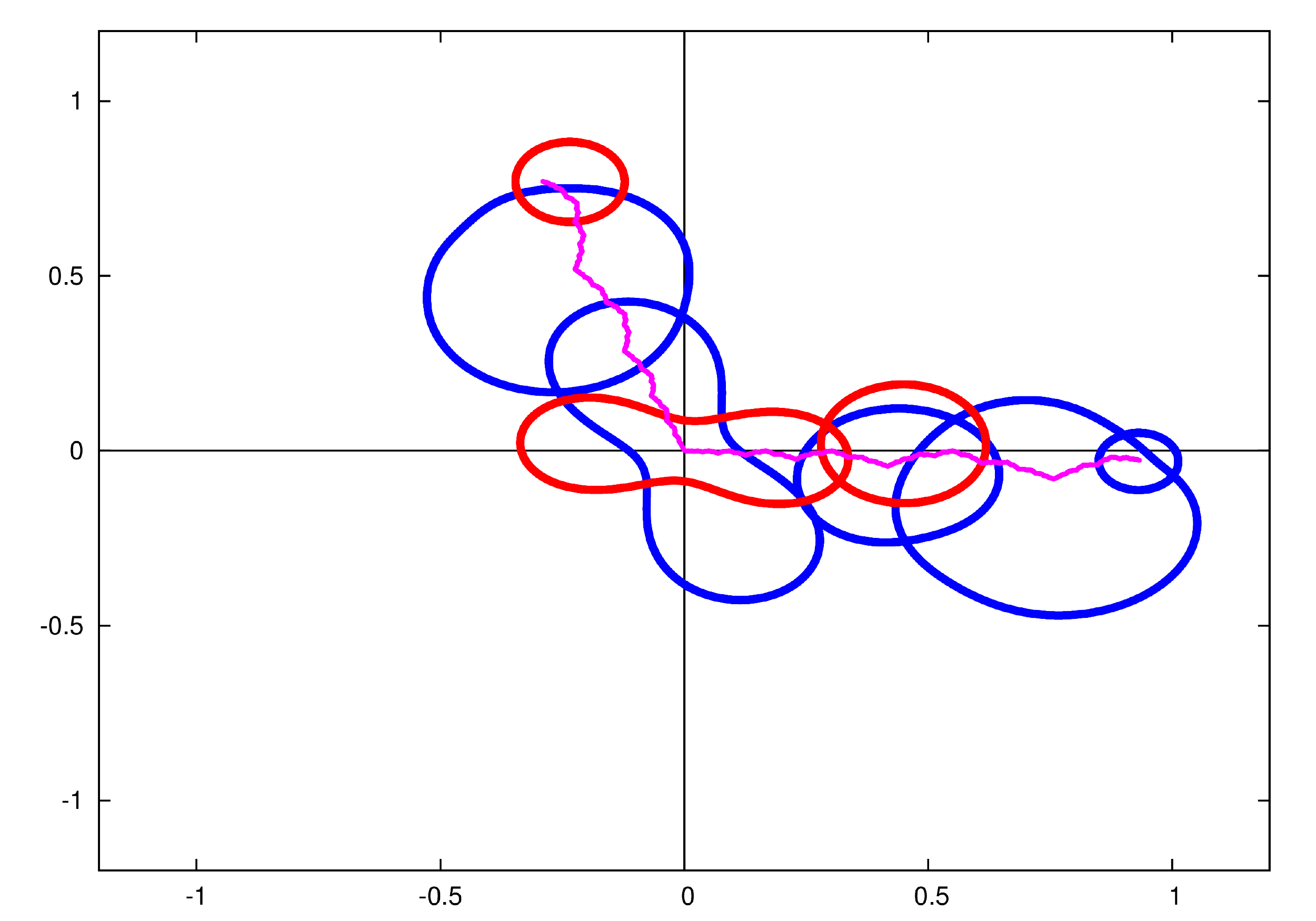}
\end{tabular}
\caption{Partitions: a)  $\cV_1$,  b) $\cV_2$, c) $\cV_3$}
\label{fig:partitions}       
\end{figure}

\noindent
Let us prove the following Lemma:
\begin{lemma}
\label{lem-dyn-part}
There exist $K>1$, $k\in\NN$,  and $C>0$ such that the following properties hold.
\begin{enumerate}
\item For every $n$ and every $W\in\cV_n$, the domain $W$ is a $K$-bounded distortion image of one of the sets $D_1^*=\eta_*(Z_1)$ or $B_n^*=\xi_*(W_1)$.
\item Let $T_{n+k}$ be an interval in the dynamical partition $\cP_{n+k}$ which is contained inside the interval  $S_n\in\cP_n$.
Let $P_{n+k} \in \cV_{n+k}$ and  $Q_n\in \cV_n$ be the elements with the same multi-indices as $T_{n+k}$ and $S_n$. Then
$P_{n+k}\Subset Q_n$ and $$\operatorname{mod}(Q_n\setminus P_{n+k})>C.$$
\item Let $Q_1$ and $Q_2$ belong to the $n$-th partition $\cV_n$ and $Q_1\cap Q_2\neq\emptyset$. Then $Q_1$ is $K$-commensurable with $Q_2$.
\end{enumerate}
\end{lemma}
\begin{proof}
To fix the ideas, let us assume that $W=\zeta^{\bar w}(Z_n)$ and that $W \neq Z_n$.  Consider the domain $D_n=\eta_n(Z_n)$.  Set $\bar q=\bar s_n \ominus \bar w$,  and consider the univalent  inverse branch $\psi=\zeta^{-\bar q}$ which maps $D_n$ to $W$. 

Since $\eta_n$ extends to a strictly larger domain $\tl Z_n$ as a $2$-to-$1$ branched covering map, the branch $\psi$ has a conformal extension to $\tl D_n \Supset D_n$. Since $\zeta \in W^s(\zeta_*)$, the sets $D_n$ and $\tl D_n$ are close to being linear rescalings of the domains $D_1^*=\zeta_*(Z_1)$ and $\tl D_1^*=\zeta_*(\tl Z)$, and therefore,  $\operatorname{mod}(\tl D_n \setminus  D_n)$ is universally  bounded. 

The first claim now follows immediately from Koebe Distortion Theorem applied to $\zeta^{-\bar q}$. 

The second claim follows from the fact that $Z_{n+1}$ is a linearly scaled copy of $Z_n$ with an asymptotically constant scaling factor (and similarly for $W_{n+1}$ and $W_n$.

Finally, the last claim follows from  commensurability of $Z_n$ and $W_n$ and Koebe Distortion Theorem.
\end{proof}

 \begin{proof}[Proof of \thmref{mainthm4}.] Denote $\Delta_n^0=J_n$ and $\Delta_n^1=I_n$. Let $T_n \in \cP_n$ be a nested sequence  $T_n\supset T_{n+1}$ with  $T_n=H^{\bar w_n}(\Delta_n^{\omega_n})$, where $\omega_n \in \{0,1\}$.
 
Denote $Q_n^0=Z_n$ and $Q_n^1=W_n$. Let $P_n \in \cV_n$ the corresponding sequence for $\zeta$: $P_n \cap P_{n+1} \ne 0$ and $P_n=\zeta^{\bar w_n}(Q_n^{\omega_n})$. 

By \lemref{lem-dyn-part} part (2),  $\diam(P_n)\to 0$.
Therefore, $\cap P_n$ is a unique point $y$. Finally,  set $x=\cap T_n$, and $\varphi(x) \equiv y$. 

By \lemref{lem-dyn-part} (1) and (2), the map $\varphi$ is a  topological conjugacy. By \lemref{lem-dyn-part} (3), it is quasi-symmetric.
 \end{proof}



%% file: holomorphic.tex
\section{Holomorphic pairs}
\label{sec:holomorphic}

Let $\eta: \Omega_1 \mapsto \Sigma$ and $\xi: \Omega_2 \mapsto \Sigma$ be two univalent maps  between quasidisks in $\field{C}$, with $\Omega_i \in \Sigma$. Suppose $\eta$ and $\xi$ have homeomorphic extensions to the boundary of domains $\Omega_i$ and $\Sigma$. Following \cite{McM} we call such a pair $\zeta=(\eta,\xi)$ a holomorphic pair if
\begin{itemize}
\item[1)] $\Sigma \setminus \overline{\Omega_1 \cup \Omega_2}$ is a quasidisk;
\item[2)] $\overline{\Omega_i} \cap \partial \Sigma=I_i$ is an arc; 
\item[3)] $\eta(I_1) \subset I_1 \cup I_2$ and $\xi(I_2) \subset I_1 \cup I_2$; 
\item[4)] $\overline{\Omega_1} \cap \overline{\Omega_2}=\{ c\}$, a single point.
\end{itemize}

We refer the reader to \cite{McM} and \cite{Ya3} for a detailed discussion of McMullen holomorphic pairs.
 
Let  $\hat\zeta$ denote the McMullen renormalization fixed point constructed in \cite{McM}.  As before, let $$P_{\theta_*}(z)=z^2+e^{2\pi i{\theta_*}}z,\text{ where }\theta_*=(\sqrt{5}-1)/2,$$ and let $\zeta_1=(P^2_{\theta_*}|_{X},P_{\theta_*}|_Y)$ be the McMullen holomorphic pair corresponding to the first pre-renormalization of $P_{\theta_*}$. As was shown in \cite{McM}, $$\cR^n\zeta_1\to\hat\zeta$$ geometrically fast.

\begin{figure}[htb]
\centerline{\includegraphics[width=0.6\textwidth]{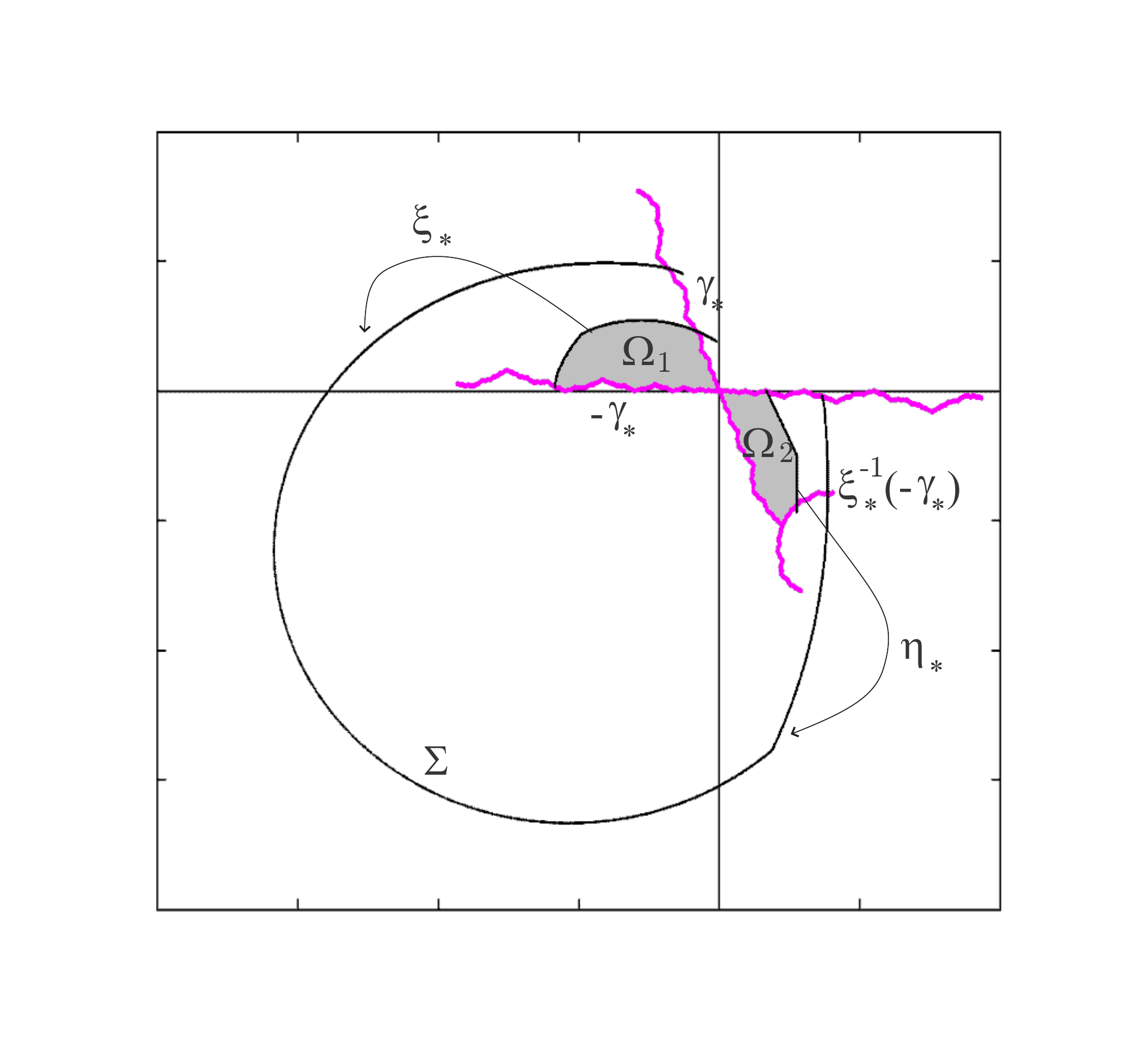}}
\caption{\label{fig-mcmpair}The domains of the McMullen holomorphic pair extension of $\zeta_*$.}
\end{figure}

We use rigorous computer-assisted estimates to prove the following:
\begin{prop}
\label{prop:mcm-extension}
There exists a neighborhood $W \subset C(Z_1,W_1)$, such that for all $\zeta \in W^s(\zeta_*) \cap W$, the pair $\cP^3 \zeta$ has  an extension to a McMullen holomorphic pair 
$$\tl \zeta=(\tl \xi:\Omega_1\to\Sigma, \tl \eta:\Omega_2\to\Sigma).$$
\end{prop}
We use rigorous computer-assisted techniques to construct the domains $\Sigma$, $\Omega_1$, $\Omega_2$ (see Fig. \ref{fig-mcmpair}); the Jordanness of these domains follows from \thmref{mainthm4}.

Following \cite{McM}, define the filled Julia set of a McMullen holomorphic pair as 
$$K(\tl \zeta)=\cap_{n>0} \tl \zeta^{-n} (\Omega_1 \cup \Omega_2).$$

Given a subset $\Lambda \subset \field{C}$, we say that $z$ is a {\it measurable $\delta$-deep point} of $\Lambda$ if for some $\delta>0$, 
$${\rm area}(B_r(z)-\Lambda)=O(r^{2+\delta}).$$

Combining \propref{prop:mcm-extension} with a standard pull-back argument as well as the results of \cite{McM}, we have:
\begin{thm}
\label{mainthm5}
For every $\zeta \in W^s(\zeta_*) \cap W$  the McMullen holomorphic pair extension  $\tl \zeta$ of $\cP^3 \zeta$ is quasiconformally conjugate to $\hat \zeta$, this conjugacy is conformal on ${\rm int}(K(\tl \zeta))$.
\end{thm}

As a consequence of the above theorem, we have the following.
\begin{cor}
\label{cor1}
The critical point $0$ is a measurable deep point if $K(\tl \zeta)$.
\end{cor}
\begin{proof}
By the results of \cite{McM}, $0$ is a measurable deep point of $K(\hat \zeta)$, and this property is preserved under a quasiconformal conjugacy. 
\end{proof}

We say that a quasiconformal map $\phi$ of subset of $\field{C}$ is $C^{1+\alpha}$-conformal at point $z$ if $\phi'(z)$ exists and 
$$\phi(z+t)=\phi(z)+\phi'(z) t+O(|t|^{1+\alpha}).$$

\begin{cor}
Let $h$ be the conjugacy from \thmref{mainthm5}. Then $h$ is $C^{1+\alpha}$-conformal at $0$.
\end{cor}
\begin{proof}
By Corollary $\ref{cor1}$ point $0$ is the measurable $\delta$-deep point of the measurable set ${\rm int}(K(\tl \zeta)$, while by \thmref{mainthm5} $h$ is a $K$-quasiconformal map, conformal on ${\rm int}(K(\tl \zeta)$.

By the Boundary Conformality Theorem (see Theorem 2.25 in \cite{McMBook}), there exists $\alpha=\alpha(K,\delta)$ such that $h$ is $C^{1+\alpha}$-conformal.
\end{proof}

\begin{cor}
The maps of the holomorphic pairs $\zeta_*$ and $\hat\zeta$ are analytic continuations of each other.
$$\cR^n\zeta_1\to\zeta_*,$$ geometrically fast in the uniform metric, where defined.
\end{cor}

%% file: dissipative.tex
\section{Renormalization for  pairs of two-dimensional dissipative maps}\label{sec:RenACM}
Let $\Omega, \Gamma$ be  domains in $\CC^2$. We denote $O(\Omega,\Gamma)$ the Banach space of pairs of bounded analytic functions 
$F=(F_1(x,y),F_2(x,y))$  from $\Omega$ and $\Gamma$ respectively to $\CC^2$ equipped with the norm
\begin{equation}
\label{eq:unormm}\| F\|= \frac{1}{2} \left(   \sup_{(x,y) \in \Omega}|F_1(x,y)|+  \sup_{(x,y)\in \Gamma} |F_2(x,y) | \right).
\end{equation}
We let $O(\Omega,\Gamma,\delta)$ stand for the $\delta$-ball around the origin in this Banach space. 

In what follows, we fix $W_1$, $Z_1$, $\tl Z$, and $\tl W$ as in \thmref{mainthm3}, and $R>0$ such that $\DD_R\subset Z_1\cap W_1$, and let 
$\Omega=Z_1 \times \DD_R$, $\Gamma=W_1 \times \DD_R$.
We select $\hat Z_1$ and $\hat W_1$ so that
$$Z_1 \Subset \hat Z_1\Subset \tl Z,\; W_1\Subset \hat W_1\Subset \tl W.$$
We define an isometric embedding $\iota$ of the space $\cH(Z_1,W_1)$ into $O(\Omega,\Gamma)$ which send the pair $\zeta=(\eta,\xi)$ to
the  pair of functions $\iota(\zeta)$:
\begin{equation}
\label{eq:embed}
\left(\left(x \atop  y\right)\mapsto \left(\eta(x) \atop \eta(x) \right), 
\left(x \atop  y\right)\mapsto \left(\xi(x) \atop \xi(x) \right)   \right).
\end{equation}

Let $\cU$ be an open neighborhood of $\zeta_*$ in $\cC(Z_1,W_1)$, and let $Q$ be a neighborhood of $0$ in $\CC$.
We will consider an open subset of  $O(\Omega,\Gamma)$ of pairs of maps of the form
\begin{eqnarray}
\label{eq:A} A(x,y)&=&(a(x,y),h(x,y))=(a_y(x),h_y(x)),\\
\label{eq:B} B(x,y)&=&(b(x,y),g(x,y))=(b_y(x),g_y(x)), 
\end{eqnarray}
such that
\begin{itemize}
\item[1)] the pair $(a(x,y),b(x,y))$ is in a $\delta$-neighborhood  of  $\cU$ in $\cH(Z_1,W_1)$,
\item[2)]  $(h,g) \in O(\Omega,\Gamma)$ are such that  $|\partial_x h(x,0)|>0$ and  $|\partial_x g(x,0)|>0$ whenever $x \notin \bar Q$,  and 
$$(h(x,y)-h(x,0),g(x,y)-g(x,0)) \in \cO(\Omega,\Gamma,\delta).$$ 
\end{itemize}
This open subset of $O(\Omega,\Gamma)$ will be denoted $\cB(\cU,Q,\delta)$ for brevity.

Given a pair $\Sigma=(A,B)$ as in $(\ref{eq:A})-(\ref{eq:B})$ we set 
$$\cL \Sigma \equiv (a(x,0),b(x,0)).$$

\subsection{The first step in defining renormalization: appropriate coordinate transformations}\label{transformations}
We will now define a renormalization operator on pairs of 2D maps in several steps. Similarly to \cite{CLM}, the first 
step will be to define  a sufficiently high pre-renormalization in a neighborhood of $(\eta^{-1}(0),0)$, and then pull such a 
pair back to neighborhood of $(0,0)$ by a coordinate change which will reduce the order of magnitude of the
 dependence of our pair on the $y$-coordinate.

Let $(\eta,\xi) \in \cB(Z_1,W_1)$ be $n\geq 2$ times renormalizable,  and consider its $n$-th renormalization written as
$$\cR^n\zeta=\lambda_n^{-1}\circ \left(\eta \circ \xi \circ \zeta^{\bar l_n}, \eta \circ \xi \circ \zeta^{\bar m_n } \right) \circ \lambda_n,$$
where $\lambda_n$ is the appropriate (anti-)linear transformation. 

 For a  sufficiently large $n$, the function $\eta^{-1}$ is a diffeomorphism of the neighborhood $\lambda_n (Z_1 \cup W_1)$, and one can define the $n$-th pre-renormalization of $\zeta$ on $\eta^{-1}(\lambda_n (Z_1 \cup W_1))$ as 
$$\hat p \cR^n\zeta=\left(\xi \circ \zeta^{\bar l_n} \circ \eta, \xi \circ \zeta^{\bar m_n } \circ \eta \right).$$

Next, suppose that $\Sigma=(A,B)$ lies in $\cB(\cU,Q,\delta)$ with $\cU$ and $\delta$ sufficiently small, so that 
the following pre-renormalization  is defined in a neighborhood of $\eta^{-1}(\lambda_n (Z_1 \cup W_1))\times \{ 0\}$:
$$\hat p \cR^n \Sigma = \left(B \circ \Sigma^{\bar l_n} \circ A, B \circ \Sigma^{\bar m_n} \circ A \right).$$

We will denote $$\pi_1(x,y)=x\text{ and }\pi_2(x,y)=y.$$
Set
$$\phi_y(x) := \pi_1 \circ A \circ B(x,y) \equiv a(b(x,y),g(x,y)).$$
For sufficiently small $Q$ and $\delta$ and for all  $z \in \DD_r$ for some large $r$, the map $\phi_z$ is close to $\eta \circ \xi$ and is a diffeomorphism of a neighborhood of $\pi_1 \Sigma^{\bar l_n}(\lambda_n (Z_1),0) \approx \zeta^{\bar l_n}(\lambda_n (Z_1))$.  Similarly,  for all $z \in \DD_r$ for some large $r$, the map $g_z$ is a diffeomorphism of  a neighborhood of $\pi_1  \Sigma^{\bar l_n}(\lambda_n (Z_1),0)$. We can chose the constants to guarantee that $r>2R$.

Furthermore, let
$$w_z(x)\equiv w(x,z) := g_{z}\left(\phi_{z}^{-1}( x) \right).$$
This is a diffeomorphism of a neighborhood of   $\pi_1 A \circ B \circ \Sigma^{\bar l_n}(\lambda_n (Z_1),0)$ in $\CC^2$ onto its image for all $z \in \DD_R$ for some large $R>0$.  Notice, that $\partial_z  w_z(x)$ and $\partial_z w^{-1}_z(x)$ are functions whose uniform norms are $O(\delta)$.

We are now ready to define our first change of coordinates:
$$H(x,y) =  (a_y(x),w_{g_0^{-1}(y)}^{-1}(y)),$$
where the function $g^{-1}_z$ is defined by 
\begin{equation}\label{g_inv}
g_z^{-1}(g(x,z))=x
\end{equation}
outside a neighborhood of $g(Q,0)$ in  $\CC$. The transformation $H$ is close to $\left(\eta(x),\eta(\xi(g_0^{-1}(y))) \right)$, and therefore, for small $\delta$,  is a diffeomorphism of a neighborhood of $$\pi_1  B \circ \Sigma^{\bar l_n}(\lambda_n (Z_1),0) \approx \xi(\zeta^{\bar l_n}(\lambda_n (Z_1)))$$ onto its image. Notice that
\begin{eqnarray}
\label{correction} w_{g_0^{-1}(y)}^{-1}(y)-w_{g_z^{-1}(y)}^{-1}(y) \sim \left( \partial_z w_{g_z^{-1}(y)}^{-1}(y)  \right) \left(\partial_z g_z^{-1}(y)  \right) z&=&O(\delta^2) O(z),\\
\label{comp} w_{z_1}^{-1} \circ w_{z_2} -id \sim \partial_z w_z (z_2-z_1)&=&O(\delta) O(z_2-z_1).
\end{eqnarray}

 We use $H(x,y)$ to pull back $\hat p \cR^n\Sigma$ to a neighborhood of definition of the $n$-th pre-renormalization of an almost commuting pair $(\eta,\xi)$ - that is, a neighborhood of  $\lambda_n (Z_1 \cup W_1)$ in $\CC^2$:
$$p \cR^n\Sigma\equiv(\bar A, \bar B) = H \circ \left( B \circ \Sigma^{\bar l_n} \circ A, B \circ \Sigma^{\bar m_n} \circ  A \right) \circ H^{-1}(x,y).$$

The first map in this pair becomes
\begin{eqnarray}
\nonumber  \bar A(x,y)&=&H \circ B \circ \Sigma^{\bar l_n} \circ A \circ H^{-1}(x,y) \\
\nonumber &=&\left( { \pi_1 A \circ B \circ \Sigma^{\bar l_n}  \atop    w_{g_0^{-1}( \pi_2 B \circ \Sigma^{\bar l_n}  )}^{-1}  \circ \pi_2 B \circ \Sigma^{\bar l_n} } \right) \circ A \circ H^{-1}(x,y) \\
\nonumber &=&\left( { \pi_1 A \circ B \circ \Sigma^{\bar l_n}  \atop    w_{(g_0^{-1} \circ g) \circ \Sigma^{\bar l_n}  }^{-1}  \circ \pi_2 B \circ \Sigma^{\bar l_n} } \right) \circ A \circ H^{-1}(x,y) \\
\nonumber &{(\ref{correction}) \atop =}&\left( { \pi_1 A \circ B \circ \Sigma^{\bar l_n}  \atop    w_{ {(g_{ \pi_2 \Sigma^{\bar l_n}  }^{-1} \circ g) \circ \Sigma^{\bar l_n}  }  }^{-1}  \circ \pi_2 B \circ \Sigma^{\bar l_n} +O(\delta^2)} \right) \circ A \circ H^{-1}(x,y) \\
\nonumber & {(\ref{g_inv}) \atop =}& \left( { \pi_1 A \circ B \circ \Sigma^{\bar l_n}  \atop   w_{\pi_1 \Sigma^{\bar l_n} }^{-1}   \circ g_{\pi_2 \Sigma^{\bar l_n} } \circ \pi_1 \Sigma^{\bar l_n} +O(\delta^2)} \right) \circ A \circ H^{-1}(x,y) \\
\nonumber  &=&\left( { \pi_1 A \circ B \circ \Sigma^{\bar l_n}   \atop     w_{\pi_1 \Sigma^{\bar l_n} }^{-1}  \circ w_{\pi_2 \Sigma^{\bar l_n} }\circ \phi_{\pi_2 \Sigma^{\bar l_n}} \circ \pi_1 \Sigma^{\bar l_n} +O(\delta^2)} \right) \circ A \circ H^{-1}(x,y) \\
\nonumber  &{(\ref{comp}) \atop =}&\left( { \pi_1 A \circ B \circ \Sigma^{\bar l_n}   \atop  \pi_1 A \circ B \circ \Sigma^{\bar l_n} +O(\delta) O\left({\pi_1 \Sigma^{\bar l_n}} -{\pi_2 \Sigma^{\bar l_n} }\right)+O(\delta^2)} \right) \circ A \circ H^{-1}(x,y).
\end{eqnarray}

\noindent
 We remark that the inverse $H^{-1}(x,y)=H^{-1}_{appr}(x,y)+O(\delta^2)$, where
$$H^{-1}_{appr}(x,y):=\left(a^{-1}_{w_{\phi_0^{-1}(y)}(y)}(x), w_{\phi_0^{-1}(y)}(y) \right).$$
Indeed, set  $\chi(y):=w_{\phi_0^{-1}\left(w^{-1}_{g_0^{-1}(y)}(y) \right)} \left(w^{-1}_{g_0^{-1}(y)}(y)\right)$, then we can use the fact that $\phi_0^{-1}=g_0^{-1} \circ w_0$, together with (\ref{comp}),  in the following computation:
\begin{eqnarray}
\nonumber \chi(y)&=&w_{g_0^{-1} \circ \left(w_0 \circ w^{-1}_{g_0^{-1}(y)}\right)(y)} \left(w^{-1}_{g_0^{-1}(y)}(y)\right)=w_{g_0^{-1}\left(y+O(\delta)\right)} \left(w^{-1}_{g_0^{-1}(y)}(y)\right) \\
\nonumber &=&w_{g_0^{-1}(y)+O(\delta)} \left(w^{-1}_{g_0^{-1}(y)}(y)\right)=y+O(\delta^2),
\end{eqnarray}
therefore,
\begin{eqnarray}
\nonumber H_{appr}^{-1} \circ H(x,y)&=& \left(a^{-1}_{w_{\phi_0^{-1} \circ \pi_2} \circ \pi_2 } \circ \pi_1, w_{\phi_0^{-1} \circ \pi_2} \circ \pi_2 \right) \circ (a_y(x),w_{g_0^{-1}(y)}^{-1}(y)) \\
\nonumber &=&\left(a^{-1}_{\chi(y)}(a_y(x)), \chi(y)\right) \\
\nonumber &=&(x,y)+O(\delta^2).
\end{eqnarray}
We now have
$$A \circ H^{-1}(x,y)=(a,h) \circ \left(a^{-1}_{w_{\phi_0^{-1}(y)}(y)}(x), w_{\phi_0^{-1}(y)}(y) \right) +O(\delta^2)=(x,h(H^{-1}_{appr}(x,y)))+O(\delta^2),$$
\noindent
and
$$ \bar{A}(x,y)= \left( { \pi_1 A \circ B \circ \Sigma^{\bar l_n} \atop  \pi_1 A \circ B \circ \Sigma^{\bar l_n} + O(\delta) \ O\!\left({\pi_1 \Sigma^{\bar l_n}} -{\pi_2 \Sigma^{\bar l_n} }\right) } \right)(x,h(H^{-1}_{appr}(x,y)))+O(\delta^2).$$



Similarly,
\begin{eqnarray}
\nonumber \bar{B}(x,y)&=& H \circ B \circ \Sigma^{\bar m_n}\circ A \circ H^{-1}(x,y) \\
\nonumber &=&\left( { \pi_1 A \circ B \circ \Sigma^{\bar m_n} \atop  \pi_1 A \circ B \circ \Sigma^{\bar m_n} + O(\delta) \ O\!\left({\pi_1 \Sigma^{\bar m_n}} -{\pi_2 \Sigma^{\bar m_n} }\right)} \right)(x,h(H^{-1}_{appr}(x,y)))+O(\delta^2).
\end{eqnarray}

\noindent
We have thus proved the following:
\begin{lem}
\label{lem-preren}
There exists an $n \in \NN$, and a choice of $\cU$, $Q$, and $\delta_0$ such that there exists $C>0$ such that the following holds. For every $\delta<\delta_0$ and every $\Sigma\in\cB(\cU,Q,\delta)$ the pair $p\cR^n\Sigma$ is defined, lies in $O(\hat\Omega,\hat\Gamma)$, and 
$$\dist(p \cR^n \Sigma, \iota(\cH(\lambda_n(\hat Z_1),\lambda_n(\hat W_1))))<C \delta \left( \|\pi_1 \Sigma -\pi_2 \Sigma \|  +\delta \right).$$
\end{lem}

Let us write
\begin{equation}
\bar{A}(x,y)=\left( { \bar \eta_1(x)+\bar \tau_1(x,y) \atop   \bar \eta_2(x)+\bar \tau_2(x,y) } \right),
\end{equation}
where 
$$\bar \eta_1(x) \equiv \pi_1 \bar{A}(x,0), \quad  \bar \eta_2(x) \equiv \pi_2 \bar{A}(x,0)$$
are $O\left( \delta \left( \|\pi_1 \Sigma -\pi_2 \Sigma \|  +\delta \right) \right)$-close to each other, and both are $\delta$-close to $\pi_\eta p \cR^n \zeta=\eta \circ \xi \circ \zeta^{\bar l_n}$, and 
$$\bar \tau_1(x,y) \equiv  \pi_1 \bar{A}(x,y)- \pi_1 \bar{A}(x,0), \quad \bar \tau_2(x,y)= \pi_2 \bar{A}(x,y)- \pi_2 \bar{A}(x,0),$$
are functions whose norms are $O(\delta^2)$. Similarly,
$$
\nonumber \bar{B}(x,y)=\left( { \bar \xi_1(x)+\bar \pi_1x,y) \atop   \bar \xi_2(x)+\bar \pi_2(x,y) } \right),
$$
where
$$\bar \xi_1(x) \equiv \pi_1 \bar{B}(x,0), \quad  \bar \xi_2(x) \equiv \pi_2 \bar{B}(x,0)$$
are $O\left( \delta \left( \|\pi_1 \Sigma -\pi_2 \Sigma \|  +\delta \right) \right)$-close to each other, and both are $\delta$-close to $\pi_\xi p \cR^n \zeta=\eta \circ \xi \circ \zeta^{\bar m_n}$, and 
$$\bar \pi_1(x,y) \equiv  \pi_1 \bar{B}(x,y)- \pi_1 \bar{B}(x,0), \quad \bar \pi_2(x,y)= \pi_2 \bar{B}(x,y)- \pi_2 \bar{B}(x,0),$$
are functions whose norms are $O(\delta^2)$.

\subsection{Critical projection} By the Argument Principle, 
if $\delta$ is sufficiently small, then 
the function $\pi_1 \bar B \circ \bar A (x,0)$ has a unique critical point $c_1$ in a neighborhood of $0$. Set $T_1(x,y)=(x+c_1,y)$, then
$$\partial_x \left(\pi_1 T_1^{-1} \circ \bar B \circ \bar A \circ T_1  \right)(0,0)=0.$$ Similarly, if $\delta$ is sufficiently small, the function $\pi_1 T_1^{-1} \circ \bar A \circ \bar B \circ T_1(x,0)$ has a unique critical point $c_2$ in a neighborhood of $0$.  Set $T_2(x,y)=(x+c_2,y)$, then
$$\partial_x \left(\pi_1 T_2^{-1} \circ T_1^{-1} \circ \bar A \circ \bar B \circ T_1  \circ T_2 \right)(0,0)=0.$$

We now set
\begin{eqnarray}
\nonumber \Pi_1(\bar A, \bar B)&=&(\tilde A, \tilde B):=(T_2^{-1} \circ T_1^{-1} \circ \bar A \circ T_1, T^{-1}_1 \circ \bar B \circ T_1 \circ T_2)\\
\nonumber &=& \left( \left( \tilde \eta_1(x) + \tilde \tau_1(x,y) \atop   \tilde \eta_2(x) + \tilde \tau_2(x,y)  \right), \left( \tilde \xi_1(x) + \tilde \pi_1(x,y) \atop   \tilde \xi_2(x) + \tilde \pi_2(x,y)  \right)   \right),
\end{eqnarray}
where the norms of the functions $\tilde \tau_k$, $\tilde \pi_k$, $k=1,2$, are $O(\delta^2)$. 

According to the discussion at the end of Section \ref{transformations}, the critical points  of the functions $\pi_1 (\bar A \circ \bar B)(x,0)$ and $\pi_1 (\bar B \circ \bar A)(x,0)$ are $O\left( \delta \left( \|\pi_1 \Sigma -\pi_2 \Sigma \|  +\delta \right) \right)$-close to each other, and therefore, 
\begin{equation}
\label{T2small} T_2=\text{Id}+O\left( \delta \left( \|\pi_1 \Sigma -\pi_2 \Sigma \|  +\delta \right) \right).
\end{equation}
At the same time, since $\pi_1 (\bar A \circ \bar B)(x,0)$  is $\delta$-close to the map $(\pi_\xi p \cR^n \zeta) \circ (\pi_\eta p \cR^n \zeta)$, $T_1=\text{Id}+O\left( \delta \right)$. Furthermore,  
\begin{eqnarray}
\nonumber 0&=&\partial_x (\pi_1 \tilde A \circ \tilde B)(0,0)\\
\nonumber &=&\left( \partial_x (\pi_1 \tilde A) \circ \tilde B(0,0)\right)  \partial_x (\pi_1 \tilde B)(0,0)+\left( \partial_y (\pi_1 \tilde A) \circ \tilde B(0,0) \right) \partial_x (\pi_2 \tilde B)(0,0) \\
\nonumber &=&O(1) \hspace{0.5mm} \partial_x (\pi_1 \tilde B)(0,0) + O(\delta^2) \implies  \partial_x (\pi_1 \tilde B)(0,0)=O(\delta^2),
\end{eqnarray}
and similarly for $ \partial_x (\pi_1 \tilde A)(0,0)$. We, therefore, have the following
\begin{lem}
\label{lem-proj1}
There exists an $n \in \NN$, a neighborhood $\cB(\cU,Q,\delta)$ of $\zeta_*$, and $C>0$, such that:
\begin{itemize}
\item for every $\Sigma \in \cB(\cU,Q,\delta)$ there exists $p\cR^n\Sigma\in O(\hat\Omega,\hat\Gamma)$;
\item the projection $\Pi_1 p\cR^n$ is a  well-defined and analytic operator  $$\cB(\cU,Q,\delta)\to O(\hat\Omega,\hat\Gamma);$$
\item for every $\Sigma \in \cB(\cU,Q,\delta)$,
$$\dist(\Pi_1 p \cR^n \Sigma, \iota(\cC(\lambda_n(\hat Z_1),\lambda_n(\hat W_1))))<C \delta \left( \|\pi_1 \Sigma -\pi_2 \Sigma \|  +\delta \right)$$;
\item for every $\Sigma \in \iota(\cC(\lambda_n(\hat Z_1),\lambda_n(\hat W_1)))$, the composition
$$\Pi_1 p \cR^n \Sigma =p \cR^n \Sigma.$$
\item for every  pair $\Sigma=(A,B)$ such that $A\circ B=B\circ A$ we have
$$\Pi_1 p \cR^n \Sigma =T_1^{-1} \circ  p \cR^n \Sigma \circ T_1.$$
\end{itemize}
\end{lem}

\subsection{An almost commutation condition, and the second projection}
Let us set 
$$\tl\Sigma=(\tl A,\tl B)=\Pi_1 p \cRG^n\Sigma.$$
At the next step we will project the pair $(\tilde A, \tilde B)$ onto the subset of pairs satisfying the following almost commutation conditions:
\begin{eqnarray} \label{commutation}
\partial_x^i \pi_1(\tilde A \circ \tilde B(x,0) - \tilde B \circ \tilde A(x,0))\arrowvert_{x=0}&=&0, \quad i=0,2 \\
\label{normalization} \pi_1 \tilde B(0,0)&=&1.
\end{eqnarray}
To that end we set 
$$\Pi_2 (\tilde A,\tilde B)(x,y)=\left( \left( \tilde \eta_1(x) +a x^4+b x^6 + \tilde \tau_1(x,y) \atop   \tilde \eta_2(x) +a x^4+b x^6+ \tilde \tau_2(x,y)  \right), \left( \tilde \xi_1(x)+c  + \tilde \pi_1(x,y) \atop   \tilde \xi_2(x) +c +\tilde \pi_2(x,y)  \right)   \right),$$
and require that $(\ref{commutation})$ and $(\ref{normalization})$ are satisfied for maps in the pair $\Pi_2 (\tilde A,\tilde B)(x,y)$. The following Proposition is proved below.
\begin{prop}\label{prop:2Dprojection}
There exists $\rho>0$ such that for all $\tl\Sigma$ in the $\rho$-neighborhood of $$\iota(\cC(\lambda_n(\hat Z_1),\lambda_n(\hat W_1)))$$
there is a unique tuple $(a,b,c,d)$ such that the pair $ \Pi_2 (\tilde A,\tilde B)$ satisfies the equations $(\ref{commutation})$ and $(\ref{normalization})$. Moreover, in this neighborhood, the dependence of $\Pi_2$ on $\Sigma$ is analytic.  
Furthermore, if $A\circ B=B\circ A$, then 
$\Pi_2=\text{Id}.$

Finally, if $$\Sigma=\left( \zeta \atop \zeta \right) \equiv \left(\left(\eta \atop \eta  \right), \left(\xi \atop \xi  \right)  \right) \in \iota(\cC(Z,W)), $$ then 
$$\Pi_2 (\tl A,\tl B)=\left(\cP \hspace{0.5mm} p \cR^n\zeta \atop \cP \hspace{0.5mm} p \cR^n \zeta  \right),$$
where $\cP$ has been defined in (\ref{Pproj}). 
\end{prop}

\begin{proof}
Consider the following system of $3$ equations $\bF(a,b,c)=0$:
\begin{eqnarray}
\nonumber a&+&b-c-\\
\nonumber &-&\left(\tilde \eta _1(\tilde \xi_1(0))-\tilde \eta_1(\tilde \xi_1(0)+c)+\tilde \tau_1(\tilde \xi_1(0), \tilde \xi_2(0))-\tilde \tau_1(\tilde \xi_1(0)+c, \tilde \xi_2(0)+c) \right)\\
\nonumber &=& \pi_1( \tilde B \circ \tilde A(0,0)- \tilde A \circ \tilde B(0,0))\\
\nonumber \tilde \xi_1'(0)^2 (12 a &+&30 b)+\tilde \xi_1''(0) (4 a +6 b)+\\
\nonumber &+&\tilde \eta_k''(\tilde \xi_1(0)+c) \tilde \xi_1'(0)^2+\tilde \eta_k'(\tilde \xi_1(0)+c) \tilde \xi_1''(0)-\tilde \eta_1''(\tilde \xi_1(0)) \tilde \xi_1'(0)^2-\tilde \eta_1'(\tilde \xi_1(0)) \tilde \xi_1''(0)-\\
\nonumber&+&\sum_{i,j=1,2}\partial_{i,j}\tilde \tau_1(\tilde \xi_1(0)+c,\tilde \xi_2(0)+c) \tilde \xi_i'(0) \tilde \xi_j'(0)+\\
\nonumber &+& \nabla \tilde \tau_1(\tilde \xi_1(0)+c,\tilde \xi_2(0)+c) \cdot (\tilde \xi_1''(0), \tilde \xi_2''(0))-\\
\nonumber&-&\sum_{i,j=1,2}\partial_{i,j}\tilde \tau_1(\tilde \xi_1(0),\tilde \xi_2(0)) \tilde (\xi_i'(0)) (\tilde \xi_j'(0))+\\
\nonumber &-& \nabla \tilde \tau_1(\tilde \xi_1(0),\tilde \xi_2(0)) \cdot (\tilde \xi_1''(0), \tilde \xi_2''(0))\\
\nonumber & =& \pi_1( \tilde B \circ \tilde A(x,0)- \tilde A \circ \tilde B(x,0))''\arrowvert_{x=0} \\
\nonumber c&=&1-\tilde \xi_1(0).
\end{eqnarray}
Notice, that when the  equations $(\ref{commutation})$ and $(\ref{normalization})$ are satisfied this system of equations is solved by $a=b=c=0$. Furthermore, denote $\mathbf{p}=(a,b,c)$, then the derivative $D_{\mathbf p} \bF(\mathbf 0)$ is given by
$$
\left[
\begin{array}{c c c }
\vspace{4mm}
1 & 1 & -1+\tilde n _1'(\tilde \xi_1(0)) +\rho_1 \\
\vspace{4mm}
12 \vareps_1^2 + 4 \tilde \xi_1''(0) & 30 \vareps_1^2 +6 \tilde \xi_1''(0) & {\tilde \eta_1'''( \tilde \xi_1(0) ) \tilde \xi_1'(0)+  \tilde \eta_1''( \tilde \xi_1(0) ) \tilde \xi_1''(0)\atop +\sum_{i,j=1,2}\rho_5^{i,j} \vareps_i \vareps_j +\rho_6}\\ 
\vspace{4mm}
0 & 0  & 1
\end{array}
\right],
$$
where $\vareps_i=\tilde \eta_i'(0)$, $i=1,2$, and the following are the functions with the $O(\rho^2$) norm:
\begin{eqnarray}
\nonumber  \rho_1&=&\partial_1 \tilde \tau_1(\tilde \xi_1(0), \tilde \xi_2(0))+\partial_2 \tilde \tau_1(\tilde \xi_1(0), \tilde \xi_2(0)),\\
\nonumber \rho_5^{i,j}&=&\partial_{1,i,j}\tilde \tau_1(\tilde \xi_1(0),\tilde \xi_2(0))+\partial_{2,i,j}\tilde \tau_1(\tilde \xi_1(0),\tilde \xi_2(0))\\
\nonumber \rho_6&=&\nabla \partial_1 \tilde \tau_1(\tilde \xi_1(0),\tilde \xi_2(0)) \cdot (\tilde \xi_1''(0), \tilde \xi_2''(0))+\nabla \partial_2 \tilde \tau_1(\tilde \xi_1(0),\tilde \xi_2(0)) \cdot (\tilde \xi_1''(0), \tilde \xi_2''(0))
\end{eqnarray}

If $(A,B)$ is in a $\rho$-neighborhood of  $\iota(\cC(\lambda_n(\hat Z_1),\lambda_n(\hat W_1)))$, then the determinant of the above matrix is $\rho$-close to $2 \tilde \xi_1''(0)$ and is nonzero. By the Regular Value Theorem, there exists $\rho>0$ and for each such $(A,B)$, a unique triple $(a,b,c)$ such that  the  equations $(\ref{commutation})$ and $(\ref{normalization})$ are satisfied.
\end{proof}

\noindent
Let us fix $n\in 2\NN$, $\cU$, $Q$, $\delta$ so that Lemma~\ref{lem-preren} and Lemma~\ref{lem-proj1} hold, and furthermore, the image 
$\Pi_1p\cR^n\cB(\cU,Q,\delta)$ lies in the $\rho$-neighborhood of $\iota(\cC(\lambda_n(\hat Z_1),\lambda_n(\hat W_1)))$ as in Proposition~\ref{prop:2Dprojection}. We then have:

\begin{prop}
\label{small-pert2}
For every $\Sigma \in \cB(\cU,Q,\delta)$,
$${\rm dist}(\Pi_2\Pi_1 p \cR^n \Sigma, \iota(\cB(\lambda_n(\hat Z_1),\lambda_n(\hat W_1)))) < C \delta \left( \|\pi_1 \Sigma -\pi_2 \Sigma \|  +\delta \right).$$
\end{prop}
\begin{proof}
This follows immediately from Lemma~\ref{lem-preren} and the fact that $\Pi_2\Pi_1$ is an analytic operator, which is equal to the identity on $\iota(\cB(\lambda_n(\hat Z_1),\lambda_n(\hat W_1)))$.
\end{proof}

 Let 
$\ell_n=\pi_1\bar B(0,0)$ and  $L_n(x,y)=(\ell_n x, \ell_n y)$. 
\begin{defn}
 We define {\it the $n$-th renormalization } of a pair $\Sigma \in \cB(\cU,Q,\delta)$ as
\begin{equation}\label{2Drenorm}
\cRG^n \Sigma = L_n^{-1} \circ\Pi_2\circ\Pi_1\circ p\cR^n \Sigma \circ L_n.
\end{equation}
\end{defn}

\subsection{Hyperbolicity of renormalization of 2D dissipative maps}
Let $n$ be an even number as above. 
We conclude this section by formulating the following theorem:
\begin{thm}
\label{mainthm6}
The point $\iota(\zeta_*)$ is a fixed point of $\cRG^n$ in $O(\Omega,\Gamma)$. The linear operator
$N=D \cRG^n \arrowvert_{\iota(\zeta_*)}$ is compact. The spectrum of $N$ coincides with the spectrum of $M^n$, where $M$ is as in Main Theorem 3. More
specifically, $\lambda\neq 0$ is an eigenvalue of $M^n$, and $\bar v$ is a corresponding eigenvector if and only if $\lambda$ is an eigenvalue of $N$, and $D\iota(\bar v)$ is a corresponding eigenvector.

\end{thm}
\begin{proof}
Notice, that according to  Proposition~\ref{small-pert2}
$${\rm dist}(\left( \cRG^n\right)^2 \Sigma, \iota(\cB(l_n^{-2} \lambda_n^2(\hat Z_1),l_n^{-2} \lambda_n^2(\hat W_1)))) < C \delta^2.$$
The claim now follows follows since
$$\iota\circ \cRG^n=\cRG^n\circ\iota. $$
\end{proof}

%% file: comp_proof.tex
\section{A computer-assisted proof of renormalization hyperbolicity}
\label{sec:pfmain1}

\subsection{Contraction Mapping Principle}

Our proof of the existence of a fixed point $\zeta^*=(\eta_*,\xi_*)$ is a computer implementation of the Contraction Mapping Principle for the Newton map for the operator $\cRG$ in the space  $\cA_1(U,V)$.

Let $\cF$ be an operator analytic in some neighborhood $\cN$ in a Banach space $\cZ$, which has a fixed point $\zeta_*$ in this space,
 and whose differential $D|_{\zeta_*} \cF$ is a hyperbolic linear map in $\cN$.  Suppose, that one knows an approximation
of the fixed point  $\zeta_0 \approx \zeta_*$.
 Set
$$M \equiv \left[ \field{I}-L_0 \right]^{-1},$$
where $L_0$ is an approximation of $D \cF[\zeta_0]$. For all $\zeta$, such that $\zeta_0+M \zeta \in \cN$, set
$$N_{\cF,L_0}[\zeta]=\zeta+\cF[\zeta_0+M \zeta]-(\zeta_0+M \zeta).$$

Notice, that if $\hat\zeta$ is a fixed point of $N_{\cF,L_0}$, then $\zeta^*=\zeta_0+M \hat\zeta$ is a fixed point of $\cF$.

The operator $\field{I}-L_0$ is indeed invertible if $L_0$ is close to $D \cF[\zeta_0]$, since $D \cF$ is hyperbolic at $\zeta_0$. If $\zeta_0$ is a reasonably good approximation of the true fixed point of $\cF$, then the operator $N_{\cF,L_0}$ is expected to be a strong contraction in a neighborhood of $0$:
\begin{eqnarray}
\nonumber D N_{\cF,L_0}[\zeta] &=& \field{I} +D \cF[\zeta_0+M \zeta] \cdot M -M= \left[ M^{-1} +D \cF[\zeta_0+M \zeta]  -\field{I} \right] \cdot M \\
\nonumber &=& \left[\field{I}-L_0  +D \cF [\zeta_0+M \zeta]  -\field{I} \right] \cdot M = \left[ D \cF [\zeta_0+M \zeta] - L_0 \right] \cdot M.
\end{eqnarray}

The last expression is typically small in a small neighborhood of $0$, if the norm of $M$ is not too large (if $M$ is large, one might have to find a better approximation $\zeta_0$ and take a smaller neighborhood of $0$).  The following well-known version of the contraction mapping theorem specifies a sufficient condition for the existence of a fixed point $\hat\zeta$:

\begin{cmp}
Let $\cN$ be an open subset of a Banach space $\cZ$.
Suppose that  $N_{\cF,L_0}$ is a well-defined  analytic operator from  $\cN$ to $\cZ$. Let $\zeta_0\in \cN$ and $B_\delta(\zeta_0) \subset \cN$ be such that
$$\| D N_{\cF,L_0}[\zeta]\| \le  \cD <1, $$
for any $\zeta \in  B_\delta(\zeta_0)$, and
$$\|N_{\cF,L_0}[\zeta_0]-\zeta_0\|\le \epsilon.$$

If $\epsilon < (1-\cD) \delta$ then the operator $N_{\cF,L_0}$ has a fixed point $\hat\zeta$ in $B_\delta(\zeta_0)$, such that
$$\| \hat\zeta-\zeta_0\| \le {\epsilon \over 1-\cD }.$$
\end{cmp}

In practice, for a computer-assisted proof, the linear map $L_0$ is always chosen as a finite-dimensional approximation of $D \cF[\zeta_0]$.

\subsection{Differential of the renormalization operator}

We will now compute the action of the differential $D \cRG(\phi,\psi)$ on a vector $(u,v) \in T_{(\phi,\psi)} \cA_1(U,V)$. First, 
let $a$, $b$, and $c$ be the coefficients of the projection $\cP$ considered as functions of the pair $(\phi,\psi)$. Then the differential
\begin{equation}
D \cP(\phi,\psi) (u,v)=(u,v)+( Da (u,v) z^2+Db (u,v) z^3, Dc (u,v)). 
\end{equation}
The differential $D c (u,v)$ can be found easily from the equation $(\ref{eq:newFs1})$:
$$D c(u,v)=-v(0),$$
while   $D a (u,v)$ and  $D b (u,v)$ are found by differentiating the system $(\ref{eq:newFs2})$:
\begin{eqnarray} \label{eq:newFs3}
A \cdot \left[
\begin{array}{c}
a \\
b
\end{array}
\right]+B \cdot \left[
\begin{array}{c}
D a (u,v) \\
D b (u,v)
\end{array}
\right]=\left[
\begin{array}{c}
c_1 \\
c_2
\end{array}
\right],
\end{eqnarray}
where $$A=\left[ a_{11} \  a_{12}  \atop a_{21}  \  a_{22} \right]$$ with
\begin{eqnarray}
\nonumber a_{11}&=&4 \hat{\psi}(0)^{3}(v(0)+ Dc (u,v))=0, \\
\nonumber a_{12}&=&6 \hat{\psi}(0)^{5}(v(0)+ Dc (u,v))=0, \\
\nonumber a_{21}&=&6 \hat{\psi}(0)^{2}  \hat{\psi}'(0) (v(0)+ Dc (u,v))+ 2 \hat{\psi}(0)^{3} v'(0)=2 \hat{\psi}(0)^{3} v'(0), \\ 
\nonumber a_{22}&=&15 \hat{\psi}(0)^{4} \hat{\psi}'(0)(v(0)+ Dc (u,v))+3 \hat{\psi}(0)^{5} v'(0)= 3 \hat{\psi}(0)^{5} v'(0) \Rightarrow \\
\nonumber A &=&\left[
\begin{array}{c c}
0 \quad  & \quad 0   \\
2 \hat{\psi}(0)^{3} v'(0) \quad  & \quad 3 \hat{\psi}(0)^{5} v'(0)
\end{array}
\right], \\
\nonumber  B&=&  \left[
\begin{array}{c c}
\hat{\psi}(0)^{4} \quad  & \quad \hat{\psi}(0)^{6}   \\
 2 \hat{\psi}(0)^{3}  \hat{\psi}'(0) \quad  & \quad  3 \hat{\psi}(0)^{5} \hat{\psi}'(0)
\end{array}
\right],\\
\nonumber c_1&=&(v(\phi(0)^2)+ Dc (u,v))+2 \hat{\psi}'(\phi(0)^2) \phi(0) u(0) \\
\nonumber &\phantom{=}& \phantom{(v(\phi(0)^2)}- u(\hat{\psi}(0)^2)-2 \phi'(\hat{\psi}(0)^2)\hat{\psi}(0) (v(0)+ Dc (u,v)) \\
 \nonumber &=&(v(\phi(0)^2)+ Dc (u,v))+2 \hat{\psi}'(\phi(0)^2) \phi(0) u(0) - u(\hat{\psi}(0)^2)
\end{eqnarray}
\begin{eqnarray} \label{eq:matrices}
\nonumber c_2&=&v'(\phi(0)^2) \phi(0) \phi'(0)+2 \hat{\psi}''(\phi(0)^2) \phi(0)^2 \phi'(0) u(0)+\hat{\psi}'(\phi(0)^2) u(0) \phi'(0)+\\
\nonumber &\phantom{=}& \phantom{v'(\phi(0)^2) \phi(0) \phi'(0)}+ \hat{\psi}'(\phi(0)^2) \phi(0) u'(0)-u'(\hat{\psi}(0)^2) \hat{\psi}(0) \hat{\psi}'(0) -\\
\nonumber &\phantom{=}& \phantom{v'(\phi(0)^2) \phi(0) \phi'(0)}-2 \phi''(\hat{\psi}(0)^2) \hat{\psi}(0)^2 \hat{\psi}'(0) (v(0)+ Dc (u,v)) \\
\nonumber &\phantom{=}&\phantom{v'(\phi(0)^2) \phi(0) \phi'(0)}-\phi'(\hat{\psi}(0)^2) (v(0)+ Dc (u,v)) \hat{\psi}'(0)-\phi'(\hat{\psi}(0)^2) \hat{\psi}(0) v'(0) \\
\nonumber &=&v'(\phi(0)^2) \phi(0) \phi'(0)+2 \hat{\psi}''(\phi(0)^2) \phi(0)^2 \phi'(0) u(0)+\hat{\psi}'(\phi(0)^2) u(0) \phi'(0)+ \\
\nonumber &\phantom{=}& \phantom{v'(\phi(0)^2) \phi(0) \phi'(0)} +  \hat{\psi}'(\phi(0)^2) \phi(0) u'(0)-u'(\hat{\psi}(0)^2) \hat{\psi}(0) \hat{\psi}'(0)- \\
\nonumber &\phantom{=}& \phantom{v'(\phi(0)^2) \phi(0) \phi'(0)} - \phi'(\hat{\psi}(0)^2) \hat{\psi}(0) v'(0) \\
\nonumber &=&\left(v'(\phi(0)^2) +2 \hat{\psi}''(\phi(0)^2) \phi(0) u(0) \right)\phi(0) \phi'(0)  + \\
\nonumber  &\phantom{=}&\phantom{\left(v'(\phi(0)^2) \right.}   + \hat{\psi}'(\phi(0)^2) \left(u(0) \phi'(0)+\phi(0) u'(0) \right)-\\
\nonumber  &\phantom{=}&\phantom{\left(v'(\phi(0)^2) \right.}  - u'(\hat{\psi}(0)^2) \hat{\psi}(0) \hat{\psi}'(0)-\phi'(\hat{\psi}(0)^2) \hat{\psi}(0) v'(0).
\end{eqnarray}
The action of the differential of the operator $\cRG$ on a vector $(u,v)$, therefore, is given by
$$D \cRG (\phi,\psi) (u,v)=(\tilde{u},\tilde{v})+( Da (\tilde{u},\tilde{v}) z^2+Db (\tilde{u},\tilde{v}) z^3, -\tilde{v}(0)),$$
where  $Da (\tilde{u},\tilde{v})$ and  $Db (\tilde{u},\tilde{v})$ are the solutions of the system $(\ref{eq:newFs3})$ for the functions  $\tilde{u}$ and $\tilde{v}$ given by
\begin{eqnarray}
\nonumber \tilde{u}&:=&\Pi_\phi D \cR(\phi,\psi) (u,v)\\
\nonumber &=& c \circ \left\{\left( \left(\lambda^{-1} \circ \phi \circ q_2 \right)' \circ \psi \circ \lambda^2 \right) \cdot v \circ  \lambda^2 + \lambda^{-1} \circ u \circ q_2 \circ \psi \circ \lambda^2 -\right.\\
\nonumber  &\phantom{=}& \phantom{c \circ \left.\left( \left(\lambda^{-1} \circ \phi \circ q_2 \right)' \circ \psi \circ \lambda^2 \right) \cdot v \circ  \lambda^2  \ \right.} -  {v(0) \over \lambda^2} \circ \phi \circ q_2 \circ \psi \circ \lambda^2  + \\
\nonumber  &\phantom{=}& \phantom{c \circ \left.\left( \left(\lambda^{-1} \circ \phi \circ q_2 \right)' \circ \psi \circ \lambda^2 \right) \cdot v \circ  \lambda^2 \  \right. }  \left. + 2 \left(\left( \phi \circ q_2 \circ \psi \right)' \circ \lambda^2  \right) \cdot D \la v \right\} \circ c, \\
\nonumber \tilde{v}&:=&\Pi_\psi   D \cR(\phi,\psi) (u,v)= c \circ \left\{ \lambda^{-1} \circ u \circ \lambda^2+{v(0) \over \lambda^2} \circ \phi \circ \lambda^2 +\left( \phi' \circ \lambda^2\right) \cdot D \la v \right\}  \circ c,
\end{eqnarray}
here $\Pi_{\phi/\psi}$ denote the projections on the corresponding components of the pair, and  $D \la v(z):=v(0) z$.

We remark, that the action $D \cR(\phi,\psi)(u,v)$ is of the form $c \circ L (u,v) \circ c$, where $L$ is a linear map and $c$ is the complex conjugation. Therefore, the differential $ D \cR(\phi,\psi)$ is an {{\it anti-linear}} map.

\subsection{Existence of the renormalization fixed point}

A direct application of the Contraction Mapping Principle, in the form stated above, with the operator $\cRG$ acting on $\cA_1(U,V)$ is impossible: the map $\I - L_0$, being neither linear nor anti-linear on $T_{(\phi,\psi)}\cA_1(U,V)$, does not admit a matrix representation in $GL(N,\C)$ (here, $L_0$ is assumed to be an approximation of the action of $D \cRG(\phi,\psi)$ on some $N$-dimensional subspace of  $T_{(\phi,\psi)}\cA_1(U,V)$).

One way to circumvent this problem is to consider the linear operator $D \cRG \circ c$, where $c$ is complex conjugation. We, however, take a different approach, and think of $\cA_1(U,V)$ as a {\it real} Banach space, and expand its elements into
{\it real-valued} vectors over the basis
\begin{eqnarray}
\nonumber \cL:&=&\left\{ \{ (l_\phi^k,0) \}_{k=0}^\infty,\{(0,l_\psi^k)\}_{k=0}^\infty,\{(i l_\phi^k,0)\}_{k=0}^\infty,\{(0,i l_\psi^k)\}_{k=0}^\infty \right\}\\
\nonumber &=&:\left\{ \{ u_k \}_{k=0}^\infty,\{v_k\}_{k=0}^\infty,\{\upsilon_k\}_{k=0}^\infty,\{w_k\}_{k=0}^\infty \right\},
\end{eqnarray}
where 
$$l_\phi(z)={z-c_\phi \over r_\phi}, \quad l_\psi(z)={z-c_\psi \over r_\psi}.$$

This real Banach space will be referred to as  $\tilde{\cA}_1(U,V)$.

Of course, the space  $\tilde{\cA}_1(U,V)$ has a linear complex structure defined on it, specifically, 
$$J(u_k)= \upsilon_k, \  J(\upsilon_k)= -u_k, \ J(v_k)= w_k, \  J(w_k)= -\upsilon_k,$$
which, however, will play no role in our computation of the differential.

Notice, that by definition, balls $B_r((\phi,\psi))$ in both spaces   $\tilde{\cA}_1(U,V)$ and  $\cA_1(U,V)$ coincide.

In this setting, an approximation $L$  of the action of the  differential $D \cRG(\phi,\psi)$ on an $N$-dimensional subspace of  $T_{(\phi,\psi)} \tilde{\cA}_1(U,V)$ admits a representation in $GL(N,\R)$.

To construct an approximation of the differential of renormalization, we first find an  approximate renormalization fixed point by renormalizing several times the pair $(\id,p_\theta)$, where $p_\theta$ is the conjugate $ l_\psi \circ P_\theta \circ l_\psi^{-1} $ of the quadratic polynomial $P_\theta=e^{2 \pi i \theta} z  -0.5 e^{2 \pi i \theta} z^2$, $\theta=(\sqrt{5}-1)/2$, by the linear transformation $l_\psi$.  The resultant pair is referred to as $(\phi_a,\psi_a)$.

A linear approximation $L_a$ of the action of the renormalization differential  $D \cRG(\phi_a,\psi_a)$ is constructed on the span
$${\rm span} \left\{ \{ (l_\phi^k,0) \}_{k=0}^{N_1},\{(0,l_\psi^k)\}_{k=0}^{N_2},\{(i l_\phi^k,0)\}_{k=0}^{N_1},\{(0,i l_\psi^k)\}_{k=0}^{N_2} \right\},$$
where 
$$N_1=120, \quad N_2=160,$$
and is of the form 
$$L_a=\left[ \begin{array}{c c c c}
M_{1 1} & M_{1 2} & M_{1 3} & M_{1 4} \\
M_{2 1} & M_{2 2} & M_{2 3} & M_{2 4} \\
M_{3 1} & M_{3 2} & M_{3 3} & M_{3 4} \\
M_{4 1} & M_{4 2} & M_{4 3} & M_{4 4}
\end{array}
\right],
$$
where the matrices in the first column above are the following projections of the action of the differential:
\begin{itemize}
\item[$\bullet$] $M_{1 1}$ is the $(N_1+1) \times (N_1+1)$ matrix of columns $ \left[ \Re P_{N_1} \Pi_\phi  D \cRG(\phi_a,\psi_a) u_k \right]$,
\item[$\bullet$] $M_{3 1}$ is the $(N_1+1) \times (N_1+1)$ matrix of columns $ \left[ \Im P_{N_1} \Pi_\phi  D \cRG(\phi_a,\psi_a) u_k \right]$,
\item[$\bullet$] $M_{2 1}$ is the $(N_1+1) \times (N_2+1)$ matrix of columns
$ \left[ \Re P_{N_2} \Pi_\psi  D \cRG(\phi_a,\psi_a) u_k \right]$,
\item[$\bullet$] $M_{4 1}$ is the $(N_1+1) \times (N_2+1)$ matrix of columns
$ \left[ \Im P_{N_2} \Pi_\psi  D \cRG(\phi_a,\psi_a) u_k \right]$,
\end{itemize}
the matrices in the third column above are  the following projections of the action of the differential:
\begin{itemize}
\item[$\bullet$] $M_{1 3}$ is the $(N_1+1) \times (N_1+1)$ matrix of columns $ \left[ \Re P_{N_1} \Pi_\phi  D \cRG(\phi_a,\psi_a) \upsilon_k \right]$,
\item[$\bullet$] $M_{3 3}$ is the $(N_1+1) \times (N_1+1)$ matrix of columns $ \left[ \Im P_{N_1} \Pi_\phi  D \cRG(\phi_a,\psi_a) \upsilon_k \right]$,
\item[$\bullet$] $M_{2 3}$ is the $(N_1+1) \times (N_2+1)$ matrix of columns
$ \left[ \Re P_{N_2} \Pi_\psi  D \cRG(\phi_a,\psi_a) \upsilon_k \right]$,
\item[$\bullet$] $M_{4 3}$ is the $(N_1+1) \times (N_2+1)$ matrix of columns
$ \left[ \Im P_{N_2} \Pi_\psi  D \cRG(\phi_a,\psi_a) \upsilon_k \right]$,
\end{itemize}
the matrices in the second column above are  the following projections of the action of the differential:
\begin{itemize}
\item[$\bullet$] $M_{1 2}$ is the $(N_2+1) \times (N_1+1)$ matrix of columns $ \left[ \Re P_{N_1} \Pi_\phi  D \cRG(\phi_a,\psi_a) v_k \right]$,
\item[$\bullet$] $M_{3 2}$ is the $(N_2+1) \times (N_1+1)$ matrix of columns $ \left[ \Im P_{N_1} \Pi_\phi  D \cRG(\phi_a,\psi_a) v_k \right]$,
\item[$\bullet$] $M_{2 2}$ is the $(N_2+1) \times (N_2+1)$ matrix of columns
$ \left[ \Re P_{N_2} \Pi_\psi  D \cRG(\phi_a,\psi_a) v_k \right]$,
\item[$\bullet$] $M_{4 2}$ is the $(N_2+1) \times (N_2+1)$ matrix of columns
$ \left[ \Im P_{N_2} \Pi_\psi  D \cRG(\phi_a,\psi_a) v_k \right]$,
\end{itemize}
the matrices in the forth column above are  the following projections of the action of the differential:
\begin{itemize}
\item[$\bullet$] $M_{1 4}$ is the $(N_2+1) \times (N_1+1)$ matrix of columns $ \left[ \Re P_{N_1} \Pi_\phi  D \cRG(\phi_a,\psi_a) w_k \right]$,
\item[$\bullet$] $M_{3 4}$ is the $(N_2+1) \times (N_1+1)$ matrix of columns $ \left[ \Im P_{N_1} \Pi_\phi  D \cRG(\phi_a,\psi_a) w_k \right]$,
\item[$\bullet$] $M_{2 4}$ is the $(N_2+1) \times (N_2+1)$ matrix of columns
$ \left[ \Re P_{N_2} \Pi_\psi  D \cRG(\phi_a,\psi_a) w_k \right]$,
\item[$\bullet$] $M_{4 4}$ is the $(N_2+1) \times (N_2+1)$ matrix of columns
$ \left[ \Im P_{N_2} \Pi_\psi  D \cRG(\phi_a,\psi_a) w_k \right]$.
\end{itemize}

The Newton map $N_{\cRG,L_a}$ is now iterated until a good approximation $(\phi_0,\psi_0)$ of the renormalization fixed point is found. After that, a $2(N_1+1)+2(N_2+1)$ by $2(N_1+1)+2(N_2+1)$ matrix $L_0$ approximating the action of the differential $D \cRG(\phi_0,\psi_0)$ on $T_{(\phi_0,\psi_0)} \tilde{\cA}_1(U,V)$, is constructed.

This linear map $L_0$ is used in the following Theorem:

\begin{thm}\label{mainthmA} $\phantom{aaa}$\\
There exists a pair of polynomials $(\phi_0,\psi_0) \in \tilde{\cA}_1(U,V)$, where  $U=\D_{r_\phi}(c_\phi)$ and   $V=\D_{r_\psi}(c_\psi)$  are as in the Main Theorem 1, such that the operator $\cRG$ is analytic and compact in $B_r((\phi_0,\psi_0))$ in $\tilde{\cA}_1(U,V)$, $r=6.6610992 \cdot 10^{-11}$.

Furthermore, the Newton map $N_{\cRG,L_0}$ admits the following bounds:
\begin{itemize}
\item[$i)$] $\|(\phi_0,\psi_0)-N_{\cRG,L_0}(\phi_0,\psi_0)  \|_1 < \eps=1.33221983054516078 \cdot 10^{-11}$;
\item[$ii)$] $\| D N_{\cRG,L_0}(\phi,\psi)\| < \cD=0.287850876998644962$ for all $(\phi,\psi) \in B_\delta(  (\phi_0,\psi_0)  )$, $\delta=2.90241954016095172 10^{-11}$.
\end{itemize}
Therefore,
$$\epsilon < (1-\cD) \delta < -7.34735699562457612 \cdot 10^{-12}<0,$$
and the operator $N_{\cRG,L_0}$ has a fixed point in $B_\delta( (\phi_0,\psi_0) ) \subset \tilde{\cA}_1(U,V)$, while $\cRG$ has a fixed point  $(\phi_*,\psi_*)$ in  $B_r((\phi_0,\psi_0))$ in $\cA_1(U,V)$.
\end{thm}

\noindent
The compactness of $\cRG$ and $D|_{\zeta_*}\cRG$ follow from:

\begin{prop}  \label{prop:renorm_compactness}
There exists $r'<r$ such that every  $(\phi, \psi) \in  B_{r'}((\phi_*,\psi_*))$  is in $\tilde{\cA}_1(\tU,\tV)$, where $\tU \supset \fD_{\tr_\phi} (\tc_\phi)$ and  $\tV \supset \fD_{\tr_\psi}(\tc_\psi)$, with $\tr_\phi=0.937$, $\tr_\psi=0.874$,  $\tc_\phi=0.6+i 0.09$ and $\tc_\psi=c_\psi$. 
\end{prop}

\subsection{Existence of an unstable subspace}
Let $\lambda_*$ be as in the Main Theorem 1. 
We introduce an analytic operator
\begin{equation}
\cR_*(\eta,\xi)=\left(c \circ \lambda_*^{-1} \circ \eta \circ \xi \circ \lambda_* \circ c, c \circ \lambda_*^{-1} \circ \eta  \circ \lambda_*   \circ c \right),
\end{equation}
whose difference from $\cR$ is that the scaling factor $\lambda_*$ is fixed, rather than dependent on the pair $(\eta,\xi)$.
Of course, $\zeta_*$ is a fixed point of $\cR_*$. 

The following is easy to show:
\begin{lemma}
\label{difference-spectrum}
The spectra of the linear operators $D\cR_*^2|_{\zeta_*}$ and $D\cR^2|_{\zeta_*}$ satisfy
$$\text{sp}(D\cR_*^2|_{\zeta_*})=\text{sp}(D\cR^2|_{\zeta_*})\cup\{ 1\}.$$
\end{lemma}

Trivially, the map $\lambda_* \circ c$ where $\lambda_*=\lambda(\zeta^*)$ admits two invariant lines $l_+$ and $l_-$, which are perpendicular to each other.
 Specifically,
from  the invariance condition  $\lambda_* \circ c (x+a i x)=t (x+a i x)$ we have
\begin{eqnarray}
\nonumber \lambda_* \circ c (x+a i x)&=&\left( \Re(\lambda_*) x+ a \Im(\lambda_*) x \right)+i \left( \Im(\lambda_*) x- a \Re(\lambda_*) x \right)=t (x +a i x) \\
\nonumber &\implies& a \Re(\lambda_*) + a^2 \Im(\lambda_*)= \Im(\lambda_*) - a \Re(\lambda_*) \\
\nonumber &\implies&  a^2 \Im(\lambda_*) +2 a \Re(\lambda_*) - \Im(\lambda_*)=0 \\
\label{eq:apm} &\implies& a_\pm={ -\Re(\lambda_*) \pm |\lambda_*| \over \Im(\lambda_*)}.  
\end{eqnarray}

\begin{figure}[t]
\centering
\vspace{2mm}       
\includegraphics[height=7.0cm,width=6cm,angle=-90]{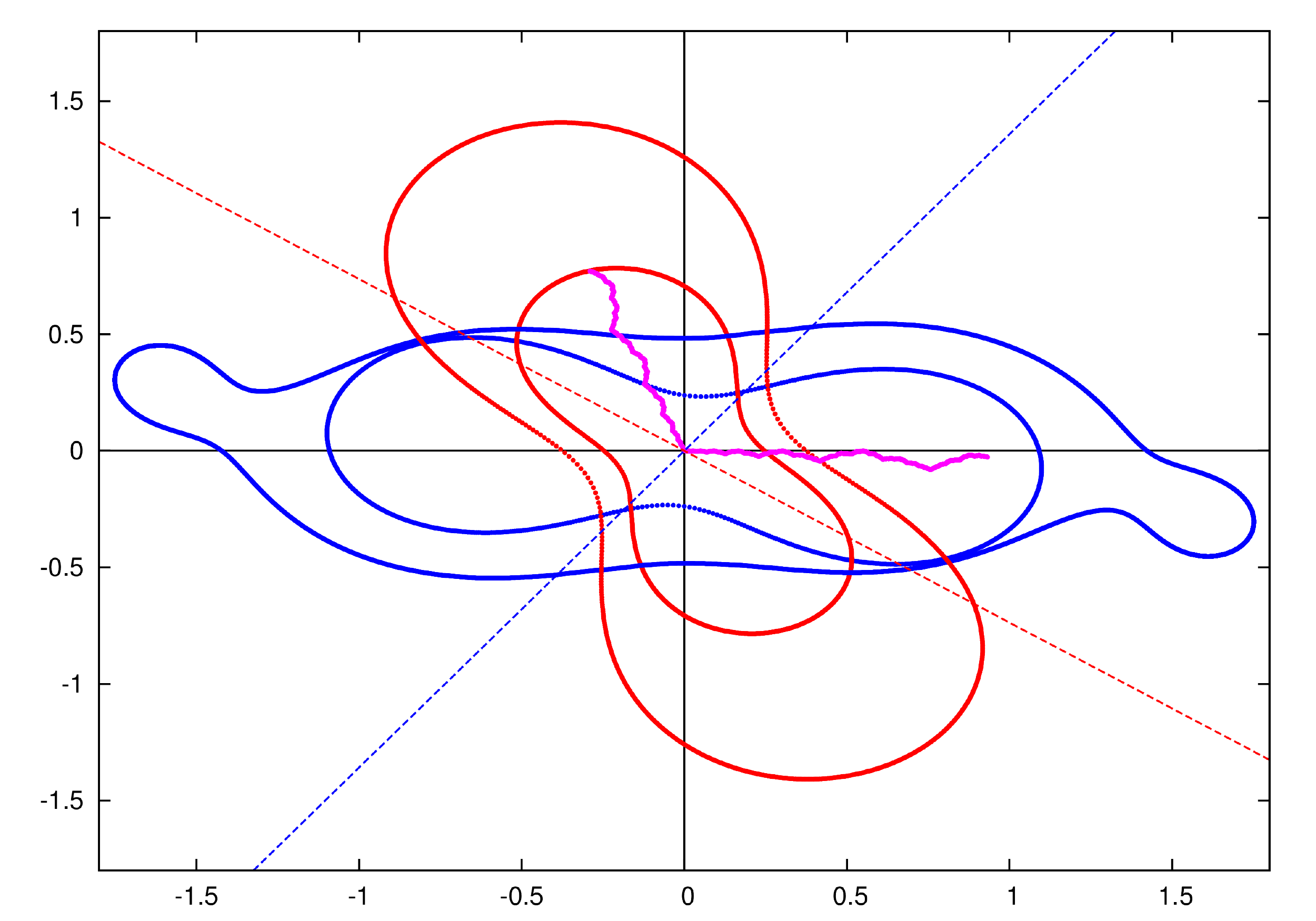}
\caption{Invariant lines: $l_+$ in blue and $l_-$ in red.}
\label{fig:lines} 
\end{figure}

Let $u(z)=\phi_u(z^2)$ and $v(z)=\psi_v(z^2)$ be analytic vector fields, such that $(\phi_u,\psi_v)$ is in $T_{(\phi_*,\psi_*)} \cA_1(U,V)$. 
Consider the action of the differential $D \cR_*(\zeta_*)(u,v)$
on a pair of such  fields   $(u_0,v_0)$, contained in some cone fields: $u_0(z) \in \cC_{u_0}(z)$ and $v_0(x) \in \cC_{v_0}(z)$:
\begin{equation}\label{eq:diraction}
(u_1,v_1):=(D \cR_*(\zeta_*)(u_0,v_0)= c \circ \lambda_*^{-1} \circ \left(u_0 \circ \xi_* + (\eta'_* \circ \xi_* ) \cdot v_0, u_0   \right) \circ \lambda_* \circ c.
\end{equation}

Clearly, $v_1 \in C_{v_0}(z)$, whenever $u_0(z) \in \cC_{u_0}(z)$,  if we set 
\begin{eqnarray} \label{eq:Cbetaw}
C_{v_0}(z):=c\left(\lambda_*^{-1}\left(C_{u_0}(\lambda_*(c(z)))\right)\right).
\end{eqnarray}
To find a pair of invariant cone fields, it is, therefore, sufficient to construct a cone field $C_{u_0}(z)$ invariant under the operator
\begin{equation}\label{eq:uoperator}
\cT(u):=c \circ \lambda_*^{-1} \circ \left(u \circ \xi_* \circ \lambda_* \circ c  +  (\eta'_* \circ \xi_* \circ \lambda_* \circ c ) \cdot (c \circ \lambda_*^{-1} \circ  u \circ |\lambda_*|^2) \right),
\end{equation}
and define $C_{v_0}(z)$ as in $(\ref{eq:Cbetaw})$. Notice that
\begin{eqnarray}
\label{eq:uoperator0} \cT(u)(0)&=&c \circ \lambda_*^{-1} \circ \left(u(\xi_*(0))  +  \eta'_*(\xi_*(0)) \cdot (c \circ \lambda_*^{-1} \circ  u(0) \right), \\
\label{eq:uoperator2n}\cT(u)(|\lambda|^{2n}) &=&c \circ \lambda_*^{-1} \circ(u(\xi_*(\lambda_*(c (|\lambda|^{2n}))))  +  \\
\nonumber &\phantom{=}& \phantom{c \circ \lambda_*^{-1}} + \eta'_*(\xi_*(\lambda_*(c(|\lambda|^{2n}) ))) \cdot (c \circ \lambda_*^{-1} \circ  u (|\lambda|^{2(n+1)}))),\\
\label{eq:uoperatorxi0} \cT(u)\circ \xi_*(0)&=&c \circ \lambda_*^{-1} \circ \left(u(|\lambda_*|^2) + \eta'_*(|\lambda_*|^2) \cdot (c \circ \lambda_*^{-1} \circ  u(|\lambda_*|^2) \right).
\end{eqnarray}

We will now proceed to construct a cone field $C_{u}$ which is invariant under the operator $\cT$ at points $|\lambda|^{2n}$, $n=0, \ldots, \infty$, and $0$. To that end, as formulas $(\ref{eq:uoperator0})-(\ref{eq:uoperatorxi0})$ suggest,  we will require bounds on derivatives of $\eta_*'$ at a sequence of points, as well as bounds on the images of the set $\lambda_*(c(|\lambda_*|^{2n}(Z)))$ under $\xi_*$.

\begin{lemma}\label{lem:conebounds}
The following bounds hold for the fixed point pair $\zeta_*$.
\begin{itemize}
\item[$1)$] $\eta_*'(\xi_*(0))=|\lambda_*|^{-2}$,
\item[$2)$] $|\arg(\eta_*'(|\lambda_*|^2)| <1.062$
\item[$3)$] $\sum_{n=1}^\infty |\arg(\eta_*'(\xi_*(\lambda_*(c(|\lambda_*|^{2n} )))))|<0.075=:\epsilon$,
\item[$4)$]  $| \arg(u(\xi_*(\lambda_*(|\lambda_*|^2))))-\arg(u(\xi_*(0)))   |<1.066$ and \\
$| \arg(u(\xi_*(\lambda_*(\bar{z}))))-\arg(u(\xi_*(0)))   |<0.398=\varepsilon$ for all $z \in [0,|\lambda_*|^4]$.
\end{itemize}
\end{lemma}
\begin{proof}

Notice that for the fixed point pair $\zeta_*=(\eta_*,\xi_*)$,
\begin{eqnarray}
\nonumber \eta_*' &=& c \circ \left( \eta'_* \circ \xi_* \cdot \xi_*'\right) \circ \lambda_* \circ c= c \circ \left( \eta'_* \circ \xi_* \cdot (c \circ \lambda_* \circ \eta_* \circ \lambda_* \circ c)'  \right) \circ \lambda_* \circ c\\
\label{eq:univbound} &=&\overline{\eta_*' \xi_* \circ \lambda_* \circ c} \cdot \left(\eta_*' \circ |\lambda_*|^2 \right)=\eta_*'  \implies  \overline{\eta_*' \xi_* \circ \lambda_* \circ c} = |\lambda_*|^{-2} { \phi_*'(z^2) \over \phi'_*(|\lambda_*|^4 z^2)}.
\end{eqnarray}
In particular, Claim $1)$ follows.
 
Consider the conformal transformation of $D_{\tr_\phi}(\tc_\phi)$, where $\tr_\phi$ and $\tc_\phi$ are as in Proposition $\ref{prop:renorm_compactness}$, onto the unit disk which maps the point $|\lambda_*|^4 z^2$ to zero:
$$ t^{-1}:=m \circ l, \quad l(x)={x  - \tc_\phi \over \tr_\phi}, \quad m(x):={x-l(|\lambda|^4 z^2) \over 1 - \overline{l(|\lambda|^4 z^2)} x}.$$
The function 
$$f(x)={ \phi_*(t(x)) -\phi_*(t(0))  \over \partial_x \phi_*(t(x))\arrowvert_{x=0} }$$
is univalent on a unit disk, normalized so that $f(0)=0$ and $f'(0)=1$, and
$$f'(t^{-1}(z^2))= { \phi_*'(z^2)   \over  \phi_*'(|\lambda|^4 z^2)  } {t'( (t^{-1}(z^2) ) \over t'(0)}.$$
We can now use  a classical distortion bound:
$$|\arg(f'(z))|<4 \arcsin{|z|},\ { \rm whenever} \ |z| <{1 \over \sqrt{2}},$$
in the class $\cS$ of functions $f$ univalent on $\fD$ and normalized so that $f(0)=0$, $f'(0)=1$, to get
\begin{eqnarray}
\nonumber && \left| \arg \left( { \phi_*'(z^2)   \over  \phi_*'(|\lambda|^4 z^2)  } {t'( (t^{-1}(z^2) )) \over t'(0) } \right) \right| = \\
\nonumber &&\phantom{aaaa}= \left| \arg \left( { \phi_*'(z^2)   \over  \phi_*'(|\lambda|^4 z^2)  } \right) +  \arg \left( \left( \tr_\phi^2-\overline{(|\lambda_*|^4 |z|^2 -\tc_\phi)}  (z^2-\tc_\phi) \right)^2 \right) \right| \\
\nonumber  &&\phantom{aaaa}= \left|\arg \left(t^{-1} (z^2) \right) \right| < 4 \arcsin\left( \left|\tr_\phi {1-|\lambda_*|^4 \over \tr^2_\phi-\overline{( |\lambda_*|^4 z^2-\tc_\phi  )}(z^2-\tc_\phi) }  |z|^2  \right|  \right).
\end{eqnarray}
Therefore,
\begin{eqnarray}
\nonumber | \arg{ \left( \eta_*'(\xi_*(\lambda_* c(z))) \right) }| &<&  4 \arcsin\left( \left|\tr_\phi{1-|\lambda_*|^4 \over \tr^2_\phi-\overline{( |\lambda_*|^4 z^2-\tc_\phi  )}(z^2-\tc_\phi) }  |z|^2  \right|  \right) \\
\label{eq:univbound1} &+& \left| \arg \left( \left( \tr_\phi^2-\overline{(|\lambda_*|^4 |z|^2 -\tc_\phi)}  (z^2-\tc_\phi) \right)^2 \right) \right|.
\end{eqnarray}
The bound in part $3)$ of the Lemma is obtain by evaluating first ten terms in the sum on a computer, and then estimating the remainder using simpler bounds on $\arcsin$ and $\arg$ functions in formula $(\ref{eq:univbound1})$. 
  
To demonstrate  the fourth claim, we first recall  the classical rotation bound
$$ \left|\arg{\left( { f(w) \over w      }   \right)}  \right|<\ln {1+|w| \over 1-|w|}$$
in the class $\cS$. In our case, $u=\varphi \circ q_2$  where  $\varphi$ is univalent on $\fD_{r_\phi}(c_\phi)$. Let $l(z)$ be the affine transformation that maps $\fD$ onto $\fD_{r_\phi}(c_\phi)$, set $a=l^{-1}(1)$, and let $M_\gamma(w)=(\gamma w +a )/(1+\bar{a} \gamma w)$, $|\gamma|=1$,   be the fractional linear transformation that preserves the unit disk and  maps $0$ into $a$. Consider the family 
$$f_\gamma(w)={ \varphi(l(M_\gamma(w)))-\varphi(1)  \over \varphi'(1) r_{\phi} \gamma (1-|a|^2)},$$
whose members are in $\cS$. Let point $w$ be the preimage in $\fD$  of some $\xi_*(\lambda_*(\bar{z}))$ under $l \circ M$, 
$$  \left| \arg{ \left( { f_\gamma(w) \over w }  \right)}  \right|=\left|\arg{\left( { \varphi(l(M_\gamma(w)))-\varphi(1)   \over  \varphi'(1) \gamma w} \right) } \right|<\ln {1+|w| \over 1-|w|}.$$ 
Now, for every $w \in \fD$ there is a choice of $\gamma=\gamma(w)$, such that  $ \varphi'(1) \gamma w \in \fR$, then
$$\left|\arg{ \left( \varphi(l(M_\gamma(w)))-\varphi(1) \right)  } \right|<\ln {1+|w| \over 1-|w|}.$$
A computer implementation of this bound results in the bounds in part $4)$.

Part $2)$ has been found directly on the computer.

\flushright $\square$
\end{proof}

We will now demonstrate existence of an invariant cone field $C$.

\begin{prop}\label{prop:unstable_dir} 
The spectral radii of the linear operators $D \cR_*^2(\zeta_*)$ and $D \cRG^2(\zeta_*)$ are larger than $1$. 
\end{prop} 
\begin{proof}
First, rewrite the equation $(\ref{eq:uoperatorxi0})$ as 
$$\cT(u)\circ \xi_*(0)=c \circ \lambda_*^{-1} \circ \left(u(|\lambda_*|^2) + c \circ \tilde{\lambda}_*^{-1} \circ  u(|\lambda_*|^2) \right), \quad \tilde{\lambda}_*= {\lambda_* \over \overline{\eta'_*(|\lambda_*|^2)}}.$$
Similarly to $(\ref{eq:apm})$, set  
$$\tilde{a}_\pm={ -\Re(\tilde{\lambda}_*) \pm |\tilde{\lambda}_*| \over \Im(\tilde{\lambda}_*)},$$
and let $\tilde{v}_\pm$ be the unit vectors of the invariant lines $\tilde{l}_\pm=(t+i \tilde{a}_\pm t)$ under the map $\tilde{\lambda}_* \circ c$.

Given a vector $v \in \fC$ and a positive $\alpha< \pi/2$, we denote a cone
$$C(v,\alpha)=\{x \in \fC \setminus \{0\}: |\arg(v)-\arg(x)| \le \alpha \} \cup \{0\}.$$

\medskip

\noindent  {\it Step 1)} We assume that cone field $\cC_u$ has some particular values at points $0$, $|\lambda_*|^2$ and $1$, and demonstrate that $\cT(\cC_u)(0) \subset {\rm Int} \ \cC_u(0)$ and $\cT(\cC_u)(1) \subset {\rm Int} \ \cC_u(1)$.

We make the following assumptions:
\begin{eqnarray}
\label{eq:conel2} \cC_u(|\lambda_*|^2)&:=&C(\tilde{v}_+,\alpha(|\lambda_*|^2)), \quad \alpha(|\lambda_*|^2)=\arccos(0.5967), \\
\label{eq:conel1} \cC_u(1)&:=&c(\lambda_*(C(\tilde{v}_+,\alpha(|\lambda_*|^2))))=:C(c(\lambda_*(\tilde{v}_+)),\alpha(1)))), \\
\label{eq:cone0} \cC_u(0)&:=&C(\tilde{v}_+,\alpha(0)), \quad \alpha(0)=\arccos(0.71).
\end{eqnarray}

The cone $C(\tilde{v}_+,\alpha(|\lambda_*|^2))$ is invariant under the map $c \circ \tilde{\lambda}_*^{-1}$, therefore for any $u(z)$ such that $u(|\lambda_*|^2) \in C(\tilde{v}_+,\alpha(|\lambda_*|^2))$, we have that  
$$\cT(u)\circ \xi_*(0) \in c(\lambda_*^{-1}(C(\tilde{v}_+,\alpha(|\lambda_*|^2))))=: C(c(\lambda_*^{-1}(\tilde{v}_+)),\alpha(1)),$$  
and, according to $(\ref{eq:uoperatorxi0})$, $\cT(u)(1) \in  \cC_u(1)$ whenever $u(|\lambda_*|^2) \in \cC_u(\lambda_*|^2)$. Additionally, the boundary $\partial \cC_u(|\lambda_*|^2)$ is a union of two lines, call them $l$ and $r$ and , and  the map $c \circ \tilde{\lambda}_*^{-1}$ interchanges them: $c \circ \tilde{\lambda}_*^{-1}(l)=r$ and   $c \circ \tilde{\lambda}_*^{-1}(r)=l$. Therefore, if $u(|\lambda_*|^2) \ne 0$ is in $\partial \cC_u(|\lambda_*|^2)=$, say $u(|\lambda_*|^2)\parallel l$, then $c \circ \tilde{\lambda}_*^{-1}(u(|\lambda_*|^2)) \parallel r$ and  $u(|\lambda_*|^2)+ c \circ \tilde{\lambda}_*^{-1}(u(|\lambda_*|^2)) \in {\rm Int} \ \cC_u(|\lambda_*|^2)$, and therefore, $\cT(\cC_u)(1) \in  {\rm Int} \ \cC_u(1)$.

We verify on the computer that  the cone $c(\lambda_*^{-1}(C(c(\lambda_*^{-1}(\tilde{v}_+)),\alpha(1))))$ is contained in  ${\rm Int} \ C(\tilde{v}_+,\alpha(0))$ (see Fig. \ref{fig:fields}). According to $(\ref{eq:uoperator0})$, this, together with the invariance of $\cC_u(0)$ under the map $c \circ \lambda_*^{-1} \circ c \circ \lambda_*^{-1}$ implies that $\cT(\cC_u)(0) \subset {\rm Int} \ \cC_u(0)$.

\begin{figure}
\centering
\vspace{2mm}       
\begin{tabular}{c c} 
\includegraphics[height=6cm,width=5.2cm,angle=-90]{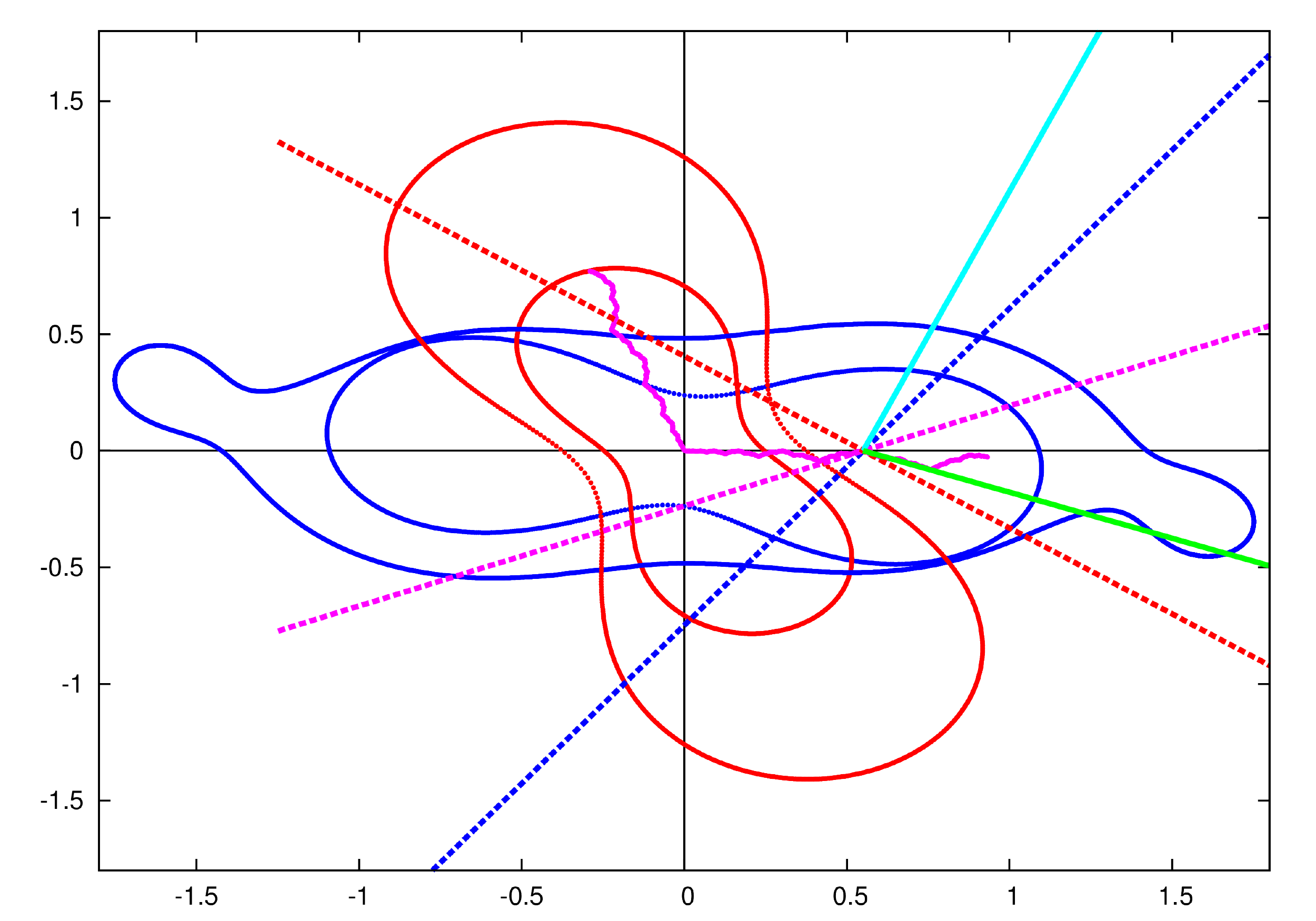}  &  \includegraphics[height=6cm,width=5.2cm,angle=-90]{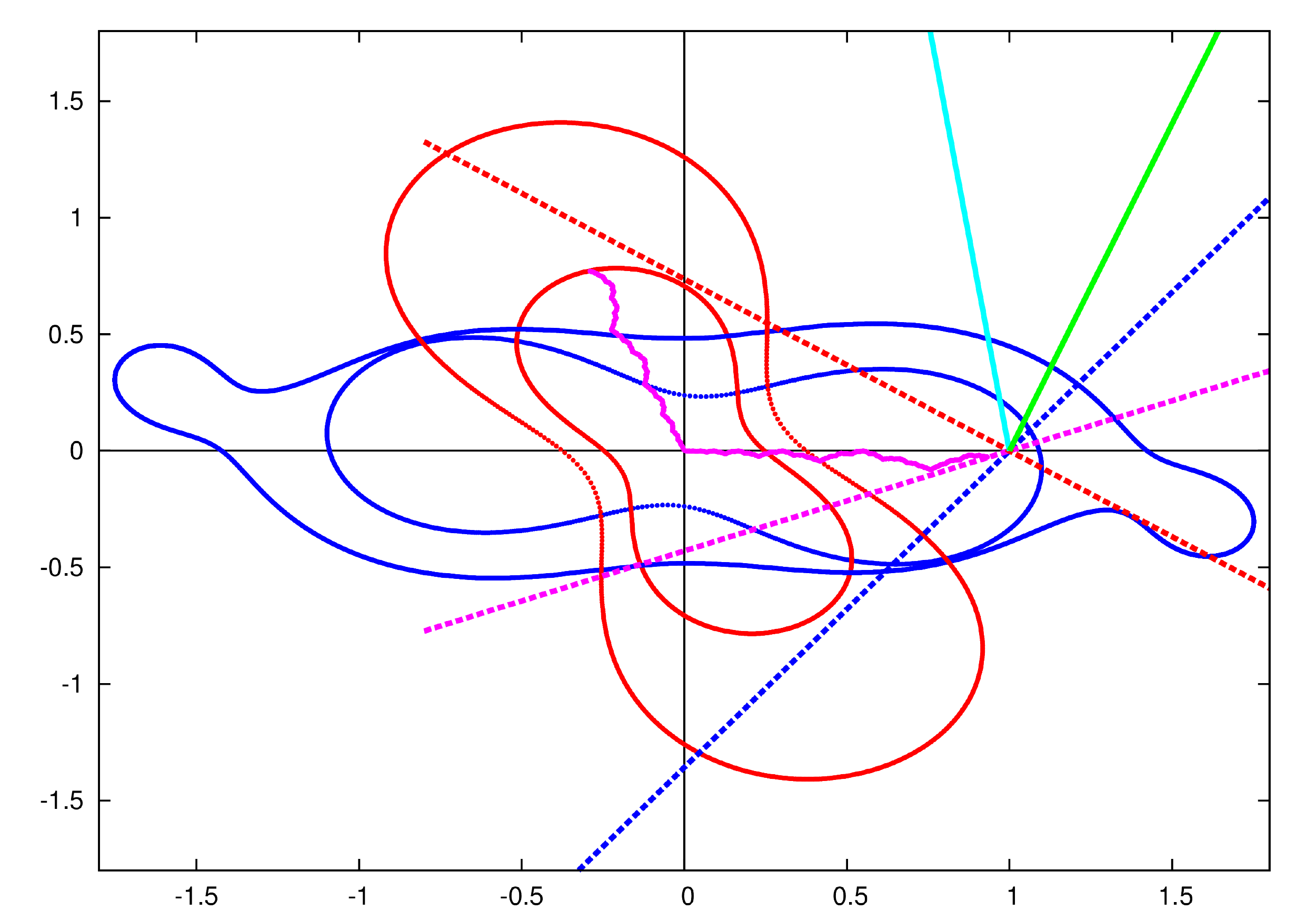}   \\
$a)$ & $b)$ \\
 \includegraphics[height=6cm,width=5.2cm,angle=-90]{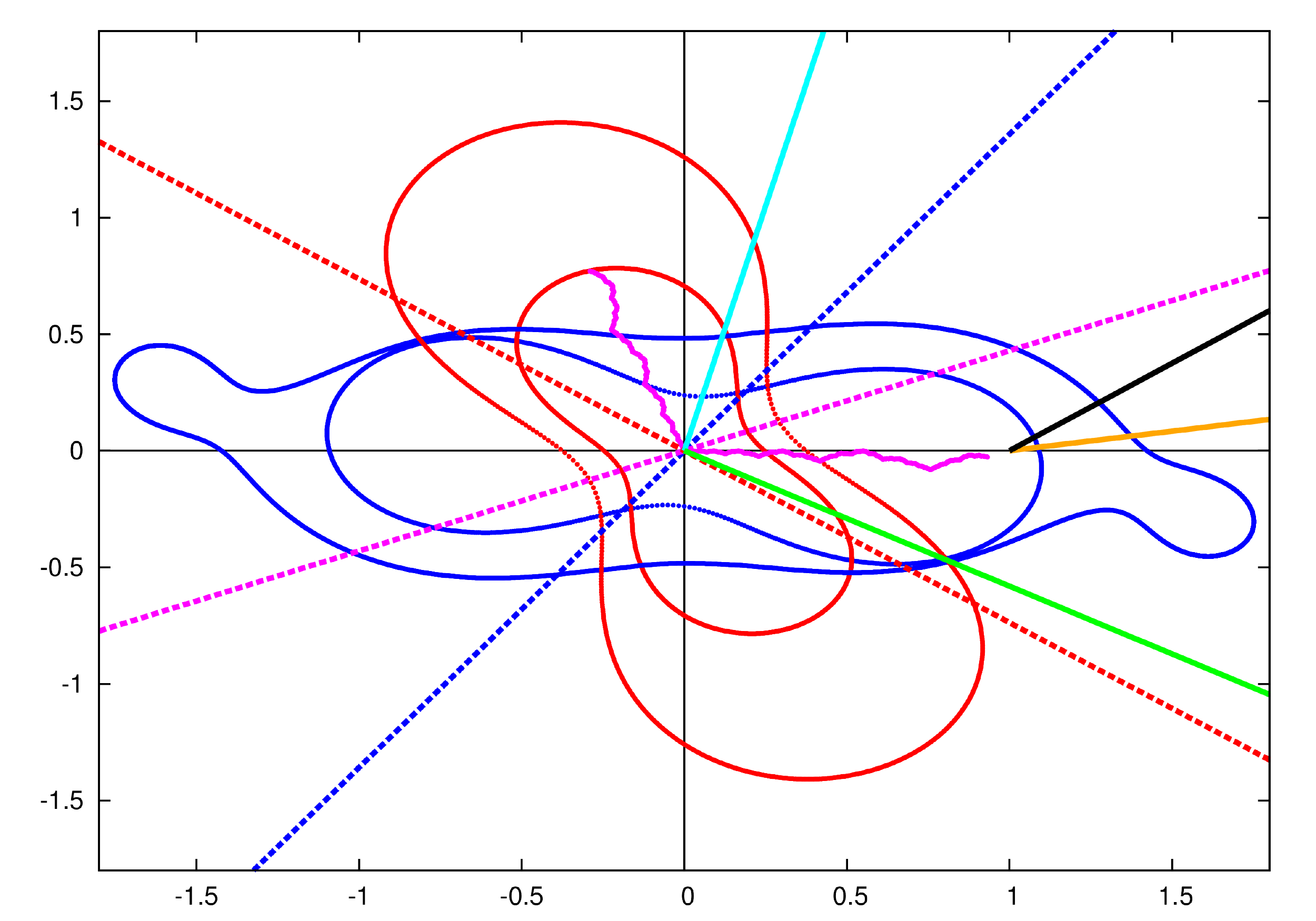}  &  \includegraphics[height=6cm,width=5.2cm,angle=-90]{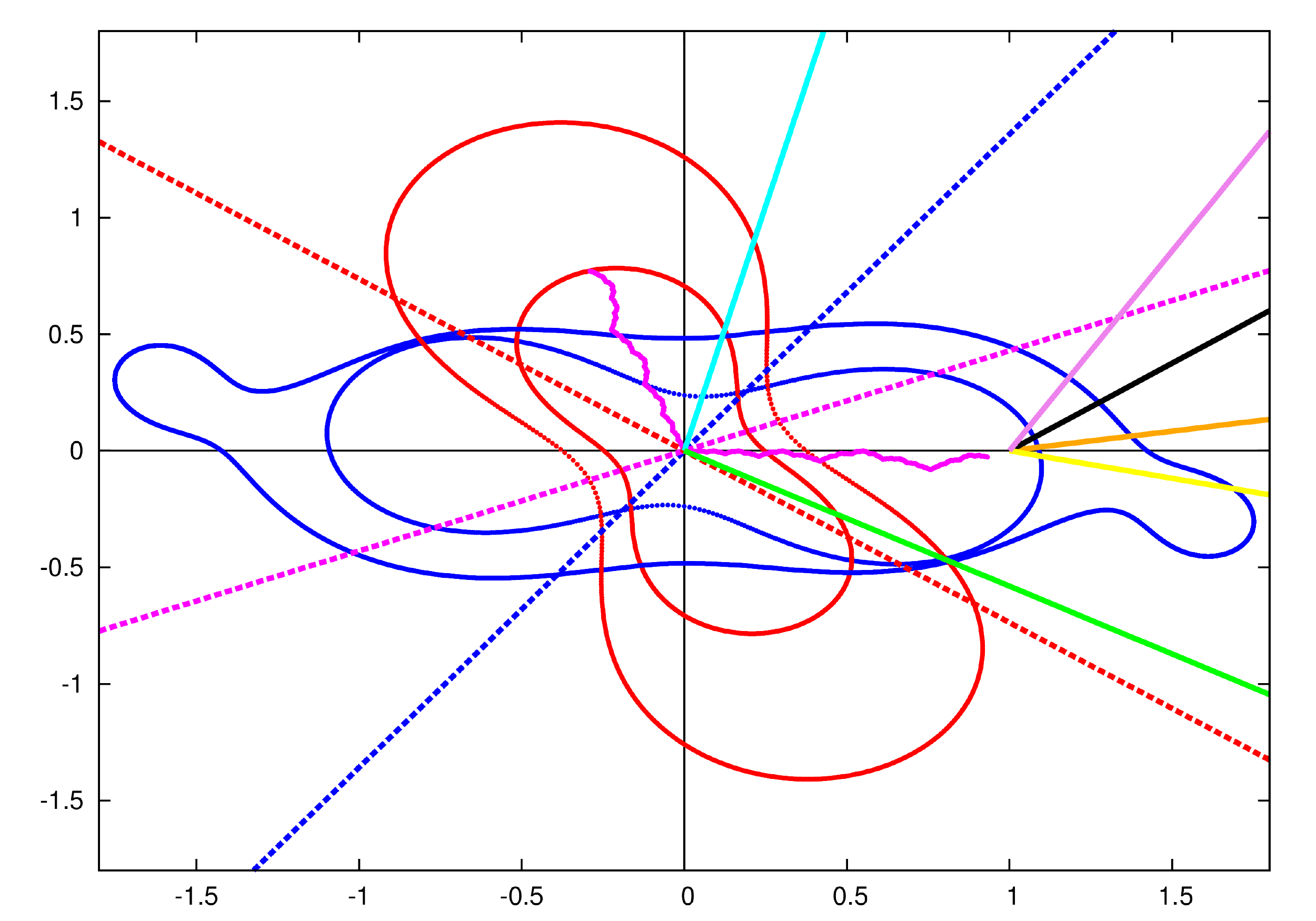}   \\
$c)$ & $d)$
\end{tabular}
\caption{The invariant lines $l_+$, $l_-$ and $\tilde{l}_+$ are drawn in blue, red and magenta, respectively.  $a)$ The cone $C(\tilde{v}_+,\alpha(|\lambda_*|^2))$ is between cyan and green lines.  b)  The cone  $C(c(\lambda_*(\tilde{v}_+)),\alpha(1)):=c(\lambda_*(C(\tilde{v}_+,\alpha(|\lambda_*|^2))))$ is between cyan and green lines. c) The image of the cone $C(c(\lambda_*(\tilde{v}_+)),\alpha(1))$  under $c \circ \lambda_*$ (orange to black) is contained in the cone $\cC_u(0)$ (green to cyan). d) $\varepsilon$-thickening (yellow and violet) of the  image of the cone $C(c(\lambda_*(\tilde{v}_+)),\alpha(1))$  under $c \circ \lambda_*$ (orange to black) is, again,  contained in the cone $\cC_u(0)$ (green to cyan)}
\label{fig:fields}       
\end{figure}

\medskip

\noindent  {\it Step 2)} We continue to make assumptions on the field $\cC(z)$:
\begin{equation}
\label{eq:conesdef} \cC_u(|\lambda_*|^{2 n})=C(\tilde{v}_+,\alpha(0)+  \sum_{i=n}^\infty |\arg(\eta_*'(\xi_*(\lambda_*(c(|\lambda_*|^{2i} )))))|), \quad n=2,3, \ldots .
\end{equation}

By the bound in part 4) in Lemma \ref{lem:conebounds},  $u(\xi_*(\lambda_*|\lambda_*|^{2n}))$, $n=2,3, \ldots$ is in $C(c(\lambda_*(\tilde{v}_+)) , \alpha(1)+\varepsilon)$, $\varepsilon=0.398$ as in part 4) of Lemma \ref{lem:conebounds}, whenever $u(\xi_*(0)) \in C(c(\lambda_*(\tilde{v}_+)),\alpha(1))$. We verify on the computer that 
\begin{equation}\label{eq:cone1}
c(\lambda_*^{-1}(C(c(\lambda_*(\tilde{v}_+)),\alpha(1)+\varepsilon))) \subset {\rm Int} \ \cC_u(0):= {\rm Int} \ C(\tilde{v}_+,\alpha(0))
\end{equation}
(see Fig. \ref{fig:fields}). Furthermore,
\begin{eqnarray}
\nonumber  c \circ \lambda_*^{-1} \! \! & \! \! \circ  \!  \! & \! \! \left( \eta'_*(\xi_*(\lambda_*(c(|\lambda|^{2n}) )) \cdot (c \circ \lambda_*^{-1} \circ  u (|\lambda|^{2(n+1)})) \right) \in \\ 
\label{eq:cone2}  \! \! &\! \! \in \! \! & \! \! C(\tilde{v}_+,\alpha(0)+\sum_{i=n}^\infty |\arg(\eta_*'(\xi_*(\lambda_*(c(|\lambda_*|^{2i} )))))|)
\end{eqnarray}
whenever  
$$u (|\lambda|^{2(n+1)}) \in C(\tilde{v}_+,\alpha(0)+\sum_{i=n+1}^\infty |\arg(\eta_*'(\xi_*(\lambda_*(c(|\lambda_*|^{2i} )))))|).$$
Conditions $(\ref{eq:cone1})$ and $(\ref{eq:cone2})$ imply that $\cT(u)(|\lambda_*|^{2 n})$, given by $(\ref{eq:uoperator2n})$,  is in ${\rm Int} \ \cC_u(|\lambda_*|^{2 n})$, as defined in $(\ref{eq:conesdef})$ whenever $u(|\lambda_*|^{2 (n+1)}) \in \cC_u(|\lambda_*|^{2 (n+1)})$, $n=1,2, \ldots$. Since $\arccos(0.5967) - \arccos(0.71)>0.075$, the cone  $\cT(\cC_u)(|\lambda_*|^{2})$, according to $(\ref{eq:cone2})$, is contained in ${\rm Int} \ \cC_u(|\lambda_*|^2)$ as defined in $(\ref{eq:conel2})$. This demonstrates invariance of cones $\cC(|\lambda_*|^{2 n})$ for all $n=1,2, \ldots$.

\medskip

\ignore{
\noindent  {\it Step 3)}  We will now demonstrate that the operator $D \cR_*(\zeta_*)$ has an eigenvalue outside the unit circle which corresponds to an eigenvector which is not tangent to the manifold of almost commutative pairs, that is, not an eigenvector in $T_{\zeta_*} \cM(U,V)$.

The equation $(\ref{eq:uoperator0})$ together with part $1)$  of Lemma $\ref{lem:conebounds}$, and the fact that $\alpha(0)<\pi/4$,  implies that 
$$|\cT(u)(0)|  \ge  {1 \over |\lambda_*|^4} |u(0)|.$$
In particular, this implies that $\cT$ maps a vector field $u$ which are non-zero in a neighborhood of point $0$ into a field which is non-zero on some neighborhood of $0$, and which, furthrmore, there exists a neighborhood $\tilde{Z}$ of $0$ and $1 < \mu_u  \le {|\lambda_*|^{-4}}$, such that
\begin{equation}
\label{eq:expansion}  |\cT(u)(z)| \ge \mu_u |u(z)|
\end{equation}
for all $z \in \tilde{Z}$.

By a standard result (see, e.g. Proposition 6.2.12 and Corollary 6.2.13 in \cite{KH}), expansion of non-zero vectors in $\cC_u(z)$ for $z \in \tilde{Z}$   and   the   invariance condition $\cT(\cC_u) \subset {\rm Int} \ \cC_u$ on the set $\tilde{Z} \cap E$, where $E=\{|\lambda_*|^{2n}\}_{n=0}^\infty \cup 0$, implies that any  analytic direction field in $\cC_u(z)$ converges under the operator $\cT$ in $\tilde{Z} \cap E$ to a  analytic invariant direction field for $\cT$.  Since the set $\tilde{Z} \cap E$ contains a limit point $0$,  any  analytic direction field in $\cC_u(z)$ in $Z$ converges under the operator $\cT$ in all of $Z$.


The equation $(\ref{eq:expansion})$ implies the claim about the spectrum. Since $D \cR_*^2(\zeta_*)$ is a compact operator whose spectrum, outside of $0$, consists of eigenvalues, the largest eigenvalue is larger than $\mu_u$. 

Notice, that this eigenvector $(u_*,v_*)$ can  not lie in $T_{\zeta_*}\cM(U,V)$ since both $u_*(0)$ and $v_*(0)=c \circ \lambda^{1}_* u_*(0)$ are non-zero.

It is easy to show that the spectrum of the operator $\cR_*$ differs from that of $\cR$ by a single neutral eigenvalue $\delta=1$ (see Subsection \ref{subsec:coord_changes}). In particular, the expanding parts of the spectra coincide.

}

\medskip

\noindent  {\it Step 3)}  Assume now, that $u(z)=a z^2+b z^4 + \ldots \in \cC(z)$ is non-zero in some punctured neighborhoods $U \setminus \{ 0\}$ of $0$ and  $V \setminus \{1\}$ of $1$. Then, the equation $(\ref{eq:uoperator0})$ together with part $1)$  of Lemma $\ref{lem:conebounds}$, and the fact that $\alpha(|\lambda_*|^{2 n})<\pi/4$,  implies that 
\begin{equation}
\label{eq:cT2n} |\cT(u)(|\lambda_*|^{2 n})| > {1 \over |\lambda_*|^4} |u( |\lambda_*|^2 |\lambda_*|^{2n})| + \gamma {1 \over |\lambda_*|}  |u(\xi_*(\lambda_* |\lambda_*|^{2 n}))|,
\end{equation}
for some $\gamma>0$. Since $u(z)=O(z^2)$ in a neighborhood of $0$ and $u(1+z)=O(z)$ in a neighborhood of $1$, there exists $C>0$ such that $\left||u(z)|-{1 \over 2} |u''(0)||z|^2 \right|<C |z|^4$ and $\left| |u(\xi_*(z))|-2|u'(1)||\psi'(0)||z| \right|<C |z|^2$ for sufficiently small $z$. Therefore, according to $(\ref{eq:cT2n})$, for sufficiently large $n$,
\begin{eqnarray}
\nonumber |\cT(u)(|\lambda_*|^{2 n})| \!\! & \!\!> \!\!&  \!\! \left\{ \! {1 \over 2} |u''(0)||\lambda_*|^{4 n} \! \!-C |\lambda_*|^4 |\lambda_*|^{8 n} \! \right\}\!+\! \gamma \left\{ 2|u'(1)||\psi'(0)||\lambda_*|^{2 n} \!\!-C |\lambda_*||\lambda_*|^{4 n}   \right\}, \\
\nonumber  \!\!& \!\!= \!\!&  \!\!\left\{ {1 \over 2} |u''(0)||\lambda_*|^{4 n} +C |\lambda_*|^{8 n}  \right\}+ \gamma \left\{ 2|u'(1)||\psi'(0)||\lambda_*|^{2 n} -C |\lambda_*||\lambda_*|^{4 n} - \right. \\
\nonumber  \!\!& \!\!\phantom{=} \!\!&  \!\!\phantom{\left\{ {1 \over 2} |u''(0)||\lambda_*|^{4 n} +C |\lambda_*|^{8 n}  \right\}}- \left. C(1+|\lambda_*|^4)  |\lambda_*|^{8 n} \right\} \\
\nonumber  \!\!& \!\!> \!\!& \!\! |u(|\lambda_*|^{2 n})|\!+\! \gamma \tilde{C} |\lambda_*|^{2 n},
\end{eqnarray}
for some constant $\tilde{C}>0$. Therefore, $\cT$ strictly increases the uniform norm of $u$ on small neighborhoods of $0$.  The corresponding statement in the norm $||\cdot ||_1$ follows from \propref{norm-equiv}.

We remark, that $D \cR_*(\zeta_*)(u,v)= D \cR(\zeta_*)(u,v)$ whenever $v(0)=0$. Furthermore for such $(u,v)$, we have on a neighborhood of $0$
$$D \cRG(\zeta_*)(u,v)(z)=(\tilde{u},\tilde{v})(z)+(D a(\tilde{u},\tilde{v}) z^4 + D b(\tilde{u},\tilde{v}) z^6,0),$$
where $(\tilde{u},\tilde{v})=D \cR_*(\zeta_*)(u,v)$. Therefore, for the operator $\cT$ is modified as follows:
$$\cT(u)(z):=c \circ \lambda_*^{-1} \circ \left(u(\xi_*(\lambda_*(c(z))))  +  \eta'_*(\xi_*(\lambda_*(c(z)))) \cdot (c \circ \lambda_*^{-1} \circ  u(|\lambda_*|^2 z )) \right) +O(z^4),
$$
and the constructed cones $\cC_u(|\lambda_*|^{2 {n}}) \times \cC_v(|\lambda_*|^{2 n})$ and $\cC_u(0) \times \cC_v(0)$ remain invariant  under the action of the differential $D \cRG(\zeta_*)$ for large $n$ (small neighborhood of zero). An analytic direction field in a punctured neighborhood of zero converges on this neighborhood under the action of $\cT$, and, hence, converges everywhere on $Z$.
\flushright $\square$
\end{proof}

\subsection{Renormalization hyperbolicity}

To diagonalize the differential, we first find numerically eigenvalues and eigenvectors of the matrix $L_0$. It is elementary to check that the anti-linearity of the renormalization differential implies that the eigenvalues of $L_0$ come in pairs: if $\lambda$ is an eigenvalue, then so is $-\lambda$:
\begin{lemma} \label{lem:spectra}
Let $\sigma$  be the spectrum of the linear operator $D \cRG[(\phi_*,\psi_*)] \circ c$ on $T \cA_1(U,V)$. Then the spectrum of the linear operator $D \cRG[(\phi_*,\psi_*)]$ in $T \tilde{\cA}_1(U,V)$ is 
$$\tilde{\sigma}=-\sigma \cup \sigma.$$
\end{lemma} 
We verify numerically, that the matrix $L_0$ has $2(N_1+1)+2(N_2+1)$ eigenvalues, some of which have a multiplicity larger than $1$. We compute all real  eigenvalues, together with their eigenvectors, and the complex pairs, together with their eigenplanes, numerically using the standard linear algebra procedure in ADA.

At the next step, we choose the first $E \equiv 2(E_1+1)+2(E_2+1)$ eigenvectors $e_i$, $E=30$ and $E_2=20$, as the basis in the following linear subspace:
$${\rm span} \left\{ \{ (l_\phi^k,0) \}_{k=0}^{E_1},\{(0,l_\psi^k)\}_{k=0}^{E_2},\{(i l_\phi^k,0)\}_{k=0}^{E_1},\{(0,i l_\psi^k)\}_{k=0}^{E_2} \right\},$$ 
keeping 
$$\left\{ \{ (l_\phi^k,0) \}_{k=E_1+1}^{N_1},\{(0,l_\psi^k)\}_{k=E_2+1}^{N_2},\{(i l_\phi^k,0)\}_{k=E_1+1}^{N_1},\{(0,i l_\psi^k)\}_{k=E_1+1}^{N_2} \right\}$$ 
as the remaining basis vectors. The new basis will be referred to as $\cE$. 

To find the linear transformation from the basis $\cL$ to $\cE$, we first find the inverse $B^{-1}$ the matrix $B=\left[ \cE \right]$ numerically, and then find a bound on the inverse using a Newton method in the equation $B B^{-1}=\I$. 

Finally, we consider a projection of the bound $D$ on operator $D \cRG$ in  $B_r$ on the complimentary subspace to the span of the approximate eigenvectors $e_0$ and $e_1$ which numerically span the expanding subspace:
$$T=(I-S_0 -S_1) D (\I -S_0 -S1), \quad S_i=\left[
\begin{array}{c}
e_i^0  B^{-1}_i \cr
e_i^1  B^{-1}_i \cr
\ldots   \cr
e_i^N  B^{-1}_i \cr
\end{array}
\right], 
$$
where $e_i^k$ are the $k$-th component of the $i$-th eigenvector, and $B^{-1}_i$ stands for the $i$-th row of the matrix $B$. We compute rigorously that
$$\| T^3 \| <1,$$ 
which implies that the stable subspace of $D \cRG[(\phi_*,\psi_*)]$ in $T_{(\phi^*,\psi^*)} \tilde{A}_1(U,V)$ has real codimension two.  This together with Proposition \ref{prop:unstable_dir} implies Part (v) of the Main Theorem 1.

\ignore{
\subsection{Renormalization spectrum associated with the coordinate changes}\label{subsec:coord_changes}
 
In this section we will consider the part of the renormalization spectrum associated with the coordinate changes. Let $\sigma$ be a generator of a coordinate transformation, then
\begin{equation} \label{eq:change} (\id +t \sigma)^{-1} \circ \eta \circ (\id +t \eps \sigma)=\eta+t h_{\eta,\sigma}+O(t^2), \ {\rm where} \ h_{\eta,\sigma} = -\sigma \circ \eta + \eta' \cdot \sigma.
\end{equation}

Consider the operator  
$$\cR_* (\eta,\xi)=(\las^{-1} \circ \eta \circ \xi \circ \las, \las^{-1} \circ \eta \circ \las),$$
where $\las=-|\eta_*(0)|$. Then
\begin{eqnarray}
\nonumber D \cR_*(\eta_*,\xi_*) \left(h_{\eta_*,\sigma},h_{\xi_*,\sigma}\right)&=&( \las^{-1} \circ \left(-\sigma \circ \eta_* \circ \xi_*  +(\eta_* \circ \xi_*)' \cdot \sigma  \right) \circ \las, \\
\nonumber &\phantom{=}&\phantom{(} \las^{-1} \circ \left(-\sigma \circ \eta_*  +\eta_*' \cdot \sigma  \right) \circ \las ) \\
\nonumber &=&( - \las^{-1} \circ \sigma \circ \las \circ \eta_*  +\eta_*' \cdot (\las^{-1} \sigma  \circ \las), \\
\nonumber &\phantom{=}&\phantom{(} -\las^{-1} \circ \sigma \circ \las \circ \xi_*  +\xi_*' \cdot (\las^{-1} \circ \sigma  \circ \las) ) \\
\nonumber &=&  \left(h_{\eta_*,\las^{-1} \circ \sigma \circ \las},h_{\xi_*, \las^{-1} \circ \sigma \circ \las} \right),
\end{eqnarray}
and is equal to $\kappa \left(h_{\eta_*,\sigma},h_{\xi_*,\sigma}\right)$ whenever 
$$\las^{-1} \circ \sigma \circ \las= \kappa \sigma.$$
Consider monomial generators $\sigma_k(z)=z^k$, corresponding to the eigenvectors
$$e_k=\left(h_{\eta_*,\sigma_k},h_{\xi_*,\sigma_k}\right),$$ 
then
$$\las^{-1} \circ \sigma_k \circ \las= \las^{k-1} \sigma_k, $$
In particular, the constant generator of translations, $\sigma_0(z)=1$, corresponds to an expanding eigenvector of the operator $D \cR_*$, while the generator of rescalings $\sigma_1(z)=z$ corresponds to a neutral eigendirection of $D \cR_*$. At the same time $\span\{e_1\}$ is the null eigenspace for the operator $D \cR$ in $T_{(\eta_*,\xi_*)}\cE(U,V)$.

Perturbations in the  directions of the eigenvectors $e_k$, $k \ge 1$, preserve the class of a.c.s. pairs. Indeed, let $\tilde{\eta}=(\id +\eps \sigma_k)^{-1} \circ \eta \circ  (\id +\eps \sigma_k)$ and $\tilde{\xi}=(\id +\eps \sigma_k)^{-1} \circ \xi \circ  (\id +\eps \sigma_k)$, then $(\ref{eq:accond})$ implies
\begin{eqnarray}
\nonumber \tilde{\eta} \circ \tilde{\xi}(0)&=& -\sigma_k \circ \eta \circ \xi(0)  + (\eta \circ \xi) '(0) \cdot \sigma_k(0) +O(\eps^2)\\
\nonumber &=&-\sigma_k \circ \xi \circ \eta(0)  + (\xi \circ \eta) '(0) \cdot \sigma_k(0) +O(\eps^2)\\
\nonumber &=&\tilde{\xi} \circ \tilde{\eta}(0)+O(\eps^2) \\
\nonumber (\tilde{\eta} \circ \tilde{\xi})'(0)&=& -\sigma_k' \circ \eta \circ \xi(0) \cdot  (\eta \circ \xi)'(0)    + (\eta \circ \xi) ''(0) \cdot \sigma_k(0)+ (\eta \circ \xi) '(0) \cdot \sigma_k'(0)+\\
\nonumber &\phantom{=}& \phantom{-\sigma_k' \circ \eta \circ \xi(0) \cdot  (\eta \circ \xi)'(0)    + (\eta \circ \xi) ''(0) \cdot \sigma_k(0)} + O(\eps^2) \\
\nonumber &=&(\tilde{\xi} \circ \tilde{\eta})'(0)+O(\eps^2),
\end{eqnarray}
\begin{eqnarray}
\nonumber (\tilde{\eta} \circ \tilde{\xi})''(0)&=& -\sigma_k'' \circ \eta \circ \xi(0) \cdot  ((\eta \circ \xi)'(0))^2-\sigma_k' \circ \eta \circ \xi(0) \cdot  (\eta \circ \xi)''(0)   \\
\nonumber &\phantom{=}&+ (\eta \circ \xi) '''(0) \cdot \sigma_k(0)+2 (\eta \circ \xi) ''(0) \cdot \sigma_k'(0)+ (\eta \circ \xi) '(0) \cdot \sigma_k''(0)+O(\eps^2)\\
\nonumber &=&(\tilde{\xi} \circ \tilde{\eta})''(0)+O(\eps^2), \\
\nonumber (\tilde{\eta} \circ \tilde{\xi})'''(0)&=& -\sigma_k''' \circ \eta \circ \xi(0) \cdot  ((\eta \circ \xi)'(0))^3-3 \sigma_k'' \circ \eta \circ \xi(0) \cdot (\eta \circ \xi)'(0) \cdot  (\eta \circ \xi)''0) -\\\nonumber &\phantom{=}& -\sigma_k' \circ \eta \circ \xi(0) \cdot  (\eta \circ \xi)'''(0)  + (\eta \circ \xi)^{(4)}(0) \cdot \sigma_k(0) +\\
\nonumber &\phantom{=}& +3 (\eta \circ \xi) '''(0) \cdot \sigma_k'(0)+3 (\eta \circ \xi) ''(0) \cdot \sigma_k''(0)+\\
\nonumber &\phantom{=}& +(\eta \circ \xi) '(0) \cdot \sigma_k'''(0) +O(\eps^2) \\
\nonumber &=&(\tilde{\xi} \circ \tilde{\eta})'''(0)+\left\{ \left( (\eta \circ \xi)^{(4)}(0)-(\xi \circ \eta)^{(4)}(0) \right) \cdot \sigma_k(0) \right\}+O(\eps^2). 
\end{eqnarray}
In the last expression in parenthesis $\sigma_k(0)=0$ for all $k > 0$, but is non-zero for $k=0$. Therefore, the eigenvector $e_0$, corresponding to translations of coordinates, is the only eigendirection of $D \cR_*$ in $T_{(\eta_*,\xi_*)}\cE(U,V)$  associated with coordinate changes which does not lie in the tangent space $T_{(\eta_*,\xi_*)}\cM(U,V)$ at $(\eta_*,\xi_*)$ to the manifold of a.c.s. pairs.

}